\documentclass[11pt,a4paper,reqno]{article}

\usepackage[a4paper,width=140mm,top=25mm,bottom=30mm]{geometry}
\usepackage{hyperref, enumitem}
\usepackage[utf8]{inputenc}
\usepackage[T1]{fontenc}
\usepackage[english]{babel}
\usepackage{xcolor}
\usepackage{float}
\usepackage{stmaryrd}
\usepackage{mathrsfs} 
\usepackage{amsthm}

\theoremstyle{plain}
\newtheorem{theorem}{Theorem}[section]
\newtheorem{corollary}[theorem]{Corollary}
\newtheorem{proposition}[theorem]{Proposition}

\newtheorem{lemma}[theorem]{Lemma}
\newtheorem{claim}[theorem]{Claim}

\theoremstyle{remark} 
\newtheorem{remark}[theorem]{Remark}
\newtheorem*{notation}{Notation}

\theoremstyle{definition}

\newcommand{\defini}{\textit}

\usepackage{amsmath}
\usepackage{amssymb}
\usepackage{dsfont} 
\usepackage{mathtools}
\mathtoolsset{showonlyrefs}

\usepackage{MnSymbol} 
\usepackage{calc} 

\newlength{\arrow}
\usepackage{xargs}
\newcommand*{\lr}[2][]{\xleftrightarrow[\ \text{\raisebox{3pt}{$#1$}}\ ]{#2}}
\newcommand{\nlr}[2][]{\mathrel{\ooalign{$\xleftrightarrow[\ \text{\raisebox{3pt}{$#1$}}\ ]{\mathmakebox[\maxof{\arrow}{\widthof{\scriptsize $#2$}}]{#2}}$\cr\hidewidth\raisebox{.13em}{\rotatebox{-20}{\scalebox{0.6}{$/$}}}\hidewidth}}}

\usepackage{tabularx}
\usepackage{graphicx}
\usepackage{caption}
\usepackage{subcaption}
\begin{document}
	\title{Supercritical percolation on graphs of polynomial growth}
	\date{\today}
	
	\author{Daniel Contreras\thanks{ETH Zürich, Rämistrasse 101, 8092 Zurich Switzerland}
		\and
		Sébastien Martineau\thanks{LPSM, Sorbonne Université, 4 place Jussieu, 75005 Paris France \newline
			daniel.contreras@math.ethz.ch\hfill
			sebastien.martineau@lpsm.paris\hfill
			vincent.tassion@math.ethz.ch
		}
		\and
		Vincent Tassion\footnotemark[1]} 
	
	
	
	
	
	
	
	

	\maketitle
	
	\begin{quote}\it 
		\hfill  À la mémoire de Claude Danthony (1961--2021)
	\end{quote}
	
	\begin{abstract}
		We consider Bernoulli percolation on transitive graphs of polynomial growth.  In the subcritical regime ($p<p_c$), it is well known that the connection probabilities decay exponentially fast. In the present paper, we study the supercritical phase ($p>p_c$) and prove the exponential decay of the truncated connection probabilities (probabilities that two points are connected by an open path, but not to infinity). This sharpness result was established by~\cite{chayes1987bernoulli} on $\mathbb Z^d$ and uses the difficult slab result of Grimmett and Marstrand. However, the techniques used there are very specific to the hypercubic lattices and do not extend to more general geometries. In this paper, we develop new robust techniques based on the recent progress in the theory of sharp thresholds and the sprinkling method of Benjamini and Tassion. On~$\mathbb Z^d$, these methods can be used to produce a new proof of the slab result of Grimmett and Marstrand.
	\end{abstract}
	
	\maketitle
	\vspace{1cm}
	\begin{figure}[h!]
		\centering
		\begin{subfigure}[b]{0.3\textwidth}
			\includegraphics[width=\textwidth]{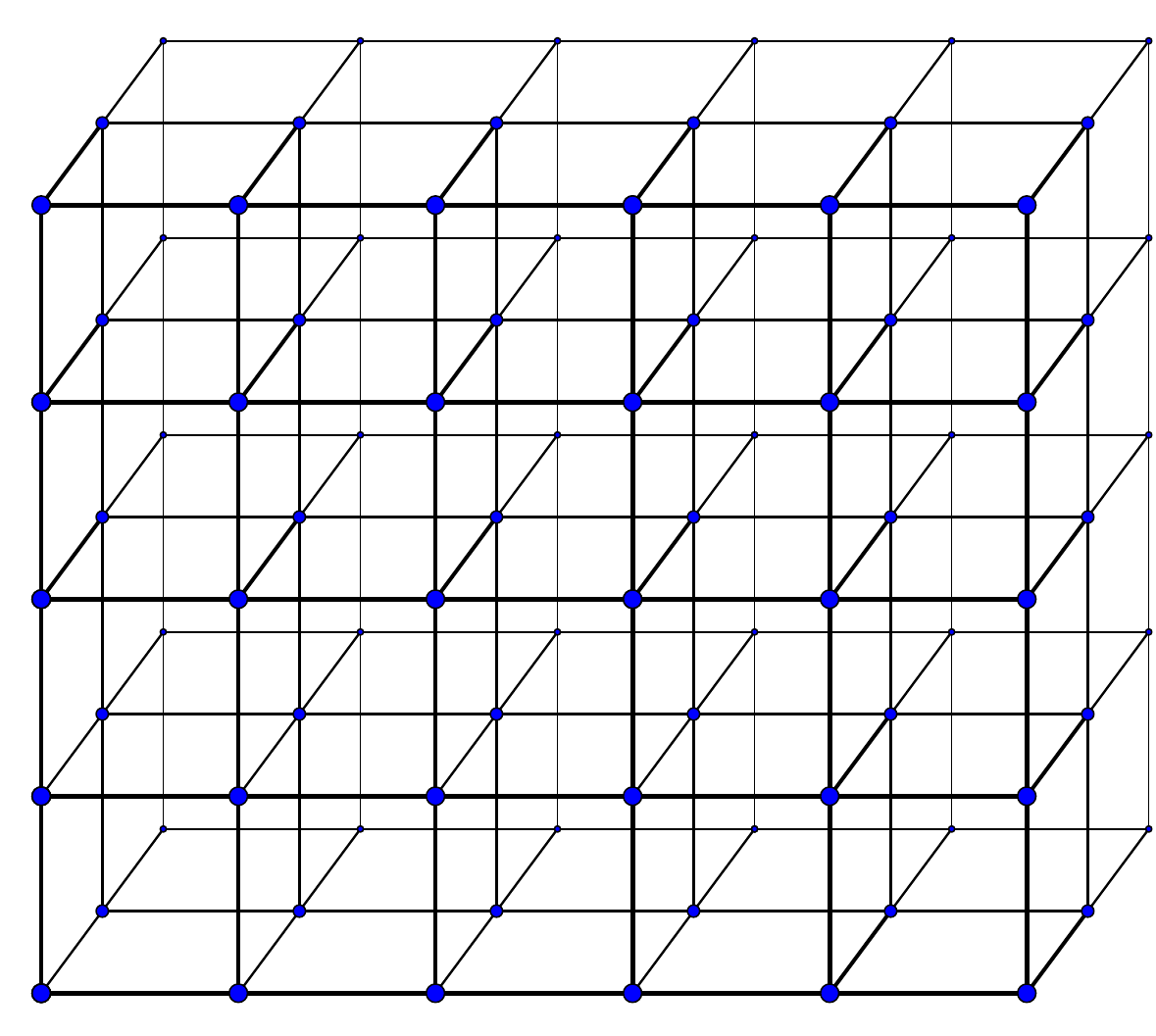}
		\end{subfigure}
		~ 
		\begin{subfigure}[b]{0.3\textwidth}
			\includegraphics[width=\textwidth]{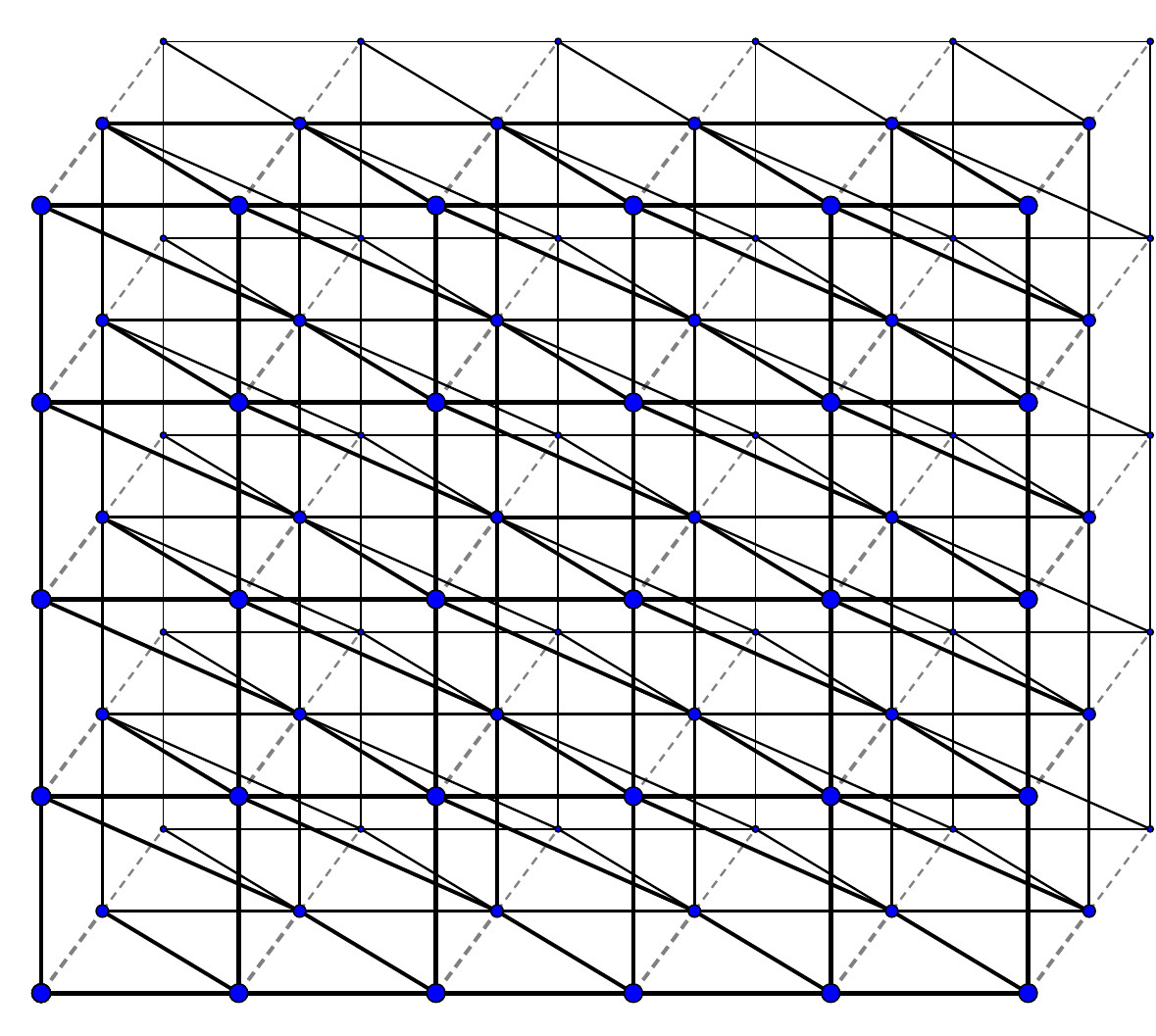}
		\end{subfigure}
		~ 
		\begin{subfigure}[b]{0.3\textwidth}
			\includegraphics[width=\textwidth]{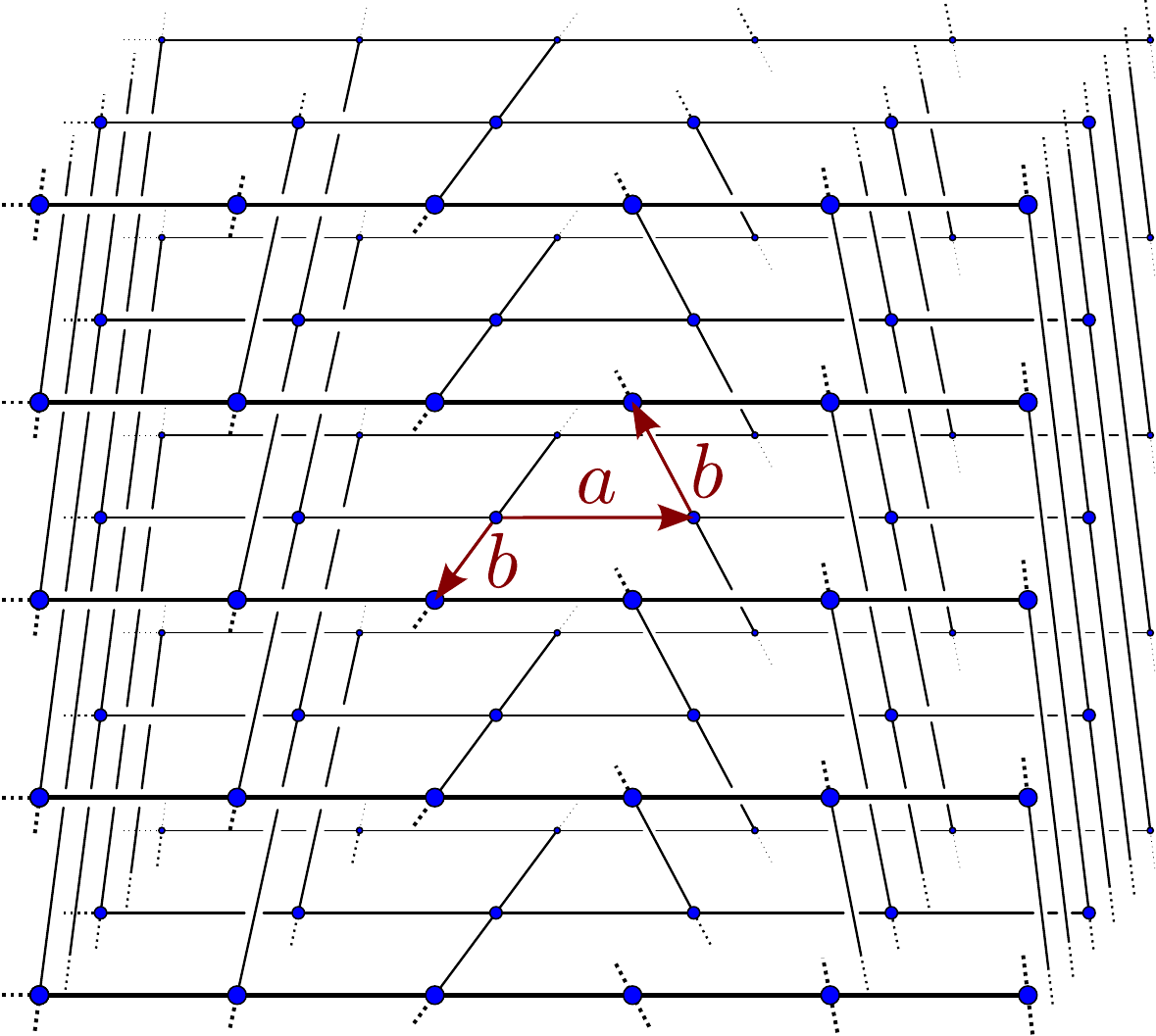}
		\end{subfigure}
		\caption{From left to right: the lattice $\mathbb{Z}^3$, some Cayley graph of $\mathbb{Z}^3$, and a Cayley graph of the Heisenberg group.}\label{fig:111}
	\end{figure}
	
	\vspace{.4cm}
	
	\pagebreak

	\section{Introduction}
	\label{sec:introduction}

	\subsection{General context}
	
	After its introduction in the sixties by Broadbent and Hammersley \cite{MR91567}, percolation was mainly studied on the hypercubic lattice $\mathbb Z^d$, where most of the main features of the model have been rigorously described. In 1996, with their paper ``Percolation beyond $\mathbb Z^d$, many questions and few answers'' \cite{MR1423907}, Benjamini and Schramm initiated the systematic study of percolation on general transitive graphs, leading to a new and fascinating research area. In this generality, new questions  emerged, new techniques were used to establish deep relations between the geometric properties of a graph and the behaviour of percolation processes on this graph. Interesting in their own right, these percolation results also shed new light on the theory on $\mathbb Z^d$. The present paper is exactly in this spirit:  motivated by questions emerging in the general study of percolation on transitive graphs (such as Schramm's Locality Conjecture \cite{benjamini2011critical}), we prove a supercritical sharpness result on transitive graphs with polynomial growth. An interested reader will also find below a new proof of the Grimmett--Marstrand Theorem, which is a central result in the study of supercritical percolation on $\mathbb Z^d$. The proof is robust, and we expect it to have applications to the study of more general percolation processes on $\mathbb Z^d$, such as FK-percolation or level sets of Gaussian processes.

	\paragraph{Geometric framework: transitive graphs of polynomial growth} Let  $G=(V,E)$ be a  vertex-transitive graph  with a fixed origin $o\in V$ (for every $x,y\in V$, there exists a graph automorphism mapping $x$ to $y$). Throughout the paper, all the graphs are assumed to be locally finite and  connected and we will always make these hypotheses without further mention. Write $B_n$ for  the ball of radius $n$ centred at $o$. We say that $G$ has polynomial growth if there exists a polynomial $P$ such that $|B_n|\le P(n)$ for every $n\ge 1$. A celebrated theorem of Gromov and Trofimov \cite{gromov81, trofimov85polynomial}  states that such a graph is always quasi-isometric to a Cayley graph of a finitely generated nilpotent group (see also Theorem VII.56 of \cite{delaharpe}, which is there attributed to ``Diximier, Wolf, Guivarc'h, Bass, and others''). This deep structure result has a long and ongoing history: see \cite{bass72, guivarch73,  kleiner2010, shalomtao2010, tesseratointon}. An important consequence is that there exists an integer $d$, called the growth exponent of $G$, such that    
	\begin{equation}
		\label{eq:4}
		\exists c>0\quad\forall n\ge 1 \qquad  c n^d \le |B_n| \le \tfrac1 c n^d.
	\end{equation}
	Important examples of graphs of polynomial growth include the hypercubic lattice $\mathbb Z^d$, more general  Cayley graphs of $\mathbb Z^d$, and the Heisenberg group\footnote{The discrete Heisenberg group
		$
		\left\{\left(\begin{smallmatrix}
			1 & x & z\\ 0 & 1 & y \\ 0 & 0 & 1
		\end{smallmatrix}\right) : x,y,z \in \mathbb{Z}\right\}
		$
		is generated by  
		$a=\left(\begin{smallmatrix}
			1 & 1 & 0\\ 0 & 1 & 0 \\ 0 & 0 & 1
		\end{smallmatrix}\right)$ and
		$b=\left(\begin{smallmatrix}
			1 & 0 & 0\\ 0 & 1 & 1 \\ 0 & 0 & 1		\end{smallmatrix}\right).$ }: see Figure~\ref{fig:111}.
	
	\paragraph{Percolation on transitive graphs} Let $G=(V,E)$ be a transitive graph. Let $\mathrm P_p$ be the Bernoulli bond percolation measure $\mathrm P_p$ on $\{0,1\}^E$, under which $\omega=(\omega(e))_{e\in E}$ is a family of i.i.d.~Bernoulli random variables of parameter $p$ (we refer to \cite{MR1707339} and \cite{MR3616205} for general introductions to percolation).  We identify $\omega$ with the  subgraph of $G$ obtained by keeping the edges such that $\omega(e)=1$ (called open edges)  and deleting the edges such that $\omega(e)=0$ (the closed edges). The connected components of $\omega$ are called clusters. Percolation on $G$ undergoes a phase transition at a critical parameter: there is a parameter $p_c\in (0,1]$ such that for all $p<p_c$, there is $\mathrm P_p$-almost surely no infinite cluster and for all $p>p_c$, there is $\mathrm P_p$-almost surely at least one infinite cluster.
	It was recently proved in \cite{pcwithgff} that this phase transition is non-trivial if and only if the graph $G$ has superlinear growth: for such graphs, we  have $0<p_c(G)<1$.

	\paragraph{Subcritical sharpness} It was proved in \cite{aizenman1987sharpness} and \cite{menshikov1986coincidence} (see also \cite{duminil2016new})  that the phase transition is sharp, in the following sense:   For every $p<p_c$, there exists a constant $c=c(p)>0$ such that
	\begin{equation}
		\label{eq:54}
		\forall n\ge1\quad \mathrm P_p[o\lr{} \partial B_n]\le e^{-cn},
	\end{equation}
	where $o\lr{} \partial B_n$  is the event that there exists an open path from a fixed origin $o$ to distance $n$ around it. The original papers \cite{aizenman1987sharpness} and \cite{menshikov1986coincidence} prove this result on the hypercubic lattice $\mathbb Z^d$, but it also holds for general transitive graphs \cite{antunovic2008sharpness}. This result is central in the theory: it leads to the important notion  of correlation length, and is the starting point of several finer analyses  of the subcritical regime (see e.g. \cite{MR1707339} and references therein).
	
	\paragraph{Supercritical sharpness} There should be a supercritical counterpart of sharpness. In the supercritical regime $p>p_c$, the natural quantities to consider are the truncated connection probabilities, which encode the  connectivity properties of the random graph obtained from $\omega$ by removing the infinite cluster(s). It is expected (see e.g \cite[Conjecture 5.3]{hermon2021supercritical}) that for every $p>p_c$, there is a constant $c=c(p)>0$ such that 
	\begin{equation}
		\label{eq:64}
		\forall n\ge 1\quad \mathrm P_p[o\lr{} \partial B_n, o\nlr{}\infty]\le e^{-cn}, 
	\end{equation}
	where $o\nlr{}\infty$ denotes the  event that the origin does not belong to an infinite cluster. Currently, the decay above is known for the hypercubic lattice \cite{chayes1987bernoulli,grimmett1990supercritical} and non-amenable graphs~\cite{hermon2021supercritical}.
	
	In general, the study in the supercritical regime is  much more delicate than the subcritical regime. A first reason is geometrical: a key idea to study subcritical clusters is a domination by  subcritical branching processes, whose asymptotic features do not rely on the precise geometry of the underlying $G$. This explains the robustness of the subcritical argument. In contrast, the study of the supercritical clusters requires to understand the infinite cluster(s), whose geometry may be very related to the underlying graph, and so far there is no robust approach to the supercritical regime. A second reason is more technical. The truncated connection events (i.e. of the form $\{A\lr{}B,\ A\nlr{}\infty\}$ where $A\nlr{}\infty$ is the event that $A$ is not connected to infinity in $\omega$) are neither increasing nor decreasing events, which reduces considerably the size of the available ``toolbox'' for the study of such events.  
	
	In this paper, we obtain a clear description of the supercritical phase for transitive graphs of polynomial growth, as presented in the next section. In particular, we prove that supercritical sharpness holds for these graphs.

	\subsection{Main results}\label{sec:main-results}

	Let $G$ be a transitive graph of polynomial growth, with growth exponent $d\ge 2$. Equivalently, $G$ is taken to be a transitive graph of polynomial growth satisfying $p_c(G)<1$  --- see \cite[Corollary 7.19]{MR3616205}. In this paper, we develop new methods that give a precise description of the supercritical phase of Bernoulli percolation on such graphs $(p>p_c)$. More precisely, we build finite-size events that ensure the local existence and uniqueness of large clusters: see Proposition~\ref{prop:1}. This enables us to use powerful renormalisation methods, which extend several perturbative arguments (valid only for $p$ close to 1) to the whole regime $p>p_c(G)$. Our main two results are about the geometry of the finite clusters in the supercritical regime, but our methods would also imply several other results regarding the geometry of the infinite cluster\footnote{A graph of polynomial growth is necessarily  amenable: otherwise, the growth would be exponential. The Burton--Keane Theorem \cite{MR990777} ensures that, for amenable transitive graphs, for any $p$, there is at most one infinite cluster almost surely. In particular, for $p>p_c$, there is a unique infinite cluster.}. Our first result states that the radius of a finite cluster has an exponential tail in the supercritical regime.
	
	\begin{theorem}\label{thm:1}
		Let $G$ be a transitive graph of polynomial growth with $d\ge2$ and $p>p_c(G)$. There exists a constant $c>0$ such that
		\begin{equation}
			\forall n\ge 1\quad \mathrm P_p[o\lr{} \partial B_n,\, o \nlr{} \infty]\le e^{-cn}.
		\end{equation}
	\end{theorem}
	The second result is the stretched-exponential decay in volume of the finite clusters.
	\begin{theorem}\label{thm:2}
		Let $G$ be a transitive graph of polynomial growth with $d\ge2$ and $p>p_c(G)$. Let  $C_o$ denote the cluster of $o$. There exists a constant $c>0$ such that
		\begin{equation}
			\forall n\ge 1\quad \mathrm P_p[n\leq|C_o|<\infty]\le \exp\left(-cn^{\frac{d-1}{d}}\right).
		\end{equation}
	\end{theorem}

	\subsection{Comments}
	\label{sec:comments}

	\paragraph{Previous work on hypercubic lattices}  Both Theorems~\ref{thm:1} and \ref{thm:2} were known for the hypercubic lattice $\mathbb Z^d$.  In dimension $d=2$, they were proved by Kesten~\cite{MR575895}. In dimension $d\ge 3$, they were proved by~\cite{chayes1987bernoulli} and~\cite{kesten1990probability}, by making use of the difficult slab result of Grimmett and Marstrand \cite{grimmett1990supercritical}: for any $p>p_c(\mathbb{Z}^d)$, the percolation on $\mathbb{Z}^d$ restricted to a sufficiently thick slab (i.e. a graph of the form $\mathbb Z^2\times\{0,\dots,k\}^{d-2}$ for $k$ large)  contains an infinite cluster. These previous methods do not extend  to general graphs of polynomial growth for two main reasons. First, the proof of Grimmett and Marstrand relies strongly on the symmetries of $\mathbb Z^d$ (in particular reflections and rotations). Such symmetries are not available for general graphs: one can think of Cayley graphs of $\mathbb Z^d$ with respect to non-symmetric generating sets (see Figure~\ref{fig:111}). As for $g\mapsto g^{-1}$, it is generally not a graph automorphism for Cayley graphs of non-Abelian groups. A second obstacle is the lack of ``slab''-structure for general graphs of polynomial growth: the hypercubic $\mathbb Z^d$ can naturally be partitioned into slabs. In contrast, the Cayley graph of the discrete Heisenberg group illustrated on Figure~\ref{fig:111} has no natural notion of slab\footnote{If we try to mimic the definition of slabs in $\mathbb Z^d$, we want to take a thickened version of the subgroup generated by $a$ and $b$. The problem is that, in the case of the Heisenberg group, this subgroup is equal to the whole group.}. In our present proof, we replace the slab result by estimates on the two-point function inside ``corridor''  subgraphs introduced in Section~\ref{sec:definition-notation}. Contrary to the lattice case, our proof does not distinguish between $d=2$ and $d\ge3$. In this sense, this unifies the two approaches for hypercubic lattices.
	
	\paragraph{Dynamic versus static renormalisation} In the case of the hypercubic lattice, the Grimmett--Marstrand Theorem is proved by using a dynamic renormalisation argument (in the sense of \cite[Chapter 7]{MR1707339}). In the present paper, we construct a local existence-and-uniqueness event, which allows us to directly perform a static renormalisation.    	
	
	\paragraph{A new quantitative proof of Grimmett--Marstrand} The interested reader may extract from the present paper a new proof of the Grimmett--Marstrand Theorem on $\mathbb{Z}^d$. When applying the methods of the present paper to hypercubic lattices, some simplifications occur at several places, using symmetries. For example, the whole Section~\ref{sec:bounds-uniq-zone} may be replaced by a reference to the stronger  result of Cerf \cite{cerf2015}, and Section~\ref{sec:properties-seeds} may be drastically simplified. The proof is quantitative and would give an estimate of the same order as the one obtained in \cite{duminil2019upper}.
	
	\paragraph{Lower bounds} The bounds in Theorems~\ref{thm:1} and \ref{thm:2} are sharp at the exponential scale: This easier result is classical for the hypercubic lattice $\mathbb{Z}^d$ and the same techniques can be used to prove that, for every $p\in(p_c,1)$,  there exists a constant $c'>0$ such that
	\begin{equation}
		\forall n\ge 1\quad \mathrm P_p[o\lr{} \partial B_n,\, o \nlr{} \infty]\ge e^{-c'n}.
	\end{equation}
	and
	\begin{equation}
		\forall n\ge 1\quad \mathrm P_p[n\leq|C_o|<\infty]\ge \exp\left(-c'n^{\frac{d-1}{d}}\right).
	\end{equation}
	Remark~\ref{rem:coarse-conn-spheres}, Lemma~\ref{lem:5} and Lemma~\ref{lem:10} guarantee that the geometry of $G$ is --- as far as this argument is concerned --- as nice as that of $\mathbb{Z}^d$.	
	
	\paragraph{Locality} In the context of graphs of polynomial growth (which are amenable), supercritical sharpness is related to the existence of some local existence and uniqueness event in a finite box: see Proposition~\ref{prop:1}. This allows us to obtain a finite volume characterisation of $p>p_c(G)$. In \cite{polynomialocality}, we use the present work to prove  Schramm's Locality Conjecture (stated in \cite{benjamini2011critical}) in the particular case of transitive graphs of polynomial growth. This extends \cite{MR3630298}, with different techniques. The case of transitive graphs with (uniformly lower-bounded) exponential growth has been established in \cite{2018arXiv180808940H}.		
	
	\paragraph{Related works} Several related works regarding the sharpness of the supercritical phase have been developed in the last few years. For the hypercubic  lattice $\mathbb{Z}^d$, a quantitative version of the Grimmett--Marstrand Theorem was presented in \cite{duminil2019upper}. In the case of non-amenable graphs, \cite{hermon2021supercritical} proved exponential decay for the finite cluster size distribution. 	For the Ising Model, \cite{duminil2020exponential} obtained exponential decay for the truncated two-point function in the supercritical regime. As for Gaussian fields, we can mention the work \cite{duminil2020equality} on the Gaussian free field in $\mathbb{Z}^d$ for $d\ge 3$ and \cite{severo2021sharp} for general results on continuous Gaussian fields in $\mathbb{R}^d$  ($d\ge2$) with correlations decaying reasonably fast.

	\subsection{Definitions and notation}
	\label{sec:definition-notation}
	\newcommand{\piv}[2]{\mathsf{Piv}(#1,#2)}
	\newcommand{\pivx}[3]{\mathsf{Piv}_{#1}(#2,#3)}
	
	Throughout the paper, $G=(V,E)$ denotes a fixed transitive graph of polynomial growth, whose growth exponent $d$ satisfies $d\ge2$. The graph $G$ is taken to be simple: there are no multiple edges, no self-loops, no orientations on the edges. We also fix some origin $o\in V$.
	
	\paragraph{Graph notation} For $x\in V$, we write $B_n(x)$ for the ball of radius $n$ centred at $x$, and we simply write $B_n=B_n(o)$ for the ball centred at the origin.       The \defini{boundary  $\partial A$} of a set $A\subset V$ is  defined as the set of edges having one endpoint in $A$ and the other in $V\setminus A$. A \defini{path of length $\ell$ in $A$} is a sequence $\gamma=(\gamma_0,\ldots,\gamma_\ell$) of vertices of $A$ such that  $\gamma_i$ and $\gamma_{i+1}$ are neighbours for every $i$.
	
	\paragraph{Percolation definitions} Let $\omega$ be a percolation configuration. A path $\gamma$ is said to be \defini{open} if all its edges $\{\gamma_i,\gamma_{i+1}\}$  are open.  For $A\subset V$, we call \defini{clusters in $A$} the connected components of the graph with vertex set $A$ and edge set the elements of $\omega$ with both endpoints in $A$. 
	For $A,B,C\subset V$,  we say that $A$ and $B$ are \defini{connected in} $C$ if there exists an open path in $C$ from $A$ to $B$. We denote the corresponding event by $A\lr{C} B$ and its complement by $A\nlr{C} B$.  In the case $C=V$, we simply write $A\lr{} B$ and  $A\nlr{} B$.  
	
	\paragraph{Corridor function} Let $m,n\geq 0$ and $p\in[0,1]$. We define the \defini{corridor function of length $m$ and thickness $n$} at parameter $p$ by 
	\begin{equation}
		\label{eq:1}
		\kappa_p(m,n)=\min_{\gamma\: :\: \mathsf{length}(\gamma)\le m}\mathrm P_p[o(\gamma) \lr{B_n(\gamma)}e(\gamma)], 
	\end{equation}
	where $o(\gamma)$ and $e(\gamma)$ denote the first and last vertices of $\gamma$ and $B_n(\gamma)=\bigcup_iB_n(\gamma_i)$. The quantity has two different interpretations, depending on whether  $m\le n$ (``short corridor'') or $m\ge n$ (``long corridor''): In the first case, the set  $B_n(\gamma)$ looks roughly like a ball, and the quantity is similar to the two point-function restricted to a ball. In contrast, when $m\gg n$, the set $B_n(\gamma)$ takes the shape of a long corridor and the parameter $m$ becomes relevant for its geometry.  This quantity will be instrumental in our paper.  First, the corridor function has an important renormalisation property, presented in Section~\ref{sec:renorm-corr-funct}. Second, it is deeply related to the local  uniqueness events,  defined below. Finally, it will be the central object in the proof of the main proposition of this paper, in Section~\ref{central-proposition}: we strengthen some a priori bounds on the short-corridor function into strong estimates in long corridors.

	\paragraph{Pivotal and uniqueness events} For $1\leq m\leq n$, we define the \defini{pivotal event}
	\begin{equation}
		\label{eq:2}
		\piv m n= \bigcup_{x,y\in B_m}\{x\lr{}\partial B_n,\, y\lr{}\partial B_n,\, x\nlr{B_n}y\}.
	\end{equation}    
	In other words, the event $\piv m n$ occurs if there are two disjoint clusters (for the configuration restricted to $B_n$) connecting $B_m$ to $\partial B_n$. Even if the event is not formally defined in terms of the pivotality of a set  for a certain event, we use the notation $\mathsf{Piv}$ because the event will typically occur when the ball $B_m$  is pivotal for large connection events. This is also known in the literature as the two-arm event (see \cite{cerf2015}). Similarly, for $x\in V$, we define $\pivx x m n$ as the event $\piv m n$, but centred at $x$ instead of $o$. For $1\leq m\leq n$, we define the \defini{uniqueness event} by
	\begin{equation}
		\label{eq:38378297}
		U(m,n)=\{\text{there exists at most one cluster in $B_n$ intersecting  $B_m$ and $\partial B_n$}\}.
	\end{equation}
	Notice that  $U(m,n):=\piv m n ^c$. We emphasise that the event $U(m,n)$ does not require the existence of a cluster crossing from $B_m$ to $\partial B_n$ in $B_n$. When the event $U(m,n)$ occurs, there is either one or no such crossing cluster. The event $U(m,n)$ is particularly useful, as it allows us to ``glue'' locally macroscopic clusters.

	\paragraph{Monotone coupling} Let  $\mathbf{P}=\mathcal U([0,1])^{\otimes E}$ be the product of Lebesgue measures on $[0,1]^E$ and consider the canonical maps $\omega_p:[0,1]^E\to\{0,1\}^E$, $(x_e)\mapsto (\mathbf{1}_{x_e\leq p})$. Notice that, under $\mathbf P$,  the random configuration $\omega_p$ has law  $\mathrm P_p$.
	The random configuration $\omega_p$ naturally gives rise to the notions of $p$-open edge, $p$-cluster, $p$-open path and $p$-connectivity.
	For $A,B,C\subset V$ we write $A\lr[p]{C}B$ if  $A$ and $B$ are $p$-connected in $C$, and $A\nlr[p]{C}B$ for the complement event. When we are looking at all the $p$-configurations coupled together, additional interesting events appear. For instance, we can look at how $p$-clusters are connected at a parameter $q\ge p$, and the following generalisation of the uniqueness event will be useful. For $p\le q$ and $n\ge m\ge1$, we define the \defini{sprinkled uniqueness event} by
	\begin{equation}
		\label{eq:3}
		U_{p,q}(m,n)=\left \{\text{
			\begin{minipage}[c]{.5\linewidth} \centering
				All the $p$-clusters in $B_n$ intersecting  $B_m$ and $\partial B_n$ are $q$-connected in $B_n$
		\end{minipage}}
		\right\}.
	\end{equation}
	The event above has some useful monotonicity properties: For fixed $n\ge m\ge 1$, the function $f(p,q)=\mathbf P[U_{p,q}(m,n)]$ is nonincreasing in $p$ and nondecreasing in $q$. In contrast, the probability of the uniqueness event $\mathrm P_p[U(m,n)]$, which is equal to $f(p,p)$, has no clear monotonicity in $p$.

	\subsection{Organisation of the paper}
	
	The main ingredient in the proof of Theorems~\ref{thm:1} and \ref{thm:2} is the following proposition, which provides local existence and uniqueness of certain crossing clusters in large boxes. From there,  we apply a coarse-graining argument, presented in Section \ref{sec:proof-theor-1}. In  $\mathbb Z^d$,  standard coarse-graining arguments use the scaling property of $\mathbb Z^d$: the set $n\mathbb Z^d$ is a rescaled version of $\mathbb Z^d$. For general transitive graphs of polynomial growth, we circumvent scaling by using an $n$-independent percolation with sufficiently high marginals.   
	\begin{proposition}
		\label{prop:1}
		Let $G$ be a transitive graph of polynomial growth with $d\ge 2$ and let $p>p_c$. Then, for all $n$ large enough, we have
		\begin{equation}\label{eq:5}
			\mathrm P_p[B_{n/10}\lr{}\partial B_{n}, U(n/5,n/2)]\geq 1-e^{-\sqrt{n}}.
		\end{equation}
	\end{proposition}
	
	The main goal of the paper is the proof of the proposition above. The proof itself is presented in Section~\ref{central-proposition}. It  relies on several independent arguments and intermediate results established in Sections~\ref{sec:geom-lemmas}--\ref{sec:uniq-via-sprinkl}.
	Section~\ref{sec:geom-lemmas} is a geometric toolbox. In each of the Sections~\ref{sec:renorm-corr-funct}--\ref{sec:uniq-via-sprinkl}, we isolate one important ingredient (for the most important one, see Section~\ref{sec:sharp-thresh-hamming}). Each section may use the results from previous ones, but can be read roughly independently.
	Once the central Proposition~\ref{prop:1} is established, the two theorems can be proved by using some adaptation of standard renormalisation arguments, which are presented in Section~\ref{sec:proof-theor-1}. Here is a more detailed roadmap of the forthcoming sections:
	
	\paragraph{\it Section~\ref{sec:geom-lemmas}: Geometric lemmas} This section provides several definitions and lemmas on the geometry of graphs of polynomial growth, in particular related to cutsets, annuli, spheres, and finding many infinite paths that exit $B_n$ and stay away from each other. The reader may want to read this section only when other sections require it.
	
	\paragraph{\it Section~\ref{sec:renorm-corr-funct}: Renormalisation of the corridor function} In this section, we reduce the proof of Proposition~\ref{prop:1} to showing that, at infinitely many scales, long corridors can be crossed with probability larger than some constant. This involves a renormalisation property of the corridor function.
	
	\paragraph{\it Section~\ref{sec:quantitative-uniqueness}: Probability that two clusters meet at one point} We prove that  the probability  $\mathrm P_p[\piv 1 n]$ that two clusters of radius $n$  meet at one point decays polynomially fast in $n$, uniformly in $p$. We use an adapted version of the beautiful exploration  argument of Aizenman, Kesten, and Newman \cite{MR901151} (see also \cite{gandolfi1988uniqueness,cerf2015,2018arXiv180808940H}), which relates such meeting points to the deviation of the sum of i.i.d.\@ Bernoulli($p$) random variables. The argument extends to amenable transitive graphs.     
	
	\paragraph{\it Section~\ref{sec:sharp-thresh-seed}: Sharp threshold results via hypercontractivity} We establish a general sharp threshold result for connectivity probabilities of the form $\mathrm P_p[A\lr{C} B]$, $A,B,C\subset V$. To do so, we rely on the general inequalities on influences for Boolean functions of Talagrand \cite{talagrand1994russo} and \cite{KKL}.  In the spirit of recent applications of sharp threshold results to percolation \cite{MR3765895,duminil2019upper}, our argument does not involve approximations by symmetric events on the torus. Instead, we only use the  bounds on the influences in the ``bulk'' (at a sufficient distance from the set $A$ and $B$) provided by Section~\ref{sec:quantitative-uniqueness}. 	
	
	\paragraph{\it Section~\ref{sec:bounds-uniq-zone}: A priori bound  on the uniqueness zone}  We consider the following problem: for $n$ large,  for which value of $s(n)$ can we ensure that at most one  cluster in $B_n$ crosses from $B_{s(n)}$ to $\partial B_n$? Ultimately, Proposition~\ref{prop:1}  shows that we can choose  $s(n)\ge n/3$. In order to prove this, we need an  a priori lower  bound as an intermediate step: in Section~\ref{sec:bounds-uniq-zone}, we prove that  for $p>p_c$, we can choose $s(n)$ larger than any power of $\log n$.   To achieve this, we have to overcome difficulties that are not present for the lattice  $\mathbb Z^d$. In the  lattice case, a result of Cerf \cite{cerf2015} shows directly that we can choose $s(n)\ge n^c$ for a positive constant $c>0$ independent of $p\in[0,1]$. Due to the lack of symmetry, this argument does not extend to general graphs of polynomial growth. Here we combine some arguments of \cite{cerf2015} together with a new renormalisation method.

	\sloppy\paragraph{\it Section~\ref{sec:sharp-thresh-hamming}: Sharp threshold results  via Hamming distance} This section presents the \emph{main new argument}, which relies on a general differential inequality involving the Hamming distance on the hypercube (if you hesitate about what to read after this introduction, then the beginning of Section~\ref{sec:sharp-thresh-hamming} is a good choice). We prove that, in the supercritical regime, if the corridor function is small, then large annuli are crossed with high probability.
	This will be helpful in Section~\ref{central-proposition}: in order to establish Proposition~\ref{prop:1}, we will need to prove that the corridor function is large. The results of Section~\ref{sec:sharp-thresh-hamming} then say that, without loss of generality, we may assume that annuli are crossed with high probability.
	
	\paragraph{\it Section~\ref{sec:uniq-via-sprinkl}: Uniqueness via sprinkling} We consider the configuration in $B_n$ and assume that the large clusters fill sufficiently well the ball $B_n$, in the sense that  all the balls $B_k(x)$, $x\in B_n$, are $p$-connected to $\partial B_n$, for some $k\ll n$. In this case, we prove that for a small sprinkling $\delta>0$, the sprinkled uniqueness event $U_{p,p+\delta}(n/4,n)$ occurs with high probability: all the $p$-clusters crossing the annulus $B_n\setminus B_{n/4}$ get connected in $B_n$ at $p+\delta$. We use a non-trivial refinement and generalisation of \cite{benjamini2017homogenization}.     
	\paragraph{\it Section~\ref{central-proposition}: Proof of Proposition~\ref{prop:1}} This section gives the proof of Proposition~\ref{prop:1}. It uses all the results from Sections~\ref{sec:renorm-corr-funct}--\ref{sec:uniq-via-sprinkl}.
	
	\paragraph{\it Section~\ref{sec:proof-theor-1}: Proof of sharpness: coarse grains without rescaling} We use Proposition~\ref{prop:1} to establish Theorems~\ref{thm:1} and \ref{thm:2}. We perform a coarse-graining argument to reduce the study of the supercritical regime to that of a suitable perturbative regime. This is not done by rescaling the graph $G$ but by defining a $k$-independent percolation process on the original graph $G$.

	\subsection{Acknowledgements}
	\label{sec:acknoledgment}
	We thank Romain Tessera for his availability and for sharing his vision of nilpotent geometry. We are grateful to Itai Benjamini, Hugo Duminil-Copin, Aran Raoufi, Alain-Sol Sznitman, and Matthew Tointon for stimulating discussions. We thank the anonymous referees for the quality of their feedback.
	
	The first and third authors are supported by the European Research Council (ERC) under the European Union’s Horizon 2020 research and innovation program (grant agreement No 851565) and by the NCCR SwissMap. The second author acknowledges the support of the ERC Advanced Grant 740943 GeoBrown.

	\section{Geometric lemmas}
	\label{sec:geom-lemmas}
	
	In this section, we collect a certain number of geometric lemmas that will prove useful throughout the paper. This can be skipped in a first reading.
	
	\subsection{Cutsets}
	Percolation is primarily concerned with ``connectivity events'' and existence of open paths. The complement of such an event is well understood in terms of closed  cutsets, which are dual to open paths. In this section, we define and provide useful properties of cutsets.
	
	If $F$ and $\Pi$ denote two subsets of $V(G)$, we say that $\Pi$ is a \defini{cutset} between $F$ and $\infty$ if $\Pi \cap F=\varnothing$ and every infinite self-avoiding path starting in $F$ has to intersect $\Pi$ at some point. If, furthermore, no strict subset of $\Pi$ is a cutset between $F$ and $\infty$, we say that $\Pi$ is a \defini{minimal} cutset between $F$ and $\infty$.
	
	Given a subset $C\subset V(G)$ and a positive number $r$, we say that $C$ is $r$-connected if it is connected for the following graph structure: for any two distinct vertices of $C$, we declare them to be \defini{$r$-adjacent} if their distance in $G$ is at most $r$.
	
	The following well-known lemma provides some geometric control on these \emph{minimal} cutsets.
	
	\begin{lemma}[Coarse connectedness of minimal cutsets]
		\label{lem:1}
		Let $G$ be a transitive graph of polynomial growth with $d\ge 2$. Then, there is some constant $R$ such that the following holds.\\ Let $o$ be a vertex of $G$. Every minimal cutset disconnecting $o$ and $\infty$ is $R$-connected.
	\end{lemma}
	
	\begin{proof}
		By \cite{trofimov85polynomial}, $G$ is quasi-isometric to a Cayley graph $H$ of some finitely generated nilpotent group $\Gamma$. As $\Gamma$ is finitely generated and nilpotent, it is finitely presented (see Propositions~13.75 and 13.84 of \cite{drutu2017}).  Therefore, there is some constant $R_H$ such that every finite cycle of $H$ is a sum modulo 2 of cycles of length at most $R_H$ --- by cycle, we mean a (not necessarily self-avoiding) path ending where it starts. As $G$ is quasi-isometric to $H$, there is some constant $R_G$ such that every finite cycle of $G$ is a sum modulo 2 of cycles of length at most $R_G$. To see this, fix  $\varphi : G\to H$ and $\psi : H \to G$  quasi-isometries that are quasi-inverse of each other.  To every edge $\{u,v\}$ of $G$ (resp. $H$) correspond two vertices $\varphi(u)$ and $\varphi(v)$ (resp. $\psi(u)$ and $\psi(v)$), which are at bounded distance of each other. For each such edge $\{u,v\}$, select once and for all some  geodesic path connecting the two corresponding vertices in the other graph. This enables us to associate with each cycle in one graph a cycle in the other. The properties of this process allow us to derive from the existence of $R_H$ the existence of a constant $R_G$ as above. It is also the case that every infinite cycle (bi-infinite path) can be written as an infinite (locally finite) sum modulo 2 of cycles of length at most $R_G$. Indeed, as $d\ge2$, the graph $G$ is one-ended, meaning that removing any finite set of vertices leaves us with a unique \emph{infinite} connected component (and possibly some finite connected components). In a one-ended graph, any infinite cycle can be seen as a pointwise limit\footnote{This means the following: If a graph is one-ended, then for every bi-infinite self-avoiding path~$\gamma$, there is a sequence of self-avoiding cycles $\gamma_n$ such that for every edge $e$, for $n$ large enough, $e$ belongs to $\gamma$ if and only if it belongs to $\gamma_n$. To build such a sequence, consider the path $\gamma$ until it exits the ball $B_n$ and then use the fact that $G\setminus B_n$  contains a unique infinite connected component to connect the two
			endpoints outside $B_n$.}  of finite cycles, whence the extension of our decomposition to infinite cycles. What we have proved about $R_G$ permits us to use Theorem 5.1 of \cite{timar2007cutsets}  (see also \cite{babson1999cut, timar2013}), which yields the conclusion.
	\end{proof}
	
	The goal of this paper is to reduce the study of the whole supercritical regime to some perturbative regime, where the Peierls' argument applies. Studying this perturbative regime will require some quantitative control in the usual entropy-vs-energy spirit of the Peierls' argument. The next three lemmas will help us get such a control.
	
	\begin{lemma}
		\label{lem:2}
		Let $G$ be a transitive graph of polynomial growth with $d\ge 2$. Let $R$ be so that the conclusion of Lemma~\ref{lem:1} holds. Let $\Pi$ be a cutset disconnecting $o$ and $\infty$. Then, $\Pi$ intersects the ball of centre $o$ and radius $R|\Pi|$.
	\end{lemma}
	
	\begin{proof}
		Assume that $\Pi$ does not intersect $B_{R|\Pi|}$. Since $G$ is an infinite transitive graph, there is a bi-infinite geodesic path passing through $o$. The cutset $\Pi$ has to intersect each of the two geodesic rays this path induces from $o$. As a result, one has $\mathrm{diam}(\Pi)>2R|\Pi|$. Since $R$ satisfies the conclusion of Lemma~\ref{lem:1}, we also have $\mathrm{diam}(\Pi)\leq R|\Pi|$, which is contradictory.
	\end{proof}
	
	The next lemma is related to Lemma~\ref{lem:5} and \cite[Proposition 5]{funar2013}. The existence of a bi-infinite geodesic path will once again be the key of the proof.
	
	\begin{lemma}\label{lem:3}
		Let $G$ be an infinite transitive graph. Let $F$ be a finite subset of $V(G)$ containing $o$. Let $\Pi$ be a cutset disconnecting $F$ and $\infty$. Then, we have $\mathrm{diam}(\Pi)\geq \frac{\mathrm{diam}(F)}{2}$.
	\end{lemma}
	
	\begin{proof}	
		Let $\tilde F\subset V(G)$ denote the set of all vertices that can be reached from $F$ by a path avoiding $\Pi$. If $F$ is connected, this is simply the connected component of $F$ in $G\setminus \Pi$. We set $m=\mathrm{diam}(\Pi)$ and $n=\mathrm{diam}(F)$. Let us prove that $m\geq \frac{n}{2}$.
		
		First, assume that there is a vertex $v\in \tilde F$ such that $B_{n/4}(v)\subset \tilde F$. As in the proof of Lemma~\ref{lem:2}, the existence of a bi-infinite geodesic path passing through $v$ yields $m> 2n/4=n/2$, hence the desired result.
		
		Now, assume that, on the contrary, for every vertex $v\in \tilde F$, one has $B_{n/4}(v)\not\subset \tilde F$. In particular, one can take $u$ and $v$ two vertices of $F$ at distance $n$ of each other, and then find vertices $u'$ and $v'$ in $\Pi$ such that $d(u,u')\leq n/4$ and $d(v,v')\leq n/4$. Indeed, by definition of $\tilde F$, its external vertex-boundary is a subset of $\Pi$ --- and it is equal to $\Pi$ if $\Pi$ is a minimal cutset. By considering $u'$ and $v'$, one gets $m\geq n - 2n/4=n/2$.
	\end{proof}
	
	\begin{lemma}\label{lem:4}
		Let $G$ be a transitive graph of polynomial growth with $d\ge 1$. There exists $c>0$ such that the following holds. Let $F$ be a finite connected subset of $V(G)$ containing $o$. Let $\Pi$ be a cutset disconnecting $F$ and $\infty$. Then, we have $|\Pi|\geq c|F|^{\frac{d-1}{d}}$.
	\end{lemma}

	\begin{proof}
		Let $\tilde{F}\subset V$ be the set of vertices that can be connected to $F$ by a path not intersecting $\Pi$. Then $F\subset \tilde{F}$ and every edge of $\partial \tilde{F}$ intersects $\Pi$. As $|B_r| \asymp r^d$ and since every transitive amenable graph is unimodular \cite{soardiwoess}, Lemma 7.2 of \cite{lyonsmorrisschramm} yields the existence of $c>0$ such that for every choice of $(F,\Pi)$, we have $|\Pi|\geq c |\tilde F|^{\frac{d-1}{d}}\geq c |F|^{\frac{d-1}{d}}$.
	\end{proof}
	
	\subsection{Spheres and annuli}
	
	The most basic example of a cutset is the sphere of radius $r$. In the hypercubic lattice of dimension $d\ge 2$, spheres are coarsely connected, in the sense of Lemma~\ref{lem:1} (one can take $R=2$). Actually, for hypercubic lattices, spheres centred at $o$ are \emph{minimal} cutsets disconnecting  $o$ and $\infty$. Such statements are not true any more for general graphs --- not even for transitive graphs of polynomial growth satisfying $d\ge 2$.
	
	In order to recover coarse connectedness of spheres, we use a notion of exposed sphere \cite{MR3909043, pete2008note, timar2013}, that only contains the points accessible from infinity, and where the finite ``pockets'' are removed.  For $r\ge 0$ and $x$ a vertex of a graph $G$, let $S_r^\infty(x)$ \label{page:123456} denote the set of all vertices $y$ such that $d(x,y)=r$ and there exists an infinite self-avoiding path starting at $y$ and intersecting $B_r(x)$ only at $y$. Notice that replacing the self-avoiding condition by ``visiting infinitely many vertices'' in this definition yields the same set $S_r^\infty(x)$.
	If some vertex $o$ is fixed as a root in $G$, we may write $S_r^\infty$ for $S_r^\infty(o)$.
	
	\begin{remark}\label{rem:coarse-conn-spheres}
		For every $r\ge 1$, the set $S_r^\infty$ is a \emph{minimal} cutset between $B_{r-1}$ and $\infty$. In particular, Lemma~\ref{lem:1} yields coarse connectedness of these sets. 
	\end{remark}
	
	Not only does $S_r^\infty$ disconnect $B_{r-1}$ and $\infty$ but it even disconnects it from $\partial B_{2r}$, as stated in Lemma~\ref{lem:5}. This lemma corresponds to  Proposition 5 in \cite{funar2013}. For the reader's convenience, we have included below its short and nice proof.
	
	\begin{lemma}
		\label{lem:5}
		Let $G$ be an infinite transitive graph and $o$ a vertex of $G$. Let $r\ge 0$ and let $\gamma$ be a finite path starting in $B_r$ and that intersects $\partial B_{2r}$. Then, the path $\gamma$ intersects $S_r^\infty$.
	\end{lemma}

	\begin{proof}
		Let us fix $x$ and $y$ such that $x$ is a vertex of $\gamma$ at distance $2r$ from the origin $o$ and $y$ is a neighbour of $x$ satisfying $d(o,y)=2r+1$. Since $G$ is infinite and transitive, we can fix some bi-infinite geodesic path $\gamma'$ passing through $y$ at time 0. It is impossible for $\gamma'$ to intersect $B_r$ in both positive and negative times. Indeed,  since  $\gamma'$ is geodesic, this would imply the existence of two points in $B_r$ at distance larger than $2r+1$ away of each other.
		
		Therefore, by following $\gamma'$ from $y$ in the positive or the negative direction, we get an infinite self-avoiding path $\gamma''$ that does not intersect $B_r$. Since $\gamma$ starts in $B_r$ and intersects $\partial B_{2r}$, it visits at least one vertex at distance exactly $r$ from the origin. Take $k$ to be the largest integer such that $d(o,\gamma_k)=r$ and set $v=\gamma_k$. This vertex $v$ necessarily belongs to $S_r^\infty$. Indeed, following the path $\gamma$ started at time $k$, then the edge $\{x,y\}$, and then the path $\gamma''$ yields a path that starts at $v$ and then leaves $B_r$ forever (and that visits infinitely many vertices).  
	\end{proof}
	
	While studying percolation on $\mathbb{Z}^d$, it is customary to use not only spheres but also annuli. Annuli will also be useful when $G$ is a transitive graph of polynomial growth with $d\ge2$. The next lemma provides some control on their geometry.
	
	\begin{lemma}[Control on the intrinsic diameter of annuli]
		\label{lem:6}
		Let $G$ be a transitive graph of polynomial growth with $d\ge 2$. Let $R\ge 1$ be such that the conclusion of Lemma~\ref{lem:1} holds, and let $n \ge k \ge R$. Then for all $x,y\in S^\infty_n$, there exists a path from $x$ to $y$ within the $2k$-neighbourhood of $S^\infty_n$, of length at most $3k |B_{3n}|/|B_k|$.
	\end{lemma}
	
	\begin{proof}
		Let $x,y\in S^\infty_n$. By definition of $R$ and $S_n^\infty$, we can fix some path $(v_0,\dots,v_\ell)$ from $x$ to $y$ that stays in the $R$-neighbourhood of $S_n^\infty$. Recursively, we define some new finite sequence of vertices as follows. Set $w_0=v_0=x$. For every $i>0$, let $m_i$ be such that $v_{m_i}\in B_{2k}(w_{i-1})$ and none of the vertices $v_{m_i+1},\dots,v_\ell$ belongs to $B_{2k}(w_{i-1})$. We then set $w_i=v_{m_i+1}$. This process is well-defined until it reaches some $w_t$ that satisfies $d(w_t,y)\leq 2k$, and the process stops there. Notice that the sets $B_k(w_i)$ are disjoint when $i$ ranges over $\{0,\dots,t\}$ and that all these balls are subsets of $B_{3n}$. Therefore, we have $t+1\leq \frac{|B_{3n}|}{|B_k|}$. By means of paths of length at most $2k+1$ each, we can connect $w_0$ to $w_1$, \dots, $w_{t-1}$ to $w_t$, and $w_t$ to $y$. Concatenating theses paths produces a path of length at most $(2k+1)(t+1)$ that stays within the $2k$-neighbourhood of $S^\infty_n$. As $(2k+1)(t+1)\leq 3k |B_{3n}|/|B_k|$, the proof is complete.
	\end{proof}

	\subsection{Crafting many separated paths}
	
	In Section~\ref{sec:renorm-corr-funct}, we will need many paths that stay away from each other in $G$. The purpose of the current subsection is to explain how to get such paths. The proper (deterministic) statement is given by Lemma~\ref{lem:7}, which makes use of the following definition.
	
	A path $\gamma:\mathbb{N}\to V$ is said to be a \defini{$c$-quasi-geodesic ray} if for any $m$ and $n$ in $\mathbb{N}$, we have $d(\gamma_m,\gamma_n)\ge c|m-n|$.

	\begin{lemma}
		\label{lem:7}
		Let $G$ be a transitive graph of polynomial growth with $d\ge 2$. Then, there exist $c>0$ such that for every  $n\geq 1/c$, for every $a\in [1,cn]$, one can find $\lceil \frac{n}{a}\rceil$ distinct $c$-quasi-geodesic rays that intersect $B_n$ and that stay at distance at least $ca$ from each other --- if $x$ belongs to some ray  and $y$ to another one, then $d(x,y)\geq ca$.
	\end{lemma}

	\begin{proof}
		By \cite{gromov81,trofimov85polynomial}, there is a finitely generated nilpotent group $\Gamma$ such that $G$ is quasi-isometric to any Cayley graph of $\Gamma$. Let $S$ be some finite generating subset of $\Gamma$. By \cite[Lemma 3.23]{hutchcroft2021nontriviality}, because we have $d\ge2$, we can pick a surjective homomorphism $\pi : \Gamma\to\mathbb{Z}^2$.
		
		Notice that the conclusion of Lemma~\ref{lem:7} obviously holds for the square lattice. As this conclusion is stable under quasi-isometry, it holds for the Cayley graph of $\mathbb{Z}^2$ relative to the generating subset $\pi(S)$. By using $\pi$, lifting suitable paths for this Cayley graph of $\mathbb{Z}^2$ yields suitable paths for the Cayley graph of $\Gamma$ relative to $S$.
		As this graph is quasi-isometric to $G$, we have the desired result.
	\end{proof}
	
	\section{Renormalisation of the corridor function}\label{sec:renorm-corr-funct}
	
	Recall that we fixed $G=(V,E)$ a transitive graph of polynomial growth and that we assumed its growth exponent $d$ to satisfy $d\ge2$. The growth exponent $d$ is the unique integer such that $|B_n|\asymp n^d$.  We fix $R\ge 1$ some constant such that the conclusion of Lemma~\ref{lem:1} holds.
	
	\begin{proposition}\label{prop:2}
		Let $G$ be a transitive graph of polynomial growth with $d\ge 2$. Let $p\in (0,1)$ be such that 
		\begin{equation}
			\limsup_n \kappa_p(n \log^{3d} n,n)>0.\label{eq:6}
		\end{equation}
		Then, for every $q\in (p,1]$, for every $n$ large enough, we have
		\begin{equation}
			\mathrm P_q[\{B_{n/10}\lr{}\partial B_{n}\}\cap U(n/5,n/2)]\geq 1-e^{-\sqrt{n}}.\label{eq:7}
		\end{equation}
	\end{proposition}
	
	In order to prove Proposition~\ref{prop:2}, we will mimic the orange-peeling argument developed for percolation on $\mathbb{Z}^d$ \cite[Lemma 7.89]{MR1707339}. Since we will extend this argument in several ways, let us recall briefly how it works to obtain some lower bound on the probability of $U(n,2n)$ for percolation on $\mathbb Z^d$ at a parameter $p>p_c$. First, partition $B_{2n}\setminus B_{n}$ into $n/m$ disjoint annuli of the form $B_{k+m}\setminus B_k$ ($m$ is a fixed constant here, but its role will be important in our context so we keep it explicit). Using the slab technology \cite{grimmett1990supercritical}, one can show that in each of these annuli, any two vertices are connected with probability larger than some positive constant $\delta$. Now, consider two clusters $\mathcal C(x)$ and $\mathcal C(y)$ that cross the annulus $B_{2n}\setminus B_{n}$ and start at some fixed vertices $x$ and $y$ at the boundary of $B_{n}$. Explore them in a recursive manner by revealing the configuration in the annuli $B_{k+m}\setminus B_k$ the one after the other. At each such step, conditionally on the past, there is a probability at least $\delta$ that the clusters $\mathcal C(x)$ and $\mathcal C(y)$ get connected in $B_{k+m}\setminus B_k$. This shows that the probability that $\mathcal C(x)$ and $\mathcal C(y)$ reach the boundary of $B_{2n}$ without merging is smaller than $e^{-\delta n/m}$. Summing over all the possible $x$ and $y$, we get
	\begin{equation}
		\label{eq:8}
		\mathbb P[U(n,2n)]\ge1-|\partial B_{n}|^2e^{-\delta n/m}.
	\end{equation}
	
	However, two difficulties arise when extending this argument to more general graphs. First,  the annuli given by the induced subgraph of $B_{k+m}\setminus B_m$ may not be connected in our setting.  To overcome this, we rather work with the following annuli. For $n\ge m \ge0$, we define
	\begin{equation}
		A(n,m) = \bigcup_{x\in S^\infty_{n}}B_{m}(x).
	\end{equation}
	If $m$ is large enough, Lemma~\ref{lem:1} implies that $A(n,m)$ is connected and Lemma~\ref{lem:6} provides us with a good control on the distances in the graph induced by $A(n,m)$. 
	
	A second and more serious difficulty  is the lack of symmetry, which prevents us from using the slab technology of the Euclidean lattice, and makes the control of the two-point functions in annuli  more delicate. Instead of symmetries, our approach uses a renormalisation property of the corridor function, presented in Lemma~\ref{lem:9}. The latter  relies on a ``sprinkled'' version of the orange-peeling argument, presented in Lemma~\ref{lem:8} below.
	
	\begin{lemma}\label{lem:8} There exists a constant $c>0$ such that the following holds.  Let $p\in (0,1)$, $\delta\in (0,1-p)$ and $n\ge m\ge R$. If
		\begin{equation}
			\label{eq:9}
			\kappa_{p}({3}m{|B_{n}|}/{|B_{m}|},m)\geq \delta,
		\end{equation}
		then, for every $\ell\ge1$, we have
		\begin{equation}
			\label{eq:10}
			\kappa_{p+\delta}(\ell,n)\geq\delta^2 -\ell|B_n|^2\exp\left(-\frac{c\delta^3n}{m}\right).
		\end{equation}
	\end{lemma}

	\begin{proof}
		For simplicity, we will proceed as if all the ratios ($n/m$, $n/10$, $c_i n/m$...) appearing in the proof were integer-valued. The more general statement can easily be obtained by appropriate ``$\lceil\cdot\rceil$-and-$\lfloor\cdot\rfloor$-management''. Let $c_0>0$ be a constant as in Lemma~\ref{lem:7}.  
		
		By possibly reducing the value of the constant $c$, one may assume without loss of generality that the ratio $n/m$ is large. In particular, we may  assume that
		\begin{equation}
			\label{eq:11}
			\frac 3 {c_0}\: m\le \frac {c_0}{10}\:  n\quad\text{and}\quad \frac{11}{10c_0}\: n\le m\: \frac{|B_{n}|}{|B_m|}.
		\end{equation}
		
		In this proof, we work with the coupled configurations $(\omega_p)_{0\le p\le1}$ under the measure~$\mathbf P$. 
		First, we show that the classical annulus $B_n\setminus B_{n/10}$ is crossed by at least one $p$-cluster with high probability. To this end, apply Lemma~\ref{lem:7} with $a=\frac 3{c_0} m$.  It guarantees that  there exist  $c_1 n/m$ disjoint corridors of thickness $m$ connecting $B_{n/10}$ to $\partial B_{n}$ and of length at most $\frac{11}{10c_0}\:n$. Each one of these corridors is crossed with probability at least $\delta$, by the assumption \eqref{eq:9}. By independence, we have
		\begin{equation}
			\label{eq:12}
			\mathbf P[B_{n/10} \nlr[p]{}\partial B_{n}]\le \left(1-\delta \right)^{c_1n/m}\leq \exp\left(-\frac{\delta c_1 n} m\right). 
		\end{equation}
		
		By automorphism-invariance, any classical annulus $B_n(x)\setminus B_{n/10}(x)$ is crossed by at least one $p$-cluster with high probability. The main idea to prove \eqref{eq:10} is to consider a chain of $p$-clusters and  show that they all get connected at $p+\delta$, with high probability (as illustrated in Fig.~\ref{fig:1}). To this end,  we use an adapted version of the orange-peeling argument, and show that all the $p$-clusters in $B_{n/2}$ crossing from $B_{n/5}$ to $\partial B_{n/2}$ are locally $(p+\delta)$-connected to each other with high probability. More precisely, we prove that
		\begin{equation}
			\label{eq:13}
			\mathbf P [U_{p,p+\delta}(n/5,n/2)]\geq 1-|B_{n}|^2\exp\left(-\frac{\delta^3 n}{140m}\right).
		\end{equation}
		To achieve this, for $0\le i<\frac{n}{140m}$, consider the connected annuli $A_i=A(n_i,3m)$, where $n_i=\frac n5 +4m+7im$. Notice that $\frac n5+4m\le n_i\le \frac n4-3m$ and  the annuli $A_i$ are disjoint subsets of $B_{n/4}\setminus B_{n/5}$. The choice of $3m$ for the thickness of the annulus ensures that for any vertices $x,\,y$  in $S^\infty_{n_i}$, one can find a corridor from $x$ to $y$ of thickness $m$ and length smaller than ${3}m|B_{3n_i}|/|B_m|$ that fully lies inside $A_i$, by Lemma~\ref{lem:6}.   Let us define $\mathcal C_p(x)$ to be the $p$-cluster of $x$ in $B_{n/2}$. For $x,y\in B_{n/5}$, we have
		\begin{align}
			\label{eq:14}
			&P[x\lr[p]{}\partial B_{n/2},~y\lr[p]{}\partial B_{n/2},~x\nlr[p+\delta]{B_{n/2}}y]\\
			&\qquad= \sum_{C,C'} \mathbf P[C\nlr[p+\delta]{B_{n/2}} C'\ |\ \mathcal C_p(x)=C,~\mathcal C_p(y)=C']\cdot \mathbf P [\mathcal C_p(x)=C,~\mathcal C_p(y)=C']
		\end{align}
		where the sum is over the pairs of disjoint clusters $(C,C')$ joining respectively $x$ and $y$ to $\partial B_{n/2}$. Let us fix such $C$ and $C'$. Notice that under the conditional law 
		\begin{equation}
			\widetilde{\mathrm P}=\mathbf P[\ \cdot\ |\ \mathcal C_p(x)= C,~\mathcal C_p(y)=C'],
		\end{equation}
		the configuration $\omega_{p+\delta}$ is an independent percolation process with marginals satisfying
		\begin{equation}
			\label{eq:15}
			\widetilde{\mathrm P}[\omega_{p+\delta}(e)=1]
			\begin{cases}
				= 1 &\text{ if } e\in C\cup C',\\
				\geq  \delta &\text{ if } e\in \partial C \cup \partial C',\\
				=  p+\delta &\text{ otherwise.}
			\end{cases}
		\end{equation}
		By Lemma~\ref{lem:5},  both $C$ and $C'$ must intersect all the exposed spheres $S^\infty_{n_i}$ (since $n_i\le n/4$). Furthermore, for every $i$,  Lemma~\ref{lem:6} ensures that there exists a corridor included in $A_i$  from  a point of $C$ to  a point of $C'$, of thickness $m$ and of length smaller than ${3}m|B_{3n_i}|/|B_m|$. Using that $3n_i\le n$ and  the bound \eqref{eq:9}, we  obtain
		\begin{equation}
			\mathbf P[\partial C\lr[p]{A_i} \partial C'\ |\ \mathcal C_p(x)=C,~\mathcal C_p(y)=C']\geq \kappa_p\left({3}m{|B_{n}|}/{|B_m|},m \right)\geq \delta.
		\end{equation}
		Above, we do not directly get that the two clusters $C$ and $C'$ are connected by a $(p+\delta)$-open path, because we are conditioning on an event where the edges at the boundary of $C$ and $C'$ are $p$-closed. Nevertheless, using that each such edge is $(p+\delta)$-open with probability larger than $\delta$, independently of the event $\partial C\lr[p]{A_i} \partial C'$, one obtains the following lower bound:
		\begin{equation}
			\mathbf P[ C\lr[p+\delta]{A_i} C'\ |\ \mathcal C_p(x)=C,~\mathcal C_p(y)=C']\geq \delta^3,
		\end{equation}
		where one additional $\delta$ appears in order to leave $C$ and another $\delta$ to reach $C'$.
		Finally, since all $A_i$'s are disjoint, one has, by independence:
		\begin{equation}
			\mathbf P[C\nlr[p+\delta]{B_{n/2}} C'\ |\ \mathcal C_p(x)=C,~\mathcal C_p(y)=C']\leq (1-\delta^3)^{n/140m}.
		\end{equation}
		Plugging this into \eqref{eq:14} and then summing over all pairs $x,y\in B_{n/5}$, one gets
		\begin{equation}
			\mathbf P [U_{p,p+\delta}(n/5,n/2)]\geq 1-|B_{n}|^2\exp\left(-\frac{\delta^3 n}{140m}\right).
		\end{equation}
		
		Now, we use the two estimates \eqref{eq:12} and \eqref{eq:13} in order to prove that \eqref{eq:10} holds. Let  $\gamma=(v_0,v_1,\ldots,v_k)$ be a path of length $k\le\ell$ connecting a vertex  $x=v_0$ to  some vertex $y=v_k$. Let $K$ be the corridor around $\gamma$ of thickness $n$. Consider the following three conditions:
		\begin{enumerate}[label=(\roman*)]
			\item\label{item:1} $x\lr{}\partial B_{n}(x)$ and $y\lr{}\partial B_{n}(y)$,
			\item\label{item:2} for all $j=1,\ldots,k$, one has $B_{n/10}(v_j)\lr[p]{}\partial B_{n}(v_j)$,
			\item\label{item:3} for all $j=1,\ldots,k$, the uniqueness event $U_{p,p+\delta}(n/5,n/2)$ centred at $v_j$ occurs. 
		\end{enumerate}
		\begin{figure}[htbp]
			\centering
			\includegraphics[scale=0.28]{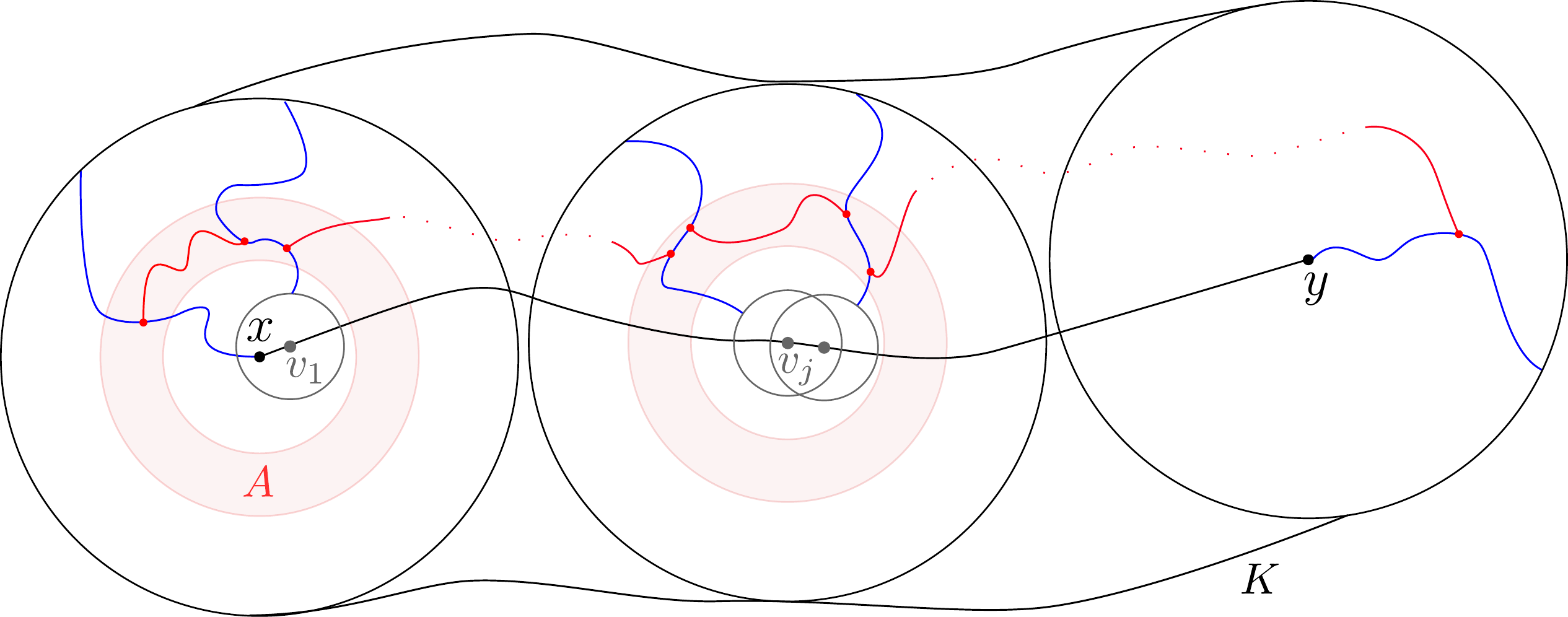}
			\caption{Gluing paths  via local uniqueness.}\label{fig:1}
		\end{figure}%
		As illustrated in Fig.~\ref{fig:1},  the simultaneous occurrence of the three events above implies that there exists a $(p+\delta)$-open path in $K$ between $x$ and $y$. The bounds $\mathbf P[z\lr{}\partial B_{n}(z)]\ge \kappa_p(n,m)$ for $z=x,y$ together with the  Harris--FKG Inequality imply that \ref{item:1} occurs with probability at least $\kappa_p(n,m)^2$. Therefore, by the union bound  and by automorphism-invariance, we obtain
		\begin{align}
			\mathbf P[x\lr[p+\delta]{K}y]&\!\!\!\!\overset{\phantom{\eqref{eq:12},\eqref{eq:13}}}{\geq} \!\kappa_p(n,m)^2 -k \mathbf P[ B_{n/10}\nlr[p]{}\partial B_{n}] -k(\mathbf P[U_{p,p+\delta}(n/5,n/2)^c])\\
			&\!\!\!\!\overset{\eqref{eq:12},\eqref{eq:13}}{\geq}\! \delta^2 - \ell \exp\left( -\frac{\delta c_1n}m\right) -\ell|B_{n}|^2\exp\left(-\frac{\delta^3 n}{140m}\right),  
		\end{align}
		which concludes the proof of \eqref{eq:10}.
	\end{proof}
	
	\begin{lemma}[Renormalisation property of the corridor function]
		\label{lem:9}
		Let $p>0$ and assume that 
		\begin{equation}
			\label{eq:16}
			\limsup_{n\to\infty} \kappa_p(n\log^{3d}n,n)>0.
		\end{equation}
		Then for all $q>p$ and $C>2d$,  for all $n$ large enough, one has 
		\begin{equation}
			\label{eq:17}
			\kappa_{q}(n^C,n) \geq n^{-1/C}.
		\end{equation}
	\end{lemma}
	
	\begin{remark}
		Assuming some nice behaviour at infinitely many scales, Lemma~\ref{lem:9} enables us to get some control valid at \emph{every} scale large enough.
	\end{remark}
	
	\begin{proof}
		Let us fix $p>0$ such that \eqref{eq:16} holds. Let $\eta\in(0, \frac{1-p}2)$ and $q=p+2\eta$, let  $C>2d$.	The assumptions imply that 
		\begin{equation}
			\label{eq:18}          
			\delta:=\min\big\{\eta,\,\tfrac12\limsup_{n\to\infty}\kappa_p(n\log^{3d-3}n,n\log^{-2} n)\big\}>0.
		\end{equation}
		Consider a large $n$ satisfying $\kappa_p(n\log^{3d-3}n,n\log^{-2}n)\ge\delta$ and write $m=\frac{n}{\log^2 n}$.   As $|B_n|\asymp n^d$, we have ${3}\frac{|B_{n}|}{|B_m|}\le \log^{2d+1}n$ and therefore $\kappa_{p}({3}m\frac{|B_{n}|}{|B_m|},m) \ge \delta$. Applying Lemma~\ref{lem:8} to  $n,m$ and $\ell=n^{C^2}$, we obtain
		\begin{equation}
			\label{eq:19}
			\kappa_{p+\eta}(n^{C^2},n) \ge \kappa_{p+\delta}(n^{C^2},n)\ge \delta^2-n^{C^2}|B_n|^2\exp\big(-c\delta^3\log^2 n \big).  
		\end{equation}
		Since one can find arbitrarily large $n$ satisfying the equation above, we get 
		\begin{equation}
			\label{eq:1729}
			\displaystyle\limsup_{n\to\infty}\kappa_{p+\eta}(n^{C^2},n)>0.
		\end{equation}
		To get \eqref{eq:17}, we perform a renormalisation argument. Let $c_0$ be a constant as in Lemma~\ref{lem:8}. Equation~\eqref{eq:1729} allows us to take  $m_0\geq\max\{\frac{1+\eta}{\eta},R\}$  such that   
		\begin{equation}
			\label{eq:20}
			\kappa_{p+\eta}(m_0^{C^2},m_0)\ge\frac{1}{m_0^{1/C}}
		\end{equation}
		and, for every $m\ge m_0$,
		\begin{equation}
			\label{eq:21}
			m^{C^3}|B_{m^2}|^2e^{-c_0 m^{1-3/C}}\le \frac{1}{m^{2/C}}-\frac{1}{m}\qquad \text{and}\qquad  {3} m\frac{|B_{m^2}|}{|B_{m}|}\le m^{C^2}.
		\end{equation}
		Let $m_i=m_0^{C^i}$ and $p_{i+1}=p_{i}+1/m_{i-1}$, with $p_0=p+\eta$. We will prove by induction that for every $i\ge 0$,
		\begin{equation}
			\label{eq:22}
			\kappa_{p_i}(m_{i+2},m_i)\ge \frac{1}{m_{i-1}}.
		\end{equation}
		Notice that the relation holds for $i=0$, by definition of $m_0$.
		Let us fix some $i$ such that \eqref{eq:22} holds.  Since ${3} m_i\frac{|B_{m_i^2}|}{|B_{m_i}|}\le m_i^{C^2}=m_{i+2}$,  one has
		\begin{equation}
			\kappa_{p_i}\big({3} m_i{|B_{m_i^2}|}/{|B_{m_i}|},m_i\big)\ge\frac{1}{m_{i-1}}.
		\end{equation}
		Using Lemma~\ref{lem:8} and observing that $m_{i+1}\geq m_i^2$, we get
		\begin{align}
			\kappa_{p_{i+1}}(m_{i+3},m_{i+1})\geq\kappa_{p_{i+1}}(m_{i+3},m_i^2)&\geq\frac{1}{m_{i-1}^2}-m_{i+3}|B_{m_i^2}|^2\exp\left(-\frac{c_0m_i^2}{m_{i-1}^3 m_i}\right)\\
			&\geq \frac{1}{m_i},
		\end{align}
		where the last inequality holds due to the choice of $m_0$. This completes the induction.
		
		Now, for $n$ large, consider $j$ such that $m_j\leq n<m_{j+1}$. Using that $p_j\le p+2\eta$ and monotonicity, we get
		\begin{equation}	
			\kappa_{p+2\eta}(n^C,n)\geq\kappa_{p_j}(m_{j+2},m_j)\geq \frac{1}{m_j^{1/C}}\geq\frac{1}{n^{1/C}}.
		\end{equation}
	\end{proof}
	
	\begin{proof}[Proof of Proposition~\ref{prop:2}]
		As mentioned before, the strategy of the proof is an adaptation of the orange-peeling argument used in \cite[Lemma 7.89]{MR1707339}. Fix $p\in(0,1)$ and $\eta\in(0,(1-p)/2)$. We will prove that the statement \eqref{eq:7} holds for $q=p+2\eta$. Let $C=8d$. By Lemma~\ref{lem:9}, we can fix  $m_0\ge1$ such that for all $m\geq m_0$, we have $\kappa_{q}(m^C,m)\geq m^{-1/C}$. Let~$n$ be large enough, so  that it is possible to choose an $m$ satisfying
		\begin{equation}
			\label{eq:23}
			m_0 \leq m\leq n^{1/9}\text { and } m^C\geq {3}m|B_{3n}|/|B_m|.
		\end{equation}
		As in the proof of Lemma~\ref{lem:8}, we consider disjoint annuli inside $B_{n/4}\setminus B_{n/5}$: for $0\le i<\frac{n}{140m}$, consider the connected annuli $A_i=A(n_i,3m)$, where $n_i=\frac n5 +4m+7im$. The gist of the proof is to establish, for all $i<\frac n{140m}$ and all $x,y\in A_i$, a good lower bound for $\mathrm P_q[x\lr{A_i} y]$. Let $i<\frac1{140m}$ and $x,y\in A_i$. Denote by $x'$ a vertex of $S^\infty_{n_i}$ such that $d(x,x')=d(x,S^\infty_{n_i})$. Define $y'$ in the same manner, with respect to $y$ instead of $x$. By Lemma~\ref{lem:6}, one can find a corridor from $x'$ to $y'$ of thickness $m$ and of length smaller than $3m|B_{3n_i}|/|B_m|$ that fully lies within $A_i$. Using the Harris--FKG Inequality, we get
		\begin{equation}
			\label{eq:24}
			\mathrm P_{q}[x\lr{A_i}y]\geq \mathrm P_{q}[x\lr{B_{3m}(x')}x']\cdot\kappa_{q}(3m |B_{3n}|/|B_{m}|,m)\cdot \mathrm P_{q}[y\lr{B_{3m}(y')}y'].
		\end{equation}
		Then, we only need to find a good lower bound for each factor of this product. Observe that \eqref{eq:23} and Lemma~\ref{lem:9} directly yield the following lower bound for the middle factor:
		\begin{equation}
			\label{eq:25}
			\kappa_{q}({3}m|B_{3n}|/|B_{m}|,m)\geq m^{-1/C}.
		\end{equation}
		For the other factors of \eqref{eq:24}, by automorphism-invariance, it is enough to get a lower bound for $\mathrm P_q[o\lr{B_{3m}} z]$, where $z\in B_{3m}$. In order to get this bound, we will use a chaining argument, which goes as follows. For all $k$, set $m_k=m_0^{2^k}$. Let $j\geq 0$ be such that 
		\begin{equation}
			m_0+\cdots+m_{j-1}\leq d(o,z) < m_0+\cdots+m_j.
		\end{equation}
		\begin{figure}
			\centering
			\includegraphics[scale=0.48]{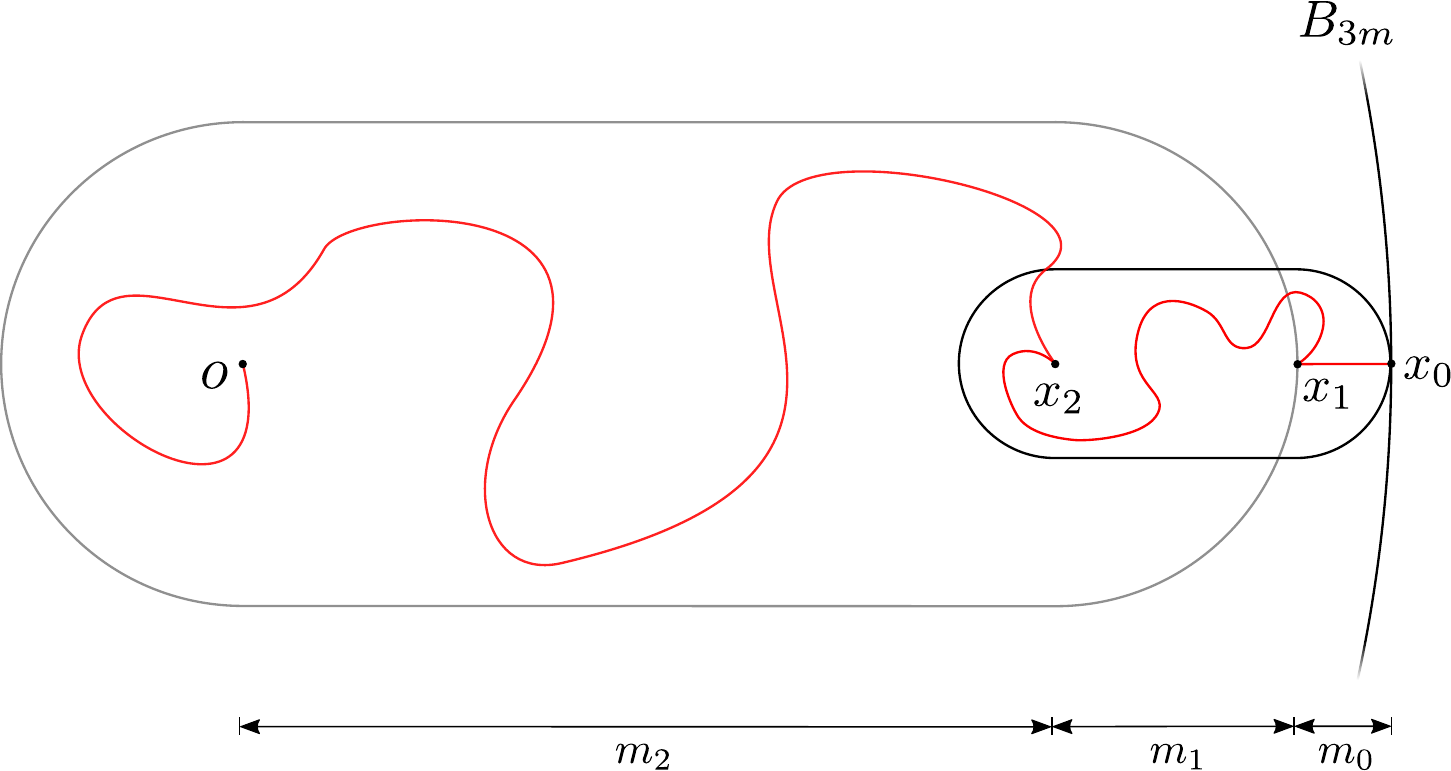}
			\label{fig:2}
			\caption{Illustration of the chaining argument: a chain of corridors is used to  connect the  centre of the ball to a boundary vertex $x_0$.  For the corridors to be subsets of the ball, the thickness  is chosen to be smaller when approaching the boundary.}
		\end{figure}%
		Let $\gamma$ be a geodesic path between $z$ and $o$. Fix a sequence of points $x_0=z, x_1,\ldots, x_{j+1}=o$ in $\gamma$ such that for all $0\leq i<j$, $d(x_i,x_{i+1})=m_i$ and $d(x_j,o)< m_j$. Notice that for every $1\leq i < j$, the corridor around $\gamma$ of thickness $m_{i-1}$ between $x_i$ and $x_{i+1}$ is contained in the ball $B_{3m}$. This is not the case for $i=0$,  but in this case we use the trivial bound $\mathrm P_q[x_0\lr{B_{3m}}x_1]\geq p^{m_0}$. Then, using the Harris--FKG Inequality, we get
		\begin{equation}
			\label{eq:26}
			\mathrm P_{q}[o\lr{B_{3m}} z]\geq p^{m_0}\prod_{i=1}^{j}\kappa_{q}(m_i,m_{i-1}).
		\end{equation}
		Notice that $\prod_{i=1}^{j}\kappa_{q}(m_i,m_{i-1})\ge\prod_{i=1}^{j}m_0^{-2^{i-1}/C} = m_0^{\frac{1-2^j}{C}}\ge \tfrac{1}{m_{j-1}}$, where the last inequality uses $C\ge2$. Combining this inequality with \eqref{eq:26} and $m_{j-1}\le d(o,z)\le 3m$ yields
		\begin{equation}
			\label{eq:299}
			\mathrm P_{q}[o\lr{B_{3m}} z]\geq \frac{p^{m_0}}{3m}.
		\end{equation}
		Using \eqref{eq:24}, \eqref{eq:25} and \eqref{eq:299}, we obtain that for all $m$ large enough,
		\begin{equation}
			\label{eq:27}
			\mathrm P_{q}[x\lr{A_i}y]\geq \frac{p^{2m_0}}{(3m)^2}\frac{1}{m^{1/C}}\geq \frac{1}{m^3}.
		\end{equation}
		
		Now, we describe the adaptation of the orange-peeling argument presented in \cite[Lemma 7.89]{MR1707339} to our context. Let $a$ and $b$ denote two distinct vertices of $B_{n/10}$. We explore their clusters step by step, starting from the inside ball $B_{n/10}$. The $i^\text{th}$ step explores these clusters until they touch the annulus $A_i$ --- or until we are sure they will never do so. Denote by $a_i$ and $b_i$ the respective points where they touch $A_i$ for the first time.
		Notice that the information revealed up to the end of the $i^\text{th}$ step does not reveal the status of any edge in $A_i$. By independence, \eqref{eq:27} tells us that, conditionally on the information revealed up to the end of the $i^\text{th}$ step, and provided $a_i$ and $b_i$ are well-defined, then the probability to have $a_i\lr{A_i}b_i$ is at least $\tfrac{1}{m^3}$. We deduce that, for every $n$ large enough,
		\begin{equation}
			\mathrm P_q[U(n/5,n/2)]\geq 1-|B_{n/10}|^2\left(1-\frac{1}{m^{3}}\right)^{n/140m}\geq 1-\frac{e^{-\sqrt{n}}}{2},
		\end{equation}
		where we use that $m\le n^{1/9}$.
		\begin{figure}
			\centering
			\includegraphics[scale=0.23]{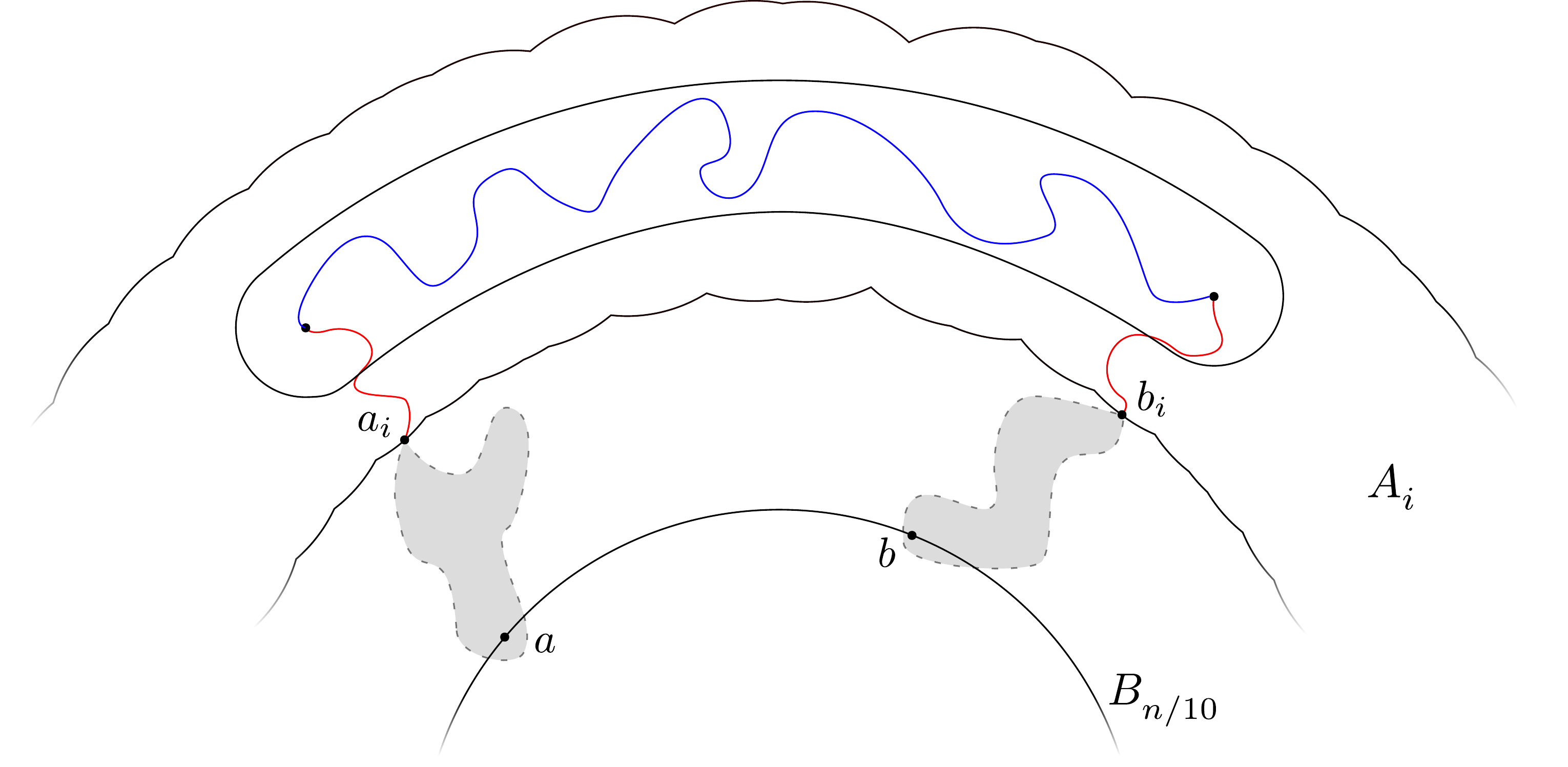}
			
			\caption{Orange-peeling argument: when the explorations of two clusters from $B_{n/10}$ reach the annulus $A_i$, they get connected  inside $A_i$ with probability at least $1/m^3$.}	\label{fig:3}
		\end{figure}%
		It remains to prove that for every $n$ large enough, we have
		\begin{equation}
			\mathrm P_q[B_{n/10}\lr{} \partial B_n]\geq 1-\frac{e^{-\sqrt{n}}}{2},
		\end{equation}
		which can be proved as Equation~\eqref{eq:12} by using that there exist  at least $c n^{1-{1/C}}$ disjoint corridors of thickness $n^{1/C}$ from $B_{n/10}$ to $ \partial B_n$, and that each of these corridors is independently crossed with probability at least $n^{-1/C}$.    
	\end{proof}
	
	\section{Probability that two clusters meet at one point}
	\label{sec:quantitative-uniqueness}
	
	This section is devoted to the proof of the following proposition, which extends the  quantitative uniqueness argument of \cite{MR901151,gandolfi1988uniqueness} to graphs of polynomial growth.  The proof  presented below follows the more recent presentations of \cite{cerf2015} and \cite{2018arXiv180808940H}. Recall that $\piv m n$ is defined in \eqref{eq:2}, on page~\pageref{eq:2}.
	
	\begin{proposition}\label{prop:3}
		Let $G$ be a transitive graph of polynomial growth. Let $\varepsilon>0$ and $\eta>0$. There is a constant $c=c(G,\varepsilon,\eta)$ such that for all $p\in [\eta,1]$ and all $n\ge 1$, we have
		\begin{equation}\label{eq:29}
			\mathrm P_p[\piv 1 n] \le cn^{-1/2+\varepsilon}.
		\end{equation}
	\end{proposition}
	
	The interest of Proposition~\ref{prop:3} is twofold. On the one hand, it will be useful to bound the probability  of pivotal edges when studying  the derivative of crossing probabilities: it will be important to establish the general sharp threshold result of Proposition~\ref{prop:8}.  On the other hand, the work \cite{cerf2015}  shows that for $G=\mathbb Z^d$ the bound above can be strengthened into  bounds on the probability $\mathrm P_p[\piv m n]$ for  $m\ge1$. For more general graphs, we will  obtain a  similar result in  Proposition~\ref{prop:4}, but we need to combine the approach of \cite{cerf2015} with a new  renormalisation argument in order to overcome the lack of symmetry of $G$.

	In order to prove Proposition~\ref{prop:3}, we will make use of the following geometric observation.
	
	\begin{lemma}\label{lem:10}
		Let $G$ be a transitive graph of polynomial growth. There is a constant $c$ such that the following holds: for every $n\geq 1$, there is an integer $m$ such that $n\leq m<2n$ and $|\partial B_m|/|B_{m}|\leq c/n$.
	\end{lemma}
	
	\begin{proof}
		Using that the sum $\sum_{m=n}^{2n-1}|\partial B_m|$  is smaller than the total number of edges in $B_{2n}$, we have
		\begin{equation}
			\label{eq:151}
			\min_{n\le m<2n} \frac{|\partial B_m|}{|B_{m}|}\le \frac1n \sum_{m=n}^{2n-1}  \frac{|\partial B_m|}{|B_{m}|}\le\frac1n\cdot \frac{|E(B_{2n})|}{|B_{n}|}.
		\end{equation}
		Since $|B_n|\asymp n^d$ and $G$ has bounded degree, $\frac{|E(B_{2n})|}{|B_{n}|}$ is bounded from above by a constant, which concludes the proof. 
	\end{proof}
	
	\begin{proof}[Proof of Proposition~\ref{prop:3}]
		Let $c_1>0$ be a constant such that the conclusion of Lemma~\ref{lem:10} holds, and let $n$ be a positive integer. Then, there is an integer $m$ such that $n\leq m<2n$ and
		\begin{equation}
			|\partial B_m|/|B_{m}|\leq c_1/n.\label{eq:30}
		\end{equation}
		Let us fix such an $m$. Following \cite{cerf2015}, let us consider the random set of edges $\mathcal H$ defined as follows: we look at the configuration restricted to $B_m$ and we say that $e$ belongs to $\mathcal H$ if it is closed and its endpoints belong to disjoint clusters in $B_m$ that both touch $\partial B_m$. Since $4n\ge 2m$, one can use the finite-energy property and automorphism-invariance to prove that for every $x\in B_{m-1}$,
		\begin{equation}
			\label{eq:31}
			p  \cdot \mathrm P_p[\piv 1 {4n}]\le  \mathrm{deg}(G) \sum_{e\ni x} \mathrm P_p[e\in \mathcal H].
		\end{equation}
		Indeed, whenever $\piv 1 {4n}$ holds, there must exist two edges $e$ and $e'$, both containing $x$, such that if we force $e'$ to be open, then $e$ belongs to $\mathcal{H}$ in this new configuration.
		Summing over all $x\in B_{m-1}$, we get 
		\begin{equation}\label{eq:32}
			|B_{m-1}|\cdot p\cdot  \mathrm P_p[\piv 1 {4n}]
			\leq 2\cdot \mathrm{deg}(G) \mathrm E_p[|\mathcal H|].
		\end{equation}
		Let us define $\mathfrak{C}$ to be the family of all the open clusters in $B_m$ that intersect $\partial B_m$. Define $\overline{\mathcal C}$ to be the union of all these clusters. Given a subset $S$ of $B_m$, we write $\mathsf{open}(S)$ (resp.\@ $\mathsf{closed}(S)$)  for the set of open (resp.\@ closed) edges of $E(B_m)$ adjacent to at least one vertex of $S$. In order to bound the size of  $\mathcal H$, we rely on the following identities,
		\begin{align}
			\label{eq:33}
			&\sum_{\mathcal C \in \mathfrak C}|\mathsf{open}(\mathcal C)|= |\mathsf{open}(\overline{\mathcal C})|\\
			&\sum_{\mathcal C \in \mathfrak C}|\mathsf{closed}(\mathcal C)|= |\mathsf{closed}(\overline{\mathcal C})|+|\mathcal H|,\label{eq:34}
		\end{align}
		which are proved by a counting argument. The first sum counts all the open edges in $B_m$ connected to the boundary of $B_m$, which also corresponds to the open edges adjacent to $\overline{\mathcal C}$. The second sum counts  all the closed edges which are connected to the boundary of $B_{m}$, except that  the edges of $\mathcal H$ are counted twice.
		
		Furthermore, using that  the event $\{e \lr{} \partial B_m\}$ is independent of the status of the edge~$e$, we find
		\begin{align}
			\label{eq:35}
			p\cdot \mathrm E_p[|\mathsf{closed}(\overline{\mathcal C})|]
			&=\sum_{e\subset B_m}p\cdot \mathrm P[ e\text{ is closed},\, e\lr{}\partial B_m]\\
			&=\sum_{e\subset B_m}p(1-p)\cdot\mathrm P[e\lr{}\partial B_m]\\
			&=\sum_{e\subset B_m}(1-p)\cdot \mathrm P[e\text{ is open},\,e\lr{}\partial B_m]\\
			&=(1-p)\cdot \mathrm E_p[|\mathsf{open}(\overline{\mathcal C})|].
		\end{align}
		Taking the expectation in \eqref{eq:33} and \eqref{eq:34}, and using the computation above, we get    
		\begin{equation}
			\label{eq:36}
			\mathrm E_p[|\mathcal H|]=\frac1p\mathrm E_p[\sum_{\mathcal C\in \mathfrak C} h(\mathcal C)],
		\end{equation}
		where  $h(\mathcal C)=p|\mathsf{closed}(\mathcal C)|-(1-p)|\mathsf{open}(\mathcal C)|$.
		In the right-hand side above, the fact that the sum is over a random set makes it delicate to study. To overcome this difficulty, we will ``root'' each cluster at some of its vertices, and sum over the possible roots. For every $x\in B_m$, let $\mathcal C_{x}$ be the cluster of $x$ in $B_m$. Using that for every cluster $\mathcal C$, there are  exactly $|\mathcal C|$ vertices $x$ such that $\mathcal C_x=\mathcal C$, we get
		\begin{equation}\label{eq:37}
			\mathrm E_p\left[\sum_{\mathcal C\in\mathfrak C}f( \mathcal C)\right]= \sum_{x\in B_m}\mathrm E_p\left [\frac{f(\mathcal C_x)}{|\mathcal C_x|}\mathbf{1}_{\{x\lr{}\partial B_m\}}\right],
		\end{equation}
		for every function $f:\mathcal P(B_m)\to\mathbb R$. Applying this to $f=h$, and using the Cauchy--Schwarz inequality, we get 
		\begin{align}
			\mathrm E_p\left[\sum_{\mathcal C\in \mathfrak C} h(\mathcal C)\right] \leq\left(\sum_{x\in B_m}\mathrm E_p \left[\frac{h(\mathcal C_x)^2}{|\mathcal C_x|^{1+\varepsilon}}\right]\right)^{1/2} \left(\sum_{x\in B_m}\mathrm E_p \left[\frac{\mathbf{1}_{\{x\lr{}\partial B_m\}}}{|\mathcal C_x|^{1-\varepsilon}}\right]\right)^{1/2}.\label{eq:38}
		\end{align}
		Applying \eqref{eq:37} to $f(S)=|S|^{\varepsilon}$ and using Hölder's inequality with parameters $\frac 1{\varepsilon}$ and $\frac1{1-\varepsilon}$ (we may and will assume that $\varepsilon <1)$,  we obtain
		\begin{equation}\label{eq:39}
			\sum_{x\in B_m}\mathrm E_p \left[\frac{\mathbf{1}_{\{x\lr{}\partial B_m\}}}{|\mathcal C_x|^{1-\varepsilon}}\right]
			= \mathrm E_p\big[\sum_{\mathcal C\in \mathfrak C}|\mathcal C|^\varepsilon\big]
			\leq \mathrm E_p\Big[\big(\sum_{\mathcal C\in\mathfrak C}|\mathcal C|\big)^\varepsilon|\mathfrak C|^{1-\varepsilon}\Big]
			\leq |B_m|^\varepsilon\,|\partial B_m|^{1-\epsilon},
		\end{equation}
		where in the third step we used that every cluster in $\mathfrak C$ has to touch $\partial B_m$. By \eqref{eq:32}, \eqref{eq:36}, \eqref{eq:38} and \eqref{eq:39}, we find
		\begin{equation}
			\label{eq:40}
			\mathrm P_p[\piv 1 {4n}]\le c_2 \left(\tfrac1{|B_m|}\sum_{x\in B_m}\mathrm E_p \left[\frac{h(\mathcal C_x)^2}{|\mathcal C_x|^{1+\varepsilon}}\right]\right)^{1/2} \left(\frac{|\partial B_m|}{|B_{m}|}\right)^{(1-\varepsilon)/2}, 
		\end{equation}
		where $c_2$ is a finite constant depending only on $\eta$.  
		To conclude the proof, it suffices to show that  for every fixed $x\in B_m$, the quantity
		\begin{equation}
			\label{eq:41}
			\mathrm E_p \left[\frac{h(\mathcal C_x)^2}{|\mathcal C_x|^{1+\varepsilon}}\right] 
		\end{equation}
		is smaller than some constant $c_3<\infty$. This will be achieved by interpreting $h(\mathcal C_x)$ as the value of a martingale at some stopping time, when we perform a certain exploration of the cluster $\mathcal C_x$.
		
		First, fix an arbitrary deterministic ordering of the edges of $B_m$. Set $(O_0,C_0)=(\emptyset,\emptyset)$. Then, let $e_1$ be the smallest edge adjacent to $x$ and set
		\begin{equation}
			(O_1,C_1)=
			\begin{cases}
				(\{e_1\},\emptyset)&\text{ if $e_1$ is open}\\
				(\emptyset,\{e_1\})&\text{ if $e_1$ is closed}.
			\end{cases}
		\end{equation}
		By induction, for $t\geq 2$,   define $e_t$ to be the smallest  edge  that is adjacent to an edge of $O_{t-1}$ (if it exists) and
		\begin{equation}
			(O_t,C_t)=
			\begin{cases}
				(O_{t-1}\cup\{e_t\},C_{t-1})&\text{ if $e_t$ is open}\\
				(O_{t-1},C_{t-1}\cup\{e_t\})&\text{ if $e_t$ is closed}.
			\end{cases}
		\end{equation}
		If there is no such edge, we define $(O_t,C_t)=(O_{t-1},C_{t-1})$. Observe that at this moment, $(O_t, C_t)$ corresponds to $(\textsf{open}(\mathcal C_x), \textsf{closed}(\mathcal C_x))$. Let $T=|\textsf{open}(\mathcal C_x)|+| \textsf{closed}(\mathcal C_x)|$ be the time at which the exploration stabilises. We define
		\begin{equation}X_t=\sum_{k=1}^{t\wedge T} p\mathbf{1}_{\{\omega(e_k)=0\}}-(1-p)\mathbf{1}_{\{\omega(e_k)=1\}}.
		\end{equation}
		Notice that $X_T=h(\mathcal C_x)$ and $T\leq \deg(o)|\mathcal C_x|$, since any vertex is adjacent to at most $\deg(o)$ edges in $B_m$. Thus, 
		\begin{equation}\label{eq:42}
			\mathrm E_p\left[\frac{h(\mathcal C_x)^2}{|\mathcal C_x|^{1+\varepsilon}}\right]\leq \mathrm \deg(o)^{1+\varepsilon}\mathrm E_p\left[\frac{X_T^2}{T^{1+\varepsilon}}\right].
		\end{equation}
		In order to upperbound $\mathrm E_p[X_T^2/T^{1+\varepsilon}]$, we use 
		that $(X_t)$ is a martingale with respect to the filtration generated by $\{\omega(e_t)\}_t$.  By first applying Doob's maximal  inequality and then using orthogonality of the increments, we obtain
		\begin{equation}\label{eq:152}
			\mathrm E_p[ \max_{k\le t}X_k^2]\le 4 \mathrm E_p[X_t^2] = 4 \sum_{k=1}^t \mathrm E_p[(X_k-X_{k-1})^2]\leq t.
		\end{equation}
		We conclude the proof by decomposing the expectation in the right-hand side of \eqref{eq:42} as           
		\begin{equation}
			\mathrm E_p\left[\frac{X_T^2}{T^{1+\varepsilon}}\right]\leq\sum_{i\geq 0}\frac{1}{2^{i(1+\varepsilon)}}\mathrm E_p\left[\max_{2^i\leq t\leq 2^{i+1}}X_t^2 \mathbf{1}_{\{2^i\leq T\leq 2^{i+1}\}}\right]\overset{\eqref{eq:152}}\leq \sum_{i\geq 0}\frac{2}{2^{i\varepsilon}}\leq \frac{2}{1-2^{-\varepsilon}}.
		\end{equation}
	\end{proof}
	
	\section{Sharp threshold results via hypercontractivity}
	\label{sec:sharp-thresh-seed}
	
	In this section, we establish a general sharp threshold result for connection events. Its proof involves the polynomial upper bound on the probability of $\piv 1 n$ from Section~\ref{sec:quantitative-uniqueness}, together with an abstract result from the theory of Boolean functions. Given a set $A\subset V$, we write $A^s$ for its $s$-thickening. Formally, $A^s$ is defined to be the set of vertices at distance at most $s$ from $A$.

	\begin{proposition}
		\label{prop:8}
		Let $G$ be a transitive graph of polynomial growth with $d\ge 1$. Let $\eta>0$. There is a constant $c=c(\eta,G)>0$ such that for every $p\in (\eta,1-\eta)$ and $\delta\in (0,1-\eta-p)$, the following holds. For every $C\subset V(G)$, $A,B\subset C$ and $s\ge 1$, we have
		\begin{equation}
			\label{eq:44}
			\mathrm P_p[A\lr{C} B] > \frac 1 {(cs)^{\delta/20}}\quad\implies\quad \mathrm P_{p+\delta}[A^s\lr{C^s} B^s] > 1-\frac1{(cs)^{\delta/20}}. 
		\end{equation}
	\end{proposition}
	
	\begin{remark}
		Early applications of sharp threshold techniques to percolation theory include \cite{russoapprox, bksnoise, MR2257131}. These techniques have now become a standard tool to study or prove sharpness of the phase transition for percolation processes. Applications usually involve a certain increasing event on a torus which is translation-invariant, and for which the result of \cite{KKL} ensures a sharp threshold phenomenon.
		
		Proposition~\ref{prop:8} above is inspired by the works \cite{MR3765895} and \cite{duminil2019upper} where abstract  sharp threshold results are used directly on $\mathbb Z^d$, without relying on translation-invariant events on a torus. Let us briefly expose the main difficulty we have to overcome in the present paper. An important idea in the  two works mentioned above is that  the event $\mathcal{E}=\{A \lr{C}B\}$ is ``geometrically similar'' to its translate $\tau\cdot \mathcal{E}$  by a small vector, in the sense that $A$, $B$ and $C$ are close to their translates $\tau\cdot A$, $\tau\cdot B$ and $\tau\cdot C$ respectively. This is true for $\mathbb Z^d$ (and more generally for Cayley graphs of \emph{abelian} groups), but not for general transitive graphs. 
	\end{remark}
	
	\begin{proof}[Proof of Proposition~\ref{prop:8}]
		In the proof below,  $c_1,c_2,\ldots$ denote positive constants that may depend on $\eta$ and $G$ but are independent of everything else.  Without loss of generality, we may (and will) assume that $A$, $B$ and $C$ are finite. By \cite{bourgain1992influence} and \cite[Corollary~1.2]{talagrand1994russo}, there exists $c_1=c_1(\eta)>0$ such that the following inequality holds for any increasing event $\mathcal A$ depending on finitely many edges:
		\begin{equation}\label{eq:45}
			\sum_{e\in E}\mathrm P_p[e \text{ is pivotal for }\mathcal A] \geq c_1\mathrm P_p[\mathcal A](1-\mathrm P_p[\mathcal A])f_p(\mathcal A),
		\end{equation}
		where
		\begin{equation}
			f_p(\mathcal A)=\log\left(\frac{1}{\max_{e\in E}\mathrm P_p[e\text{ is pivotal for } \mathcal A]}\right).
		\end{equation}
		Using Russo's formula together with~\eqref{eq:45},  we obtain
		\begin{equation}
			\label{eq:46}
			\frac{d}{dp}\log\left(\frac{\mathrm P_p[\mathcal A]}{1-\mathrm P_p[\mathcal A]}\right)\geq c_1 f_p(\mathcal A).
		\end{equation}
		
		To prove the desired inequality, we will define a sequence of nested and increasing events $\mathcal E_i$ such that $\{A\lr{C}B\}\subset\mathcal E_i\subset\{A^s\lr{C^s}B^s\}$. The first step will consist in finding a uniform lower bound for $f_p(\mathcal E_i)$. We will do so by proving that the probability of any edge to be pivotal is uniformly small.
		\begin{figure}
			\centering
			\includegraphics[scale=0.8]{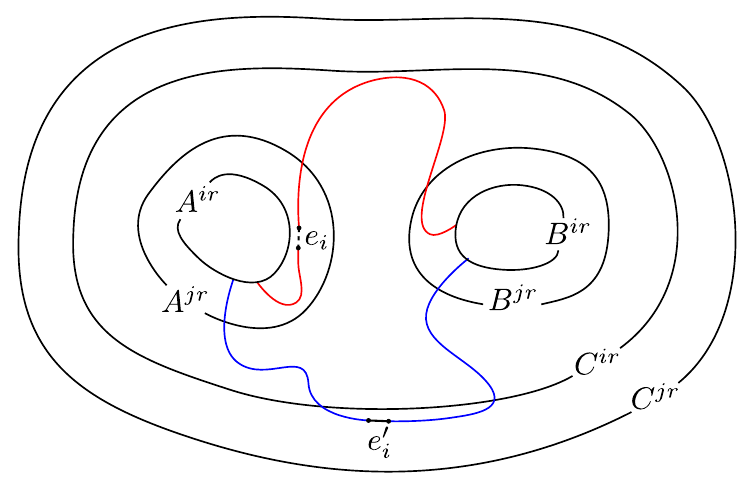}
			\label{fig:4}
			\caption{The red path illustrates the event $\mathcal F_i$ and the blue one $\mathcal F_i'$.}
		\end{figure}
		Set $r=2\lfloor\sqrt s/2\rfloor$. Without loss of generality, we may assume that $s\ge 4$ and look for a suitable $c$ in $(0,1/4)$, so that $r > 0$. For $i\in\{0,\ldots,r-1\}$, consider the set  $D_i=(A^{(i+1)r}\setminus A^{ir})\cup(B^{(i+1)r}\setminus B^{ir})$, the event $\mathcal E_i=\{A^{ir}\lr{C^{(i+1)r}}B^{ir}\}$, and
		\begin{equation}
			\label{eq:47}
			m_i=\max_{e\subset D_i}\mathrm P_p[e\text{ is a closed pivotal for } \mathcal E_i].
		\end{equation}
		For every $i$, let  $e_i$ be a fixed edge that maximises the probability in~\eqref{eq:47} and define  $\mathcal F_i$ to be the event $\{e_i \text{ is a closed pivotal for } \mathcal E_i\}$. We now fix $i,j\in\{0,\ldots,r-1\}$ with $i<j$. First, let us assume that the event $\mathcal F_i$ occurs. Then the edge $e_i$ is closed and its two endpoints are connected in $C^{(i+1)r}$ by disjoint open paths to $A^{ir}$ and $B^{ir}$, respectively. This implies that there is an open path connecting $A^{jr}$ to $B^{jr}$ in $C^{(i+1)r}$. Thus, $\mathcal F_i\subset \mathcal E_j$. On the other hand, if $\mathcal F_j$ occurs, the fact that $e_j$ is closed and pivotal implies that there is no open path connecting $A^{jr}$ to $B^{jr}$ in $C^{jr}$. Thus $\mathcal F_j\subset\mathcal E_j^c$. This leads us to a key observation, which is that $\mathcal F_i\cap \mathcal F_j=\emptyset$ for $i<j$. Therefore, we have
		\begin{equation}
			\label{eq:48}
			\sum_{i=0}^{r-1} m_i=\sum_{i=0}^{r-1}\mathrm P_p[\mathcal F_i]\leq 1.
		\end{equation}	
		Similarly,  for $i\in\{0,\ldots,r-1\}$, we can define the set $D'_i=C^{(i+1)r}\setminus C^{ir}$ and
		\begin{equation}
			\label{eq:49}
			m'_i=\max_{e'\subset D'_i}\mathrm P_p[e'\text{ is an open pivotal for } \mathcal E_i].
		\end{equation}
		For every $i$, let $e_i'$ be a fixed edge that maximises the probability in~\eqref{eq:49} and define $\mathcal F_i'$ to be the event $\{e_i' \text{ is an open pivotal for } \mathcal E_i\}$. Using the same reasoning as before, we can prove that for $i<j$, we have $\mathcal F_i'\subset \mathcal E_i$ and $\mathcal F_j'\subset \mathcal E_i^c$. Thus, $\mathcal F_i'\cap \mathcal F_j'=\emptyset$. This, together with~\eqref{eq:48}, implies that $\sum_{i=0}^{r-1} m_i+m_i'\leq 2$. From this, it follows that we can fix a set $I_p$ with $|I_p|\geq r/2$ and such that, for all $i\in I_p$, we have $\max\{m_i,m'_i\}\leq\frac{4}{r}$. Since the status of an edge is independent of the event that it is pivotal, we obtain
		\begin{equation}\label{eq:50}
			\max_{e\subset D_i\cup D'_i}\mathrm P_p[e\text{ is pivotal for } \mathcal E_i]\leq \max\left\{\frac{1}{p},\frac{1}{1-p}\right\}\cdot\frac{4}{r}\leq \frac{4}{\eta r}.
		\end{equation}
		Pick any $i\in I_p$. Let $e$ be an edge in $C^s$ such that $e\not\subset D_i\cup D_i'$. If $e\subset A^{ir}\cup B^{ir}\cup(C^s\setminus C^{(i+1)r})$ then $\mathrm P_p[e\text{ is pivotal for } \mathcal E_i]$ is equal to $0$. Otherwise, the $r$-neighbourhood of $e$ lies in $C^{(i+1)r}$. By Proposition~\ref{prop:3}, we get
		\begin{equation*}
			\mathrm P_p[e\text{ is closed pivotal for } \mathcal E_i]\leq \mathrm P_p[\piv{1}{r}]\leq \frac{1}{c_2 s^{1/5}}.
		\end{equation*}
		Combined with \eqref{eq:50}, this means that for all $i\in I_p$,
		\begin{equation*}
			\max_{e\subset C^s} \mathrm P_p[e\text{ is a pivotal for } \mathcal E_i]\leq \frac{1}{c_3 s^{1/5}}.
		\end{equation*}
		Using this, we obtain the following inequality:
		\begin{equation}\label{eq:51}
			\sum_{i=0}^{r-1} f_p(\mathcal E_i)\geq \sum_{i\in I_p}f_p(\mathcal E_i)\geq \frac{r}{2}\log(c_3 s^{1/5}).
		\end{equation}
		Now, we can use~\eqref{eq:46} for all the events $\mathcal E_i$, where $i\in\{0,\ldots,r-1\}$. Integrating between $p$ and $p+\delta$ and using~\eqref{eq:51}, we get
		\begin{equation*}
			\sum_{i=0}^{r-1}\log\left(\frac{\mathrm P_{p+\delta}[\mathcal E_i]}{1- \mathrm P_{p+\delta}[\mathcal E_i]}\cdot\frac{1-\mathrm P_p[\mathcal E_i]}{\mathrm P_p[\mathcal E_i]}\right)\geq \frac{\delta r}{2}\log(c_3 s^{1/5}).
		\end{equation*}
		Therefore, we can take some index $i_0$ such that
		\begin{equation*}
			\log\left(\frac{\mathrm P_{p+\delta}[\mathcal E_{i_0}]}{1- \mathrm P_{p+\delta}[\mathcal E_{i_0}]}\cdot\frac{1-\mathrm P_p[\mathcal E_{i_0}]}{\mathrm P_p[\mathcal E_{i_0}]}\right)\geq \frac{\delta}{2}\log(c_3 s^{1/5}),
		\end{equation*}
		which implies
		\begin{equation*}
			\frac{1}{(1-\mathrm P_{p+\delta}[\mathcal E_{i_0}])\mathrm P_p[\mathcal E_{i_0}]}\geq \frac{\mathrm P_{p+\delta}[\mathcal E_{i_0}]}{1- \mathrm P_{p+\delta}[\mathcal E_{i_0}]}\cdot\frac{1-\mathrm P_p[\mathcal E_{i_0}]}{\mathrm P_p[\mathcal E_{i_0}]}\geq (c_4 s)^{\delta/10}.
		\end{equation*}
		Since $\{A\lr{C}B\}\subset\mathcal E_{i_0}\subset \{A^s\lr{C^s}B^s\}$, we have
		\begin{equation*}
			\mathrm P_{p+\delta}[A^s\lr{C^s}B^s]\geq \mathrm P_{p+\delta}[\mathcal E_{i_0}]\geq 1-\frac{1}{\mathrm P_p[\mathcal E_{i_0}](c_4 s)^{\delta/10}}\geq 1-\frac{1}{\mathrm P_p[ A\lr{C}B](c_4s)^{\delta/10}}.
		\end{equation*}
		Proposition~\ref{prop:8} follows.
	\end{proof}
	
	\begin{corollary}
		\label{cor:1}
		Let $G$ be a transitive graph of polynomial growth with $d\ge 2$. Let $p>p_c$. There exist $\delta>0$ and $c>0$ such that, for all $s\geq 1$,  we have
		\begin{equation}
			\mathrm P_p[B_s\lr{}\infty]\geq 1-\frac{1}{(cs)^{\delta}}.\label{eq:43}
		\end{equation}
		
	\end{corollary}
	
	\begin{proof}
		Let $\delta>0$ be such that $p>p_c+20\delta$.    We know that $\mathrm P_{p-20\delta}[o\lr{}\infty]>0$.  By Proposition~\ref{prop:8} applied to $A=\{o\}$, $B=\partial B_n$ and $C=V(G)$,   we have $\mathrm P_{p}[B_s\lr{}\partial B_{n-s}]>1-(cs)^{-\delta}$, where $c$ is a positive constant independent of  $s$ and $n$. The proof follows by letting $n$ tend to infinity.
	\end{proof}

	\section{A priori bound  on the uniqueness zone}
	\label{sec:bounds-uniq-zone}
	
	The goal of this section is to prove Proposition~\ref{prop:4} below.
	
	\begin{proposition}\label{prop:4}
		Let $G$ be a transitive graph of polynomial growth. Let $p>p_c$. There are some $\chi \in (0,1)$ and $c>0$ such that the following holds: for every $q\in [p,1]$, for every $n\ge 1$, we have
		\begin{equation}
			\label{eq:52}
			\mathrm P_q[\piv{s(n)}{n}] \le cn^{-1/4},
		\end{equation}
		where $s(n)=\exp((\log n)^\chi)$.
	\end{proposition}
	
	We say that the annulus of inner radius $s(n)$ and outer radius $n$ is a \emph{uniqueness zone} because Proposition~\ref{prop:4} tells us that $\mathrm P_q[U(s(n),n)]=1-\mathrm P_q[\piv{s(n)}{n}]$ converges to 1 --- actually at a controlled speed. The size of this uniqueness zone will in fact determine the region in which we are able to glue clusters. This will lead to the important notion of seeds, which will be instrumental in Section~\ref{sec:properties-seeds}.
	
	In order to prove Proposition~\ref{prop:4}, we will use a bootstrap argument, where the iterations consist in alternating uses of Lemma~\ref{lem:12} and Lemma~\ref{lem:13}. 
	
	The first lemma, directly adapted from \cite[Lemma 7.2]{cerf2015}, provides an upper bound on the probability of $\mathsf{Piv}(r,n)$ provided with some uniform lower bound on the two-point function restricted to a ball. Conversely, the second lemma deduces some lower bound on the two-point function in a ball, provided with some upper bounds on the probability of $\mathsf{Piv}(r,n)$. 
	
	From there, the proof goes as follows. Proposition~\ref{prop:3} and Lemma~\ref{lem:12} provide us with a good upper bound on the probability of  $\mathsf{Piv}(r,n)$, for some fixed and large $(r,n)$. Then Lemma~\ref{lem:13}  provides us with a good lower bound on the two-point function at the scale above. Plugging this estimate in Lemma~\ref{lem:12} yields a good upper bound on the probability of  $\mathsf{Piv}(r,n)$ at this larger scale.  Repeated inductively, this procedure leads to the quantitative estimate (\ref{eq:52}).
	
	\begin{notation}
		In this section, for $r\ge 0$, we set $S_r$ to be the sphere of centre $o$ and radius $r$, where $o$ is the root of $G$ that we use in the notation $B_m=B_m(o)$.
	\end{notation}
	
	\begin{lemma}
		\label{lem:12}
		For all $p\in [0,1]$ and $1<r\le m\leq n/2$, we have
		
		\begin{equation}
			\label{eq:53}
			\mathrm P_p[\piv{r}{n}] \leq \mathrm P_p[\piv{1}{n/2}]\cdot\frac{|S_r|^2|B_m|}{\min_{a,b\in S_r}\mathrm P_p[a\lr{B_{2m}}b]}.
		\end{equation}
	\end{lemma}

	\begin{proof} Let us fix $u$, $m$ and $n$ as in the statement of Lemma~\ref{lem:12}. 
		Given a vertex $a\in B_n$, we denote by $\mathcal{C}_a$ the open cluster of $a$ in $B_n$.
		Given a second vertex $b\in B_n$, define $\mathfrak{C}_a^b$ to be the (deterministic) family of all connected subsets of $B_n$ that contain $a$, do not contain $b$, and intersect $\partial B_n$. By the union bound, we have
		\begin{align*}
			\mathrm P_p[\piv{r}{n}] & \leq \sum_{a,b\in S_r}\mathrm P_p[a\lr{}\partial B_n,~ b\lr{B_n\setminus \mathcal{C}_a}\partial B_n]\\
			& = \sum_{a,b\in S_r}\sum_{C\in\mathfrak{C}_a^b}\mathrm P_p[b\lr{B_n\setminus C}\partial B_n]\mathrm P_p[\mathcal{C}_a=C],
		\end{align*}
		where we used that the two last events are independent, since they depend on disjoint sets of edges. Let us fix some $C\in\mathfrak{C}_a^b$.
		The Harris--FKG Inequality yields
		\begin{equation*}
			\mathrm P_p[b\lr{B_n\setminus C}\partial B_n]\leq\frac{\mathrm P_p[b\lr{B_n\setminus C}\partial B_n,~b\lr{B_m}\partial^{\mathrm{out}} C]}{\mathrm P_p[b\lr{B_m}\partial^{\mathrm{out}}C] },
		\end{equation*}
		where $\partial^{\mathrm{out}} C$ is the (deterministic) set of vertices adjacent to $C$ but not belonging to it. 
		Observe that for every $C\in\mathfrak{C}_a^b$, as $C$ contains $a$ but not $b$, we have $\{a\lr{B_m}b\} \subset \{b\lr{B_m}\partial^{\mathrm{out}}C\}$, hence $\mathrm P_p[b\lr{B_m}\partial^{\mathrm{out}}C]\ge \mathrm P_p[a\lr{B_{2m}}b]$. Using the bound above and  independence again, we get
		\begin{align}
			\label{eq:62}
			&\sum_{C\in\mathfrak{C}_a^b}\mathrm P_p[b\lr{B_n\setminus C}\partial B_n]\mathrm P_p[\mathcal{C}_a=C]\\
			&\qquad \le \frac1{\mathrm P_p[a\lr{B_{2m}}b]} \sum_{C\in\mathfrak{C}_a^b} \mathrm P_p[\mathcal C_a=C,~ b\lr{B_n\setminus C}\partial B_n,~b\lr{B_m}\partial^{\mathrm{out}} C]\\
			&\qquad= \frac1{\mathrm P_p[a\lr{B_{2m}}b]}\mathrm P_p[a\lr{} \partial B_n,~b\lr{B_n\setminus \mathcal C_a}\partial B_n,~b\lr{B_m}\partial^{\mathrm{out}} \mathcal C_a].
		\end{align}
		\begin{figure}
			\centering
			\includegraphics[scale=0.8]{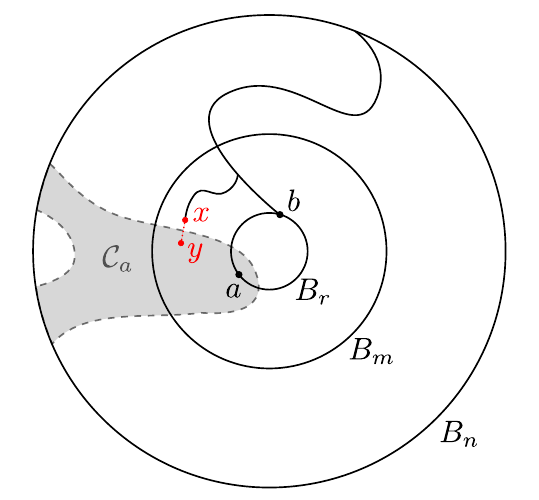}
			\caption{In the complement of $\mathcal C_a$, the existence of an open path inside $B_m$ joining the cluster of $b$ with an edge $\{x,y\}\in\partial\mathcal C_a$ implies that $x$ and $y$ are in disjoint clusters of diameter at least $n/2$.}
			\label{fig:5}
		\end{figure}%
		When the event in the last equation occurs,  we can find a closed edge $e=\{x,y\}\in\partial\mathcal C_a$ such that $x\lr{B_n\setminus \mathcal C_a}\partial B_n$ and $y\lr{\mathcal C_a}\partial B_n$. Hence, when this event occurs, so does $\pivx{x}{1}{n-m}$: see Figure~\ref{fig:5}.
		Therefore, by the union bound, we have
		\begin{equation*}
			\mathrm P_p[a\lr{} \partial B_n,~ b\lr{B_n\setminus \mathcal C_a}\partial B_n,~b\lr{B_m}\partial^{\mathrm{out}} \mathcal C_a]  \leq |B_m|\mathrm P_p[\piv{1}{n/2}].
		\end{equation*} 
		Putting all the equations together, we get
		\begin{align*}
			\mathrm P_p[\piv{r}{n}] &\leq \sum_{a,b\in S_r}\mathrm P_p[\piv{1}{n/2}]\frac{|B_m|}{\mathrm P_p[a\lr{B_{2m}}b]}\\
			&\leq \mathrm P_p[\piv{1}{n/2}]\cdot\frac{|S_r|^2|B_m|}{\min_{a,b\in S_r}\mathrm P_p[a\lr{B_{2m}}b]}.
		\end{align*}
	\end{proof}
	
	\begin{lemma}\label{lem:13}
		Let $p>p_c$. There exist $\delta>0$ and $r_0\ge1$ such that for every $r\ge r_0$ and every $m\ge r^{1+\delta}$, we have
		\begin{equation*}
			\mathrm P_p[\piv{r}{m}]\le \frac{\delta}{r^\delta}\quad \implies \quad\forall a,b \in  B_{r^{1+\delta}}\quad \mathrm P_p[a\lr{B_{2m}}b]\geq \delta.
		\end{equation*} 
	\end{lemma}
	
	\begin{proof}
		Take $c$ and $\delta$ to satisfy the conclusion of Corollary~\ref{cor:1}. We may further assume that $\delta$ is small enough, so that the following inequality holds
		\begin{equation}
			\theta(p)^2 4^{-8(2/c)^{\delta}}\ge 9\delta.
			\label{eq:55}
		\end{equation}
		Let $a,b\in B_{r^{1+\delta}}$. We can find vertices $x_0=a,x_1,\dots,x_k=b$ in $B_m$ such that $k\leq 2\frac{\lfloor r^{1+\delta}\rfloor}{\lfloor r/2\rfloor}\le 8r^\delta$ and $\forall i<k,~d(x_i,x_{i+1})\le r/2$.
		Assume that $a\lr{}\infty$, $b\lr{}\infty$, and $B_{r/2}(x_i)\lr{}\infty$ for all $i \in \{ 1,\ldots, k-1\}$. Then there are two possibilities: either $a\lr{B_{2m}}b$, or there is some $i \in \{1,\ldots,k\}$ such that $\pivx{x_i} r m$ occurs. See Figure~\ref{fig:6}. Therefore, by the union bound, we get
		\begin{equation*}
			\mathrm P_p[a\lr{}\infty,~b\lr{}\infty,~B_{r/2}(x_i)\lr{}\infty \text{ for all } i]\leq \mathrm P_p[a\lr{B_{2m}}b] + \sum_{i=1}^{k} \mathrm P_p[\pivx{x_i} r m].
		\end{equation*}
		\begin{figure}
			\centering
			\includegraphics[scale=0.8]{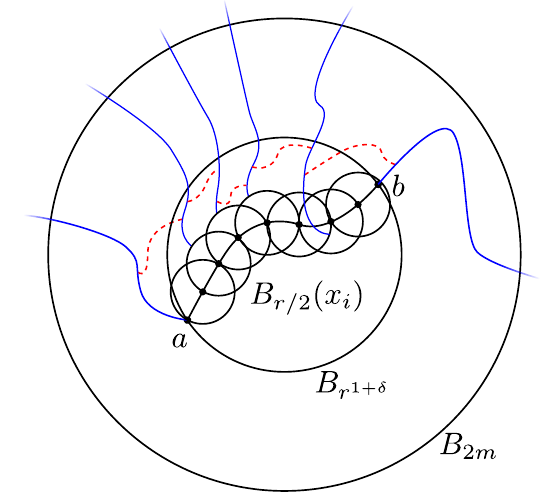}
			\caption{The event $\{a\lr{}\infty,~b\lr{}\infty,~B_{r/2}(x_i)\lr{}\infty\text{ for all }i\}$ is depicted in blue. Notice that if there is no $i$ such that $\pivx{x_i} r m$ holds, then we have the red dotted paths in the figure, thus creating an open path between $a$ and $b$ within $B_{2m}$.}
			\label{fig:6}
		\end{figure}%
		By our choice of $c$ and $\delta$, for every $i$, we have $\mathrm P_p[B_{r/2}(x_i)\lr{} \infty]\geq 1-\frac{1}{(cr/2)^\delta}$. 
		By the Harris--FKG Inequality, by transitivity of $G$, and because $\mathrm P_p$ is invariant under graph automorphisms, we have
		\begin{align*}
			\mathrm P_p[a\lr{B_{2m}}b] &\geq \mathrm P_p[a\lr{}\infty]\mathrm P_p[b\lr{}\infty](1-(cr/2)^{-\delta})^{8r^\delta}-k\mathrm P_p[\piv{r}{m}]\\
			&\geq \theta(p)^2 (1-(cr/2)^{-\delta})^{8r^\delta}-8r^\delta\mathrm P_p[\piv{r}{m}].
		\end{align*}
		Assuming that $r \ge r_0:= 2^{1+1/\delta}/c$ and that $\mathrm P_p[\piv{r}{m}] \le \frac{\delta}{r^\delta}$, we get
		\begin{equation*}
			\mathrm P_p[a\lr{B_{2m}}b] \ge \theta(p)^2 4^{-8(2/c)^{\delta}}-8\delta.
		\end{equation*}
		By \eqref{eq:55}, the proof is complete.
	\end{proof}
	
	We are now ready to prove Proposition~\ref{prop:4}, following the strategy described at the beginning of the section.
	
	\begin{proof}[Proof of Proposition~\ref{prop:4}]  Let $p>p_c$. Let $\delta\in (0,1/4)$ and $r_0\ge2$ be  as in Lemma~\ref{lem:13}. By Lemma~\ref{lem:12}, Proposition~\ref{prop:3} and the polynomial growth of $G$, we can fix $c_1=c_1(G,p)\ge 1$ such that for every $1< r\le m\le n/2$, we have
		\begin{equation}
			\label{eq:56}
			\forall q\ge p\quad  \mathrm P_q[\piv r n]\le  \frac{c_1r^{2d}m^dn^{-1/3}}{\min_{a,b\in S_r}\mathrm P_p[a\lr{B_{2m}}b]}.
		\end{equation}
		Applying the inequality above to $r=m=r_0$  allows us to choose  $n_0\ge \max(2(c_1/\delta^2)^{1/d},r_0^{1+\delta})$ such that
		\begin{equation}
			\label{eq:57}
			\forall q\ge p\quad \mathrm P_q[\piv {r_0} {n_0}]\le\frac\delta {r_0^{1/4}}. 
		\end{equation}
		Consider the sequences $(r_k)$ and $(n_k)$ defined by $r_{k+1}=r_k^{1+\delta}$, $n_{k+1}=n_k^{50d}$ for every $k\ge0$. By induction, we  will prove that for every $k\ge 0$,
		\begin{equation}
			\label{eq:58}
			\forall q\ge p\quad  \mathrm P_q[\piv {r_k} {n_k}]\le\frac\delta{n_k^{1/4}}. 
		\end{equation}
		This will conclude the proof of \eqref{eq:52} for any $\chi<\log(1+\delta)/\log(50d)$ along the sequence $(n_k)$. The statement for general $n$ is then obtained by initiating the sequence not necessarily at $n_0$ but at any value in the compact interval $[n_0,n_1]$.
		
		Let $k\ge 0$ and assume that \eqref{eq:58} holds. Since $n_k^{1/4}\ge r_k^{\delta}$, the quantity $\mathrm P_p[\piv {r_k} {n_k}]$ is at most $\delta/r_k^{\delta}$. Applying Lemma~\ref{lem:13} to $r=r_k$ and $m=n_k \ge r_k^{1+\delta}$  gives
		\begin{equation}
			\label{eq:59}
			\forall a,b \in  B_{r_{k+1}} \quad \mathrm P_p[a\lr{B_{2n_k}}b]\geq \delta.
		\end{equation}
		Equation~\eqref{eq:56} applied to $r=r_{k+1}$, $m=2n_k$ and $n=n_{k+1}$ finally gives
		\begin{equation}
			\forall q\ge p\quad \mathrm P_q[\piv{r_{k+1}}{n_{k+1}}]\leq \frac{c_1 r_{k+1}^{2d}(2n_k)^dn_{k+1}^{-1/3}}{\delta}\le \frac{\delta}{n_{k+1}^{1/4}},
		\end{equation}
		where we used $r_{k+1}=r_k^{1+\delta}\le n_k$ and $n_k\ge n_0\ge 2(c_1/\delta^2)^{1/d}$ in the last inequality.
	\end{proof}

	The bound on the uniqueness zone gives us the following  bound on the corridor function, which will be an important ingredient in the final proof of Proposition~\ref{prop:1}.  
	\begin{corollary}
		\label{cor:4}Let $G$ be a graph of polynomial growth, let $p>p_c$. Let $\{s(n)\}_{n\ge1}$
		be a sequence as in Proposition~\ref{prop:4}. Then, for every $n$  large enough, we have   
		\begin{equation}
			\label{eq:61}
			\kappa_p(s(n),n)\ge \theta(p)^2/2.
		\end{equation}
	\end{corollary}
	
	\vspace{-15pt}
	
	\begin{proof}
		Notice that  $x\in B_{s(n)}$ is connected  to the origin by an open path in  $B_n$ if the following conditions hold:
		\begin{itemize}
			\item   both $o$ and $x$ are connected   to $\partial B_{n}$ and
			
			\item there is a unique cluster crossing from $s(n)$ to $n$.
		\end{itemize}
		By the Harris--FKG Inequality and the union bound, we have
		\begin{equation}
			\label{eq:63}
			\forall x\in B_{s(n)} \quad\mathrm P_{p}[o\lr{B_n} x]\ge \theta(p)^2-\mathrm P_p[\piv{s(n)}{n}]>\theta(p)^2/2,
		\end{equation}
		where the last inequality holds if $n$ is large enough, by Proposition~\ref{prop:4}. Since any corridor of length $s(n)$ and thickness $n$ contains the ball $B_n$, the equation above implies $\kappa_{p}(s(n),n)>\theta(p)^2/2$ for every $n$ large enough.  
	\end{proof}

	\begin{corollary}\label{cor:3}
		Let $G$ be a transitive graph of polynomial growth.  For all $p>p_c$, $\delta\in(0,1-p)$, and $k\ge 1$, there exists $n_0\ge1$ such that for every $n\ge m\ge n_0$,
		\begin{equation}
			\kappa_p(2m,n)> \displaystyle\frac1{\log^k m} \implies \kappa_{p+\delta}(2m \log^k m ,2n)\geq \theta(p)^2/2.\label{eq:60}
		\end{equation}
	\end{corollary}

	The proof of Corollary~\ref{cor:3} is analogous to the proof of Lemma~\ref{lem:13}. However, here, we need to apply Proposition~\ref{prop:4} to get a quantitative estimate. 
	
	\begin{proof} 
		\sloppy Let $p$, $\delta$ and $k$ be as stated in Corollary~\ref{cor:3}, and assume that $\kappa_p(2m,n)\geq {\log^{-k} m}$. Let $\gamma$ be a path of length at most $2m \log^k m$  starting at some vertex $x$ and ending at some vertex $y$. Consider the sequence $\{s(j)\}$,   where $s(j)$ is defined as in Proposition~\ref{prop:4} at the parameter $p$. If the distance between $x$ and $y$ is smaller than $s(m)$, Corollary~\ref{cor:4} directly concludes that $x$ and $y$ are connected in $B_{n}$ with probability at least $\theta(p)^2/2$ (because $s(m)\le s(n)$). We can therefore focus on the case when $x$ and $y$ are at distance at least $s(m)$. In that case, we can find $M\le 2\log^k m$  points $x_i\in \gamma$ such that $x_1=x$, $x_M=y$ and  $s(m)/2\le d(x_i,x_{i+1})\le 2 m$.  Let $C$ be the corridor of thickness $2n$ around $\gamma$.  Set $t=s(\lfloor s(m)/3\rfloor)$. By our assumption on $\kappa_p(2m,n)$ and by Proposition~\ref{prop:8}, we have, for $m$ large enough,
		\begin{equation}
			\mathrm P_{p+\delta}\left[B_t(x_{i})\lr{C} B_t(x_{i+1})\right]\geq 1 -\frac{1}{(ct)^{\delta/20}}\geq 1-\frac{1}{\log^{k+1} m}
		\end{equation}
		for every $i$.
		If we assume that $x\lr{}\infty$, $y\lr{}\infty$ and $B_t(x_{i})\lr{C} B_t(x_{i+1})$ for every $i$, then either $x\lr{C} y$ or there exists $i$ such that $\pivx{x_i}{t}{s(m)/3}$ happens. The union bound thus yields
		\begin{equation}
			\mathrm P_{p+\delta}[x\lr{C} y]\ge \theta(p+\delta)^2\left(1-\frac{M-1}{\log^{k+1} m}\right)-\sum_{i=1}^M\mathrm P_{p+\delta} [\pivx{x_i}{t}{s(m)/3}].
		\end{equation}
		We conclude by observing that the sum converges to zero by Proposition~\ref{prop:4} and that $M\le 2\log^k(n)$.
	\end{proof}

	\section{Sharp threshold results via Hamming distance}
	\label{sec:sharp-thresh-hamming}
	
	This section is devoted to the proof of the following proposition.
	
	\begin{proposition}\label{prop:5}
		Let $G$ be a transitive graph of polynomial growth. Let $p>p_c$. There exists $n_0$ such that for every $n\ge m\ge n_0$,
		\begin{equation}
			\label{eq:65}
			\kappa_{p}(m,2n)\leq\theta(p)^2/2\implies  \mathrm P_p[B_m\lr{} \partial B_n]\ge 1-e^{-\log^3 n}.
		\end{equation}
	\end{proposition}

	The statement above may seem slightly counter-intuitive at a first look: we use some negative information (the function $\kappa_p$ is small) to obtain a large connection probability. Let us sketch the proof, which will be detailed in Section~\ref{sec:proof-prop}, and explain how this negative information can be used to our advantage.
	
	Due to the uniqueness of the infinite cluster and the Harris--FKG Inequality, the two-point function in the whole graph is uniformly lower bounded by $\theta(p)^2$. Assume that when we restrict the connections to a box, we get something substantially smaller, in the sense that
	\begin{equation}
		\label{eq:66}
		\kappa_p(m,2n) \le \theta(p)^2/2. 
	\end{equation}
	By reducing the parameter from $p$ to $p-\delta>p_c$ and using the sharp threshold results of Section~\ref{sec:sharp-threshold-via} below, we strengthen this bound as follows. For $k=\log^{10}(n)$, we show that two points $x$ and $y$ at distance~$\simeq m/k$ of each other always satisfy
	\begin{equation}
		\label{eq:67}
		\mathrm P_{p-\delta}[x\lr{B_{2n}} y]\le \frac1{k}.
	\end{equation}
	Now, consider $k$ points in $B_m$ at distance at least $m/k$ of each other. On the one hand, on average, a proportion at least $\theta(p-\delta)$ of them are connected to the boundary of $B_{2n}$. On the other hand, the estimate \eqref{eq:67} implies that all these points typically belong to different clusters of $B_{2n}$, which forces the paths that connect them to $\partial B_n$ to be disjoint. It is at this point that our ``negative'' assumption on $\kappa_p$, combined with the ``positive'' assumption that $p>p_c$, yields a ``positive'' statement regarding connectivity of our percolation process: the expected number of disjoint paths from $B_m$ to $\partial B_{2n}$ is $\gtrsim  k$.
	From this estimate, a well-known differential inequality involving the Hamming distance on the hypercube guarantees that, at parameter $p=(p-\delta)+\delta$, we have $\mathrm P_p[B_m\lr{} \partial B_{2n}]\geq 1-e^{-\delta k}$. This concludes the proof.
	
	In Section~\ref{sec:sharp-threshold-via}, we present this new Hamming distance argument in a more general framework, since we believe it can have further applications.
	
	\subsection{Connectivity bounds via Hamming distance}
	\label{sec:sharp-threshold-via}
	
	In the current Section~\ref{sec:sharp-threshold-via}, contrary to elsewhere in the paper, $G$ denotes any finite connected graph with vertex set $V(G)$ and edge set $E(G)$. Besides, $p$ is arbitrary in $[0,1]$ and $\mathrm P_p$ stands for the Bernoulli bond  percolation measure of parameter $p$ on $\{0,1\}^{E(G)}$.
	
	\begin{proposition} 
		\label{prop:6} Let $A,B\subset V(G) $. Let $p\in[0,1]$ and $\theta>0$. Assume that
		\begin{equation}
			\label{eq:68}
			\min_{x\in A} \mathrm P_p[x\lr{}B]\geq\theta\geq 2|A|\cdot\max_{\substack{x,y\in A\\x\neq y}} \mathrm P_p[x\lr{}y].
		\end{equation}
		Then, for every $\delta\in(0,1-p]$, we have
		\begin{equation}
			\label{eq:69}
			\mathrm P_{p+\delta}[A \lr{} B]\ge 1-e^{-2\delta\theta|A|}.
		\end{equation}
	\end{proposition}
	
	\begin{remark} The proposition above   also applies to FK-percolation measures with cluster weight $q\ge1$, and more generally to measures for which the ``exponential steepness'' property of \cite[Section 2.5]{MR2243761} holds. 
	\end{remark}
	Before proving Proposition~\ref{prop:6}, we recall a general inequality for  monotone events.  The \defini{Hamming distance} from a configuration $\omega$ to an event $\mathcal A$ is defined by 
	\begin{equation}
		H_{\mathcal A}(\omega)=\inf\{H(\omega,\omega'):\omega'\in \mathcal A\},
	\end{equation}
	where $H(\omega,\omega')=\sum_{e\in E(G)}|\omega_e-\omega'_e|$ is the usual Hamming distance on the hypercube $\{0,1\}^{E(G)}$. When $\mathcal A$ is decreasing, one can interpret $H_{\mathcal A}(\omega)$ as the minimal number of edges in $\omega$ that need to be closed for the event  $\mathcal A$ to occur. Furthermore,  the Hamming distance provides exponential bounds on the variation  of $\mathrm P_p[\mathcal A]$ relative to $p$ (see \cite[Theorem 2.53]{MR2243761}): for every decreasing event $\mathcal A$ and every $p\in (0,1)$, we have
	\begin{equation}
		\label{eq:70}
		\frac{d}{dp}\left(- \log \mathrm P_p[\mathcal{A}]\right)\ge 4\mathrm E_p[H_\mathcal{A}]. 
	\end{equation}
	By integrating the equation above and using that $H_\mathcal{A}$ is increasing in $\omega$, we get that for every $p\in [0,1]$ and every $\delta\in[0,1-p]$,
	\begin{equation}
		\mathrm P_{p+\delta}[\mathcal{A}]\le e^{-4\delta\mathrm E_p[H_\mathcal{A}]}	\mathrm P_{p}[\mathcal{A}]\le e^{-4\delta\mathrm E_p[H_\mathcal{A}]}.\label{eq:71}
	\end{equation}

	\begin{proof}[Proof of Proposition~\ref{prop:6}]
		Let $A,B\subset V(G)$, let $p\in [0,1]$ and $\delta\in [0,1-p]$.  Applying  \eqref{eq:71} to the decreasing event $\mathcal A=\{A\nlr{} B\}$, we get
		\begin{equation}
			\label{eq:72}
			\mathrm P_{p+\delta}[A\lr{}B]\ge 1-e^{-4\delta\mathrm E_p[H_{\mathcal A}]}.
		\end{equation}
		The Hamming distance $H_\mathcal{A}$ is clearly at least\footnote{Actually, if $A$ and $B$ are disjoint, Menger's Theorem states that this is an equality \cite[Corollary 3.3.5]{diestel2017graph} --- but we only need the easy inequality.} the maximal number of disjoint open paths from $A$ to $B$.
		In particular, $H_\mathcal{A}$ is larger than the number of disjoint clusters intersecting both $A$ and $B$.  By inclusion-exclusion, this number of crossing clusters can be lower bounded by
		\begin{equation}
			\label{eq:73}
			\sum_{x\in A}\mathbf 1[x\lr{} B]-\sum_{\substack{x,y\in A\\x\neq y}} \mathbf 1[x\lr{} y]. 
		\end{equation}
		Fixing $\theta$ as in \eqref{eq:68} and taking the expectation above, we get
		\begin{align}
			\label{eq:74}
			\mathrm E_p[H_\mathcal{A}]&\ge  \sum_{x\in A}\mathrm P_p[x\lr{} B]-\sum_{\substack{x,y\in A\\x\neq y}}\mathrm P_p[x\lr{}y] \ge \theta|A|/2.
		\end{align}
		Plugging the estimate above in \eqref{eq:72} completes the proof.  
	\end{proof}
	
	\subsection{Seeds and two-seed function}
	\label{sec:properties-seeds}
	
	The notation $G$ now recovers its initial meaning and denotes once again a transitive graph of polynomial growth with $d\ge 2$.
	In this section, we fix $p>p_c$. Let $\chi=\chi(p) \in (0,1)$ be as in Proposition~\ref{prop:4}. For every $n$, we define
	\begin{equation}
		\label{eq:75}
		\sigma(n)=\exp(\log^{\chi^3} n) \quad\text{and} \quad  t(n)=\exp(\log^{\chi^2} n).
	\end{equation}
	For every positive integer $n$ and every vertex $x$, set
	\begin{equation}
		\label{eq:76}
		S_n=B_{\sigma(n)}\quad\text{and}\quad S_n(x)=B_{\sigma(n)}(x).
	\end{equation}
	Following the terminology introduced by Grimmett and Marstrand \cite{grimmett1990supercritical}, we call $S_n(x)$ the \defini{seed} of $x$. An important property of seeds is that they are connected to infinity with high probability.  For every $n\ge1$, define
	\begin{equation}
		\varepsilon_n=\frac{1}{\log^{10} n}.
	\end{equation}
	Since  $\sigma(n)$ is asymptotically  larger than arbitrarily large powers of $\log n$, Corollary~\ref{cor:1} implies that, for   every $n$ large enough, we have
	\begin{equation}
		\label{eq:77}
		\mathrm P_p[S_n\lr{}\infty]\ge 1-\tfrac1{100}\varepsilon_n. 
	\end{equation}
	
	Another important property of seeds is that they can be used to ``glue'' clusters.  Intuitively, if two large clusters of $B_n$ touch a certain seed $S_n(x)$, then the local uniqueness around $S_n(x)$ implies that they must be connected within $B_n$ --- which means that the two clusters are equal. Formally, we will use the following  upper bound on the probability that two distinct clusters reach a fixed seed, provided by Proposition~\ref{prop:4}. For every $n$ large enough we have
	\begin{equation}
		\label{eq:78}
		\mathrm P_p[\mathsf{Piv}(\sigma(n),t(n))]\le\tfrac1{100}\varepsilon_n.
	\end{equation}
	This follows from Proposition~\ref{prop:4}  together with the  observations that $\sigma(n)=\exp(\log^\chi(t(n))$ and $c\cdot t(n)^{-1/4}\le \tfrac1{100}\varepsilon_n$ for $n$ large enough. Similarly, we also have
	\begin{equation}
		\label{eq:79}
		\mathrm P_p[\mathsf{Piv}(3t(n),n/2)]\le \tfrac1{100} \varepsilon_n,
	\end{equation}
	using that $3t(n)\le\exp(\log^\chi(n/2))$ for $n$ large.
	
	\begin{remark}
		In other works, ``seed'' may refer to other constructions that guarantee a high probability of connection to infinity and that can be used to glue clusters. In \cite{grimmett1990supercritical}, this is done by defining a seed to be a large fully open box. In \cite{MR3630298}, one takes advantage of the fact that if an exploration reaches some vertex, then there is a long open path leading to it --- in particular, see \cite[Lemma~3.6]{MR3630298}.
	\end{remark}
	
	We define the \defini{two-seed function} by 
	\begin{equation}
		\tau_{p,n}(x,y)=\mathrm P_p[S_n(x)\lr{B_{n}} S_n(y)]
	\end{equation}
	for every $x,y\in B_n$. Notice that for $n$ large enough, we have $\tau_{p,n}(x,y)=1$ whenever $x$ and $y$ are neighbours in $B_n$. As we will prove in Lemma~\ref{lem:15}, the two-seed function shares some features with the standard two-point function. One main advantage of replacing points by seeds is that we can make use of the following  sharp threshold phenomenon.
	
	\begin{lemma}\label{lem:14}
		Let $\delta\in (0,p-p_c)$. For every $n$ large enough, for every $x,y\in B_n$, we have
		\begin{equation}
			\label{eq:82}
			\mathrm P_{p-\delta}[x\lr{B_n}y]\ge \varepsilon_{2n}\implies  \tau_{p,2n}(x,y)\ge 1-\varepsilon_{2n}.
		\end{equation}
	\end{lemma}
	\begin{proof}
		It is a direct consequence of Proposition~\ref{prop:8}, together with the observation that for all $\delta>0$, $c>0$, for every $n$ large enough, we have
		$(c\sigma(2n))^{\delta/20}\geq \varepsilon_{2n}$.
	\end{proof}

	\begin{lemma}
		\label{lem:15}
		For every $n$ large enough, for all $x,y,z\in B_{n/2}$, we have
		\begin{align}
			\label{eq:83}
			& \tau_{p,n}(x,z)\ge  \tau_{p,n}(x,y) \tau_{p,n}(y,z)-\varepsilon_n,\\ 
			\label{eq:84}
			&\mathrm P_p[x\lr{B_n} y]\ge \theta(p)^2\cdot   \tau_{p,n}(x,y)-\varepsilon_n.	
		\end{align}
	\end{lemma}
	
	\begin{proof} We begin with the proof of \eqref{eq:83}. Without loss of generality, we may and will assume that $d(x,y)\ge d(y,z)$. We distinguish three different cases. For readability, we drop the indices $p$ and $n$ from the notation $\tau_{p,n}$ in this proof.
		
		\begin{description}
			\item[Case 1:] $d(x,y)\ge d(y,z)\ge 2t(n)$.
			
			By the Harris--FKG Inequality, we have    
			\begin{equation}
				\tau(x,y)\tau(y,z)\le \mathrm P_p[S_n(x)\lr{B_n} S_n(y),\,S_n(y)\lr{B_n} S_n(z) ].\label{eq:28}
			\end{equation}
			When there exist a cluster connecting $S_n(x)$ to $S_n(y)$ and  a cluster  connecting $S_n(y)$ to $S_n(z)$, then either these two clusters are connected together within $B_{t(n)}(y)$, or we observe two disjoint clusters crossing from $S_n(y)$ to the boundary of $B_{t(n)}(y)$.  Therefore, by the union bound
			\begin{equation}
				\label{eq:85}
				\tau(x,y)\tau(y,z)\le \tau(x,z)+\mathrm P_p[\mathsf{Piv}(\sigma(n), t(n))]\overset{\eqref{eq:78}}\le \tau(x,z)+\varepsilon_n.
			\end{equation}
			
			\item[Case 2:] $d(x,y)\ge 2t(n)> d(y,z)$.
			
			The probability that the seeds $S_n(z)$ and $S_n(y)$ are both connected to infinity is larger than $1-\tfrac1{50}\varepsilon_n$. Therefore, we have  
			\begin{equation}
				\label{eq:86}
				\tau(x,y)-\tfrac1{50}\varepsilon_n\le\mathrm P_p[S_n(x)\lr{B_n}S_n(y), S_n(y)\lr{} \partial B_n, S_n(z)\lr{} \partial B_n].
			\end{equation}
			If both $S_n(y)$ and  $S_n(z)$ are connected to $\partial B_n$, then either there exists a cluster in $ B_n$ that intersects $S_n(y)$,  $S_n(z)$ and $\partial B_n$,  or there exist two disjoint clusters crossing from $B_{3t(n)}(y)$ to $\partial B_{n/2}(y)$. Therefore, using the clusters inside $B_{t(n)}(y)$ as in Case 1, we have
			\begin{align}
				\label{eq:87}
				\tau(x,y)-\tfrac1{50}\varepsilon_n\le \tau(x,z)+\mathrm P_p[\mathsf{Piv}(\sigma(n), t(n))]+\mathrm P_p[\mathsf{Piv}(3t(n), n/2)]\\ \overset{\eqref{eq:78}+\eqref{eq:79}}\le  \tau(x,z)+\tfrac1{50}\varepsilon_n,\qquad \qquad\qquad \qquad
			\end{align}
			which implies $\tau(x,z)\ge \tau(x,y)-\varepsilon_n$.
			
			\item[Case 3:]   $2t(n)>d(x,y)\ge d(y,z)$.
			
			In this case, we use that
			$1-\tfrac1{50}\varepsilon_n\le \mathrm P_p[S_n(x)\lr{} \partial B_n,\,S_n(z)\lr{} \partial B_n]$. Reasoning as in Case~2, we get
			\begin{equation}
				\label{eq:88}
				1-\tfrac1{50}\varepsilon_n\le  \tau(x,z) + \mathrm P_p[\mathsf{Piv}(3t(n), n/2)]\overset{~\eqref{eq:79}}{\le}  \tau(x,z)+\tfrac1{100}\varepsilon_n,
			\end{equation}
			which implies $\tau(x,z)\ge 1-\varepsilon_n$.
		\end{description}
		
		For the proof of \eqref{eq:84}, we proceed similarly.  If $d(x,y)\ge 2t(n)$, we first use the Harris--FKG Inequality to show $$\theta(p)^2\tau(x,y)\le \mathrm P_p[x\lr{}\partial B_n,y\lr{}\partial B_n, S_n(x)\lr{B_n} S_n(y)].$$ If  the event estimated on the right-hand side occurs, then either $x$ is connected to $y$ or there is no local uniqueness around $x$ or around $y$. Therefore,
		\begin{equation}
			\label{eq:89}
			\theta(p)^2\tau(x,y)\le \mathrm P_p[x\lr{B_n} y]+2\mathrm P_p[\mathsf{Piv}(\sigma(n),t(n))]\overset{\eqref{eq:78}}\le     \mathrm P_p[x\lr{B_n} y]+\varepsilon_n.
		\end{equation}
		If $d(x,y)<2t(n)$, using the estimate \eqref{eq:79}, we directly get that $\mathrm P_p[x\lr{B_n}y]\ge \theta(p)^2-\varepsilon_n\ge \theta(p)^2\tau(x,y)-\varepsilon_n $.  
	\end{proof}

	\begin{lemma}
		\label{lem:16}
		For every  $n$ large enough, the following holds. For every $u\in B_{n/2} $  satisfying
		\begin{equation}
			\tau_{p,n}(o,u)\le 2/3,\label{eq:98}
		\end{equation}
		there exists a set $A\subset  B_{d(o,u)}$  of cardinality at least $\log^4n$ such that
		\begin{equation}
			\label{eq:99}
			\forall x,y\in A\quad    \tau_{p,n} (x,y)\le 1-\varepsilon_{n}.
		\end{equation}
	\end{lemma}
	
	\begin{remark}
		When the underlying graph $G$ is the hypercubic lattice $\mathbb Z^d$, the lemma above can be easily proved using the symmetries of the graph.
	\end{remark}

	\begin{proof}
		Let $n$ be a large integer. Let $u\in B_{n/2}$ and  fix $\gamma$ some geodesic path from the origin $o$ to $u$. If $a$ and $b$ denote two vertices belonging to this fixed path, we denote by $[a,b]$ the set of all vertices $x$ of the path lying between $a$ and $b$, i.e. satisfying $d(a,b)=d(a,x)+d(x,b)$. Sets of this form will be called \emph{segments} in this proof.  By convention, we always assume that $d(o,a)\le d(o,b)$ when we consider a segment $[a,b]$.
		
		We will build the set $A$ as a subset of the segment $I=[o,u]$, in a way reminiscent of the construction of the triadic Cantor set. We construct two suitable subsegments $I_0=[a_0,b_0]$ and $I_1=[a_1,b_1]$, such that the two-seed function $\tau(x,y)$, $x\in I_0$, $y\in I_1$ is well-controlled and both $\tau(a_0,b_0)$ and $\tau(a_1,b_1)$ have a nice upper bound.
		Then, we repeat this splitting operation in each of the two segments. After $k$ steps, we construct $2^k$ intervals, hence getting $2^{k+1}$ endpoints. The construction is such that the two-seed function between any two of these points is well-controlled. The proof is then concluded by choosing a suitable  number of steps. The ``splitting operation'' of an interval at one step relies on the following claim.
		
		For convenience, as in the proof of Lemma~\ref{lem:15}, we drop the indices $p$ and $n$ from the notation $\tau_{p,n}$ and $\varepsilon_n$. 
		
		\begin{claim}\label{claim:12345}Let $\alpha \ge \frac{1}{2}$. Let $[a,b]$ be a segment such that $\tau(a,b)\leq \alpha$. Then, there exist two vertices $a',b'\in [a,b]$ such that   
			\begin{equation}
				\label{eq:100}
				\forall x\in [a,a']\ \forall y\in  [b',b]  \quad \tau(x,y)\le \alpha^{1/3}+6\varepsilon
			\end{equation}
			and
			\begin{equation}
				\label{eq:101}
				\max(\tau(a,a'),\tau(b,b'))\le   \alpha^{1/3}.
			\end{equation}
		\end{claim}
		
		\begin{proof}[Proof of Claim~\ref{claim:12345}.]
			First, notice that if $\alpha\geq1$, then the claim is straightforward: taking $I_1$ and $I_2$ \emph{any} two (possibly equal) subsegments of $I$ works. We therefore assume that $\alpha<1$.
			
			Observe that, provided $n$ is large enough, $x\mapsto \tau(a,x)$ is $\varepsilon$-Lipschitz on $B_{n/2}$. This follows from the estimate~\eqref{eq:83} from Lemma~\ref{lem:15} and the fact that for any two adjacent vertices $x$ and $y$ in $B_{n/2}$, we have $\tau(x,y)=1$. 
			Consider all the vertices $c\in [a,b]$ such that $\tau(a,c)\le \alpha^{1/3}$. This set is nonempty, as $\tau(a,b)\le\alpha<\alpha^{1/3}$. We denote by $a'$ the vertex of this set which is closest to $a$.
			Since $\tau(a,a)=1$, the $\varepsilon$-Lipschitzianity of  $\tau(a,\hspace{0.4mm}\cdot\hspace{0.8mm})$ gives $\alpha^{1/3}-\varepsilon\le\tau(a,a')\le \alpha^{1/3}$. As a result, we have
			\begin{equation}
				\label{eq:102}
				\forall x\in [a,a']\quad \tau(a,x)\ge  \alpha^{1/3}-\varepsilon,
			\end{equation}
			where the $\varepsilon$ is actually needed only when $x=a'$.
			Similarly, we take $b'\in [a,b]$ such that $\tau(b',b)\le   \alpha^{1/3}$ and
			\begin{equation}
				\label{eq:103}
				\forall x\in [b',b]\quad \tau(x,b)\ge  \alpha^{1/3}-\varepsilon.
			\end{equation}
			Equation~\eqref{eq:101} holds by definition and it remains to prove \eqref{eq:100}.
			Let $x\in [a,a']$ and $y\in [b,b']$. Since $\tau(a,b)\leq\alpha$, using twice the estimate~\eqref{eq:83} yields
			\begin{equation}
				\alpha \geq \tau(a,x)\tau(x,y)\tau(y,b)-2\varepsilon.\label{eq:80}
			\end{equation}
			Using the inequality $\alpha \ge \frac{1}{2}$ and recalling that both $\tau(a,x)$ and $\tau(y,b)$ are larger than $\alpha^{1/3}-\varepsilon$, we obtain
			\begin{equation}
				\tau(x,y)\leq \alpha^{1/3}+{2}{\alpha^{-2/3}}\varepsilon+2\varepsilon\le \alpha^{1/3}+6\varepsilon.\label{eq:81}
			\end{equation}
		\end{proof}

		Let us prove by induction that for every $k\ge 0$, we can find a family of segments $I_1,\dots, I_{2^k}$ that satisfies the following conditions:
		\begin{itemize}
			\item for $i\neq j$, for any $x\in I_i$ and any $y$ in $I_j$, we have $\tau(x,y)\le (2/3)^{1/3^k}+6k\varepsilon$,
			\item for any segment $[a,b]$ of the family, we have $\tau(a,b)\le  (2/3)^{1/3^k}$.
		\end{itemize}
		
		For $k=0$, taking $[o,u]$ works, by hypothesis. Let $k\geq0$ be such that the property holds at step $k$, and let us prove that the property holds at step $k+1$. Let us take $I_1,\dots,I_{2^k}$ as above. For every $i$, we apply Claim~\ref{claim:12345} to $I_i$, which yields two subsegments $I^{(1)}_i$ and $I^{(2)}_i$ of $I_i$. For $x\in I^{(1)}_i$ and $y\in I^{(2)}_i$, we have
		$$\tau(x,y)\le \left((2/3)^{1/3^k}+6k\varepsilon\right)^{1/3}+6\varepsilon \le (2/3)^{1/3^{k+1}}+6(k+1)\varepsilon.$$
		Likewise, if $[a,b]$ denotes either $I^{(1)}_i$ or $I^{(2)}_i$, we have $
		\tau(a,b)\leq (2/3)^{1/3^{k+1}}$.
		It remains to check that if $i\neq j$, then for every $x\in I^{(0)}_i \cup I^{(1)}_i$ and every $y\in I^{(0)}_j \cup I^{(1)}_j$, we have $\tau(x,y)\le (2/3)^{1/3^{k+1}}+6(k+1)\varepsilon$. But this is clear: since $x\in I_i$, $y\in I_j$ and $i\neq j$, we have
		$$\tau(x,y)\le (2/3)^{1/3^k}+6k\varepsilon\le (2/3)^{1/3^{k+1}}+6(k+1)\varepsilon.$$
		The result thus holds for all $k$.
		
		Let us use this result for $k=\lceil 8 \log \log n\rceil$, which we will handle as $8 \log \log n$ for readability. Let $I_1,\dots,I_{2^k}$ be as above for this specific value of $k$. Let $x_1\in I_1, \dots, x_{2^k}\in I_{2^k}$. Given $i\neq j$, we have
		$$
		\tau(x_i,x_j)\leq (2/3)^{1/3^k}+6k\varepsilon_n\le e^{\log(2/3)/\log(n)^{8\log3}} + 48\frac{\log \log n}{\log^{10} n}.
		$$
		Since $8\log(3)<10$, taking $n$ large enough guarantees that $\tau(x_i,x_j)\leq 1-\varepsilon_{n/2}$. In particular, we have $\tau(x_i,x_j)<1$, hence $x_i\neq x_j$. We set $A=\{x_1,\dots,x_{2^k}\}$. It remains to check that $A$ contains at least $\log^4 n$ elements, which is straightforward as $2^k=(\log n)^{8 \log 2}$ and $8\log 2 > 4$.
	\end{proof}
	
	\subsection{Proof of Proposition~\ref{prop:5}}
	\label{sec:proof-prop}
	
	Without loss of generality, we can prove the statement for $p+\delta$ instead of $p$, where $p>p_c$ and $\delta>0$.
	Assume that $\kappa_{p+\delta}(m,2n)\leq \theta(p+\delta)^2/2$. We can thus take  $u\in B_m$ such that
	\begin{equation}
		\label{eq:104}
		\mathrm P_{p+\delta}[o \lr{B_{2n}} u]\le \theta(p+\delta)^2/2. 
	\end{equation}
	By the estimate~\eqref{eq:84} from Lemma~\ref{lem:15}, this implies that $\tau_{p+\delta,2n}(o,u)\leq\frac{1}{2}+ \frac{\varepsilon_{2n}}{\theta(p+\delta)^2}\leq \frac{2}{3}$, provided $n$ is taken large enough. By Lemma~\ref{lem:16}, we can find a subset $A\subset B_{m}$ of cardinality $\lceil\log^4 n\rceil$ such that for every $x,y\in A$, we have
	\begin{equation}
		\label{eq:105}
		\tau_{p+\delta,2n}(x,y)\le  1-\varepsilon_{2n}.
	\end{equation}
	In words, the inequality above states that the two-seed function is not too close to $1$ for every pair of points of $A$.  By decreasing the edge density from $p+\delta$ to $p$, we obtain that the points of $A$ are pairwise connected with low probability. More precisely, the equation above and the contrapositive of Lemma~\ref{lem:14}  imply
	\begin{equation}
		\label{eq:106}
		\forall x,y\in A \quad \mathrm P_p [x\lr{B_n} y]\le \varepsilon_{2n}.
	\end{equation}
	Together with   $\displaystyle\min_{x\in A}\mathrm P_p[x\lr{}\partial B_n]\ge \theta(p)\ge 2 |A| \varepsilon_{2n}$  (which holds for $n$ large enough), we obtain   
	\begin{equation}
		\label{eq:107}
		\min_{x\in A} \mathrm P_p[x\lr{}\partial B_n]\ge \theta(p) \ge 2 |A|  \max_{x,y\in A}\mathrm P_p [x\lr{B_n} y].
	\end{equation}
	By applying Proposition~\ref{prop:6} to the graph $G$ induced by the ball $B_n$, we finally get
	\begin{equation}
		\label{eq:108}
		\mathrm P_{p+\delta}[A\lr{} \partial B_n]\ge 1-e^{-2\delta\theta(p)|A|},
	\end{equation}
	which concludes the proof since  $A\subset B_m$ and $|A|\ge \log^4n$. \qed
	
	\section{Uniqueness via sprinkling}
	\label{sec:uniq-via-sprinkl}
	
	In this section, we establish the following proposition, which revisits the techniques of \cite{benjamini2017homogenization}. Recall that the coupled measure $\mathbf P$ and the sprinkled uniqueness event $U_{p,q}(m,n)$ are defined in Section~\ref{sec:definition-notation}.
	
	\begin{proposition}
		\label{prop:7}
		Let $G$ be a transitive graph of polynomial growth with $d\ge2$.
		Let $R \ge 1$ be so that the conclusion of Lemma~\ref{lem:1} holds.
		Let $p\in [0,1]$. Let $\eta, \delta>0$ be such that $p+\delta\leq 1$. Let $(k,m,n)$ be such that $R\le k \le m \le n$ and $n\geq \log |B_{2n}|$. We make the following two assumptions:
		\begin{enumerate}[label=(\alph*)]
			\item\label{item:a} $\mathrm P_p[B_k\lr{}\partial B_{6n}]\ge 1-\frac{\eta}{|B_{2n}|}$, 
			\item\label{item:b} $ \forall x,y\in B_{3k}$ $\mathrm P_p[x\lr{B_m} y]\ge \delta $. 
		\end{enumerate}
		Then, we have
		\begin{equation}
			\label{eq:109}
			\mathbf P[U_{p,p+\delta}(n,4n)] \ge 1- \eta  - 100|B_{2n}|^2 \exp\left(-\frac{\delta^3n}{8m \log|B_{2n}|)}\right).
		\end{equation} 
	\end{proposition}

	In this section, we consider the family of coupled configuration $(\omega_p)_{p\in[0,1]}$ under the measure $\mathbf P$. Our goal is to show that, with high probability, all the $p$-clusters  crossing from $B_n$ to $\partial B_{4n}$ are $(p+\delta)$-connected to each other within the annulus $A_{n,2n}$. In this section and contrary to the \emph{connected} annuli of Section~\ref{sec:renorm-corr-funct}, the annulus $A_{r,s}$ is defined as the set of all edges $e\subset B_{4n}$ that intersect $B_s$ but not $B_r$. If we achieve our goal, the crossing $p$-clusters will a fortiori get $(p+\delta)$-connected within the larger set $B_{4n}$. 
	
	For every $r\le 2n$, consider the percolation configuration $Y_r$ in $B_{4n}$ defined as follows: for every edge $e\subset B_{4n}$,
	\begin{equation}
		\label{eq:110}
		Y_r(e):=
		\begin{cases}
			\omega_{p}(e)\mathbf 1_{\{e\lr[p]{}\partial B_{4n}\}}& \text{if $e$ intersects $B_r$} \\
			\omega_{p+\delta}(e) & \text{if $e\in A_{r,4n}$.}
		\end{cases}
	\end{equation}
	The configuration $Y_r$ can be explored and understood in the following manner. First, we perform the classical exploration from $\partial B_{4n}$ of all the $\omega_p$-clusters touching it --- by doing so, we also reveal the (closed) edges at the boundary of these clusters. Then, we further reveal the $\omega_{p+\delta}$-status of every single edge included in $A_{r,4n}$. Conditionally on $Y_r$, the status of the edges intersecting $B_r$ are independent. Besides, for any edge $e$ that intersects $B_r$ and does not already satisfy $Y_r(e)=1$, we have:
	\begin{equation}
		\label{eq:111}
		\mathbf{P}[\omega_{p+\delta}=1~|~Y_r]\
		\begin{cases}
			\geq \delta & \text{if }e\text{ is adjacent to some }e'\text{ satisfying }Y_r(e')=1,\\
			=p+\delta& \text{otherwise}.
		\end{cases}
	\end{equation}
	
	\begin{figure}
		\centering
		\includegraphics[scale=0.8]{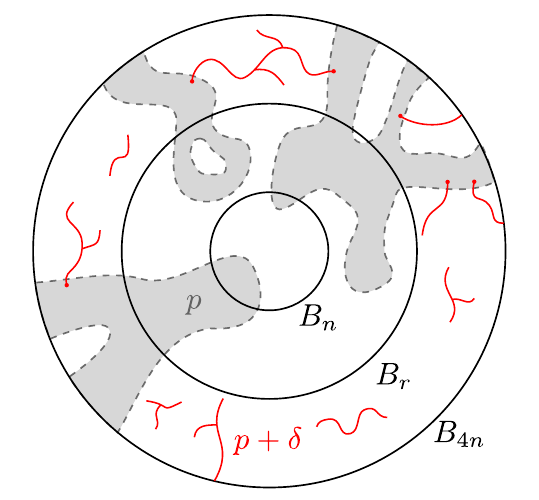}
		\caption{The configuration $Y_r$ is the union of the grey area and the red one.}
		\label{fig:7}
	\end{figure}

	For every configuration  $\omega$ in $B_{4n}$ and every $r\leq 2n$, define $\mathscr C_r(\omega)$ to be the set of all the clusters of $\omega$ intersecting both $B_r$ and $\partial B_{4n}$. We further set
	\begin{equation}
		\label{eq:112}
		N_r(\omega)=|\mathscr{C}_r(\omega)|.
	\end{equation}
	We aim at proving that with high probability, we have $N_n(Y_{n})=1$. We will do so by making use of the following event, which roughly states that ``large clusters grow from everywhere in $B_{2n}$'':
	\begin{equation}
		\label{eq:113}
		\mathcal E=\{\forall x\in B_{2n}\ B_k(x)\lr[p]{}\partial B_{4n}\}.   
	\end{equation}
	We will prove that $N_n(Y_{n})=1$ holds with high probability by proving that $\mathcal{E}$ holds with high probability and that, conditionally on this event, $N_n(Y_{n})=1$ holds with high probability.
	
	We will implement this strategy by repeatedly using the following lemma. Lemma~\ref{lem:17} guarantees that sprinkling a layer of thickness $\simeq \frac{n}{\log B_{2n}}$ is very likely to divide the number of ``crossing clusters'' by a factor at least 2 --- except if this number is already 1. Iterating this process $\simeq \log B_{2n}$ times inwards (going from the sphere of radius $2n$ to that of radius $n$) will yield $N_n(Y_{n})=1$ with high probability.

	\begin{lemma}
		\label{lem:17}
		For every $r,s\in[n,2n]$ such that $r\le s$, we have
		\begin{equation}\label{eq:114}
			\mathbf P[\mathcal E \cap \big\{N_r(Y_r)> \max(1,\tfrac12
			N_s(Y_s))\big\} ]\le 25 |B_{2n}| \exp\left(-\frac{\delta^3(s-r)}{3m}\right).
		\end{equation}
	\end{lemma}

	\begin{proof}[Proof of Lemma~\ref{lem:17}]
		Fix $r,s\in[n,2n]$ such that $r\le s$. The configuration $Y_r$ is obtained from $Y_s$ by adding some open edges in the annulus $A_{r,s}$. In particular, we have $Y_r\ge Y_s$ and our goal is to show that, with high probability, every cluster of $\mathscr{C}_s(Y_s)$ gets merged with at least one other cluster of $\mathscr{C}_s(Y_s)$ in the configuration $Y_r$. To analyse this ``merging effect'', let us introduce the configuration
		\begin{equation}
			\label{eq:115}
			Z=Y_r-Y_s.
		\end{equation}
		
		Let us condition on the possible values for $Y_s$. Say that a configuration $\xi$ in $B_{4n}$ is \emph{admissible} if $\mathbf P[Y_s=\xi]>0$ and $B_k(x)$ is connected to $\partial B_{4n}$ in $\xi$ for every $x\in B_{2n}$.  Writing $\mathbf P^{\xi}:=\mathbf P[\ \cdot\ |Y_s=\xi]$ for every admissible $\xi$, we can rewrite the left-hand side of \eqref{eq:114} as follows:
		\begin{align}
			\label{eq:116}
			&\mathbf P[\mathcal E \cap \big\{N_r(Y_r)> \max(1,\tfrac12
			N_s(Y_s))\big\} ]\\
			&\qquad	=\sum_{\xi\text{ admissible}}  \mathbf P^\xi[N_r(\xi+Z)> \max(1,\tfrac12
			N_s(\xi))] \mathbf P[Y_s=\xi].
		\end{align}
		
		From now on, we fix an admissible configuration $\xi$ and the proof will be complete once we show
		\begin{equation}
			\label{eq:117}
			\mathbf P^{\xi}[N_r(\xi+ Z)> \max(1, \tfrac12 N_s(\xi))]\le 25 |B_{2n}|\exp\left(-\frac{\delta^3(s-r)}{3m}\right). 
		\end{equation}  
		Notice that under $\mathbf P^{\xi}$, the coordinates $Z(e)$, $e\subset B_{s}$, are independent and satisfy
		\begin{equation}
			\label{eq:118}
			\mathbf P^{\xi}[Z(e)=1]
			\begin{cases}
				= p+\delta & \text{ if $e\in A_{r,s}$ and $e$ is not adjacent to an edge of $\xi$}\\
				\ge \delta& \text{ if $e\in A_{r,s}$, $\xi(e)=0$ and $e$ is adjacent to an edge of $\xi$},\\
				=0& \text{ if $e$ intersects $B_r$}.
			\end{cases}
		\end{equation}
		If $N_{r}(\xi)=1$, then we also have $N_{r}(\xi+Z)=1$. Indeed, all the open edges of $\xi$ intersecting $B_r$ are already connected to $\partial B_{4n}$ in $\xi$, and no open edge of $Z$ intersects $B_r$: therefore, adding the edges  of $Z$ to $\xi$ cannot create  a new cluster crossing from $B_r$ to $\partial B_{4n}$. As a result, if $N_{r}(\xi)=1$, the left-hand side of \eqref{eq:117} is equal to $0$ and there is nothing to prove. We thus assume that $N_{r}(\xi) > 1$ and we will prove that 
		\begin{equation}
			\mathbf P^{\xi}[N_r(\xi+ Z)>\tfrac12 N_s(\xi)]\le 25|B_{2n}|\exp\left(-\frac{\delta^3(s-r)}{3m}\right). 
		\end{equation}
		We will say that a cluster of $\mathscr C_{s}(\xi)$ is \emph{merged} if it is connected to at  least one other cluster of  $\mathscr C_{s}(\xi)$ in the configuration $\xi+ Z$. Please note that even if the cluster under consideration is taken in $\mathscr{C}_r(\xi)\subset \mathscr{C}_s(\xi)$, it is said to be merged if it is connected to at least one other cluster of $\mathscr C_{s}(\xi)$ --- with an $s$ --- in the configuration $\xi+ Z$. We will prove that typically  any cluster of  $\mathscr C_{s}(\xi)$ that crosses the annulus $A_{r,s}$ is merged with high probability. More  precisely we will prove that 
		\begin{equation}
			\label{eq:119}
			\forall C\in \mathscr C_r(\xi) \quad \mathbf P^{\xi}[C \text{ is not  merged}] \le 25\exp\left(-\frac{\delta^3(s-r)}{3m}\right). 
		\end{equation}
		
		The main idea behind \eqref{eq:119} is that the cluster $C$ crosses $\simeq (s-r)/m$ disjoint annuli of thickness $m$. In each annulus, the cluster $C$ comes at distance $\leq k$ to another cluster of $\mathscr C_s(\xi)$. Morally, Hypothesis \ref{item:b} can thus be used to bound the probability that the two clusters get connected by a $(\xi+Z)$-open path lying inside the considered annulus. Since these events are independent (we have disjoint annuli), we obtain the bound~\eqref{eq:119}.   We  postpone the rigorous derivation of
		\eqref{eq:119} to the end of the proof, and we now explain how to deduce \eqref{eq:117} from \eqref{eq:119}.

		The key observation is that if every cluster of $\mathscr C_{r}(\xi)$ is merged, then $N_r(\xi+Z)\le \frac12 N_s(\xi)$. To see this, let us assume that every cluster of $\mathscr C_{r}(\xi)$ is merged, and let us prove that every cluster of $\mathscr{C}_r(\xi+Z)$ contains at least two clusters of $\mathscr{C}_s(\xi)$. Let $C\in \mathscr{C}_r(\xi+Z)$. We have already seen that adding $Z$ to $\xi$ can create no new cluster crossing from $B_r$ to $\partial B_{4n}$. In other words, we can fix some $C_1\in \mathscr{C}_r(\xi)$ such that $C_1\subset C$. Since $C_1$ is merged, we can take some other cluster $C_2\in \mathscr{C}_s(\xi)$ such that $C_2 \subset C$, which ends the proof of the observation. Using this observation, the union bound, \eqref{eq:119} and $r\leq 2n$ yields
		\begin{align}
			\label{eq:120}
			&	\mathbf P^\xi[N_r(\xi+Z)>\tfrac12 N_s(\xi)]\\
			&\qquad\le \mathbf P^\xi[\exists C \in \mathscr C_r(\xi),\ C\text{ is not merged}]\\
			&\qquad \le  25|\mathscr C_r(\xi)| \exp\left(-\frac{\delta^3(s-r)}{3m}\right)\le 25|B_{2n}|\exp\left(-\frac{\delta^3(s-r)}{3m}\right),  
		\end{align}
		which is the desired inequality.
		
		We now give the details of the proof of \eqref{eq:119}.   Fix $C\in \mathscr C_r(\xi)$. Set $$J:=3m\mathbb Z \cap [r+3m,s-3m].$$ For every $j\in J$, we will use the set $S_j^\infty$ introduced on page~\pageref{page:123456}. 
		Notice that $|J|\ge \frac{s-r}{3m}-3$ and that $C$ intersects every $S_j^\infty$, by Lemma~\ref{lem:5}. Define
		$\widetilde{C}$ to be the union of all clusters of $\mathscr{C}_s(\xi)\setminus \{C\}$.
		We claim that for every $j\in J$, there exists $x_j\in S_j^\infty$ such that   
		\begin{equation}
			\label{eq:121}
			\mathrm d(x_j,C)\le 3k \quad\text{and}\quad  \mathrm d(x_j,\widetilde C)\le 3k.
		\end{equation}
		To prove this, consider a path $\gamma$ from some vertex $x \in C\cap S_j^\infty$ to another vertex in $\widetilde C\cap S_j^\infty$ which stays in the $R$-neighbourhood of $S_j^\infty$ (such path exists by Lemma~\ref{lem:1}). If $x$ is at distance $k$ or less from $\widetilde C$, we simply choose $x_j=x$ and we are done. Otherwise, consider $z\in\gamma$ at distance exactly $k+1$ from $\widetilde C$. Since $B_k(z)$ is connected to $\partial B_{4n}$ in $\xi$  (because $\xi$ is admissible) and $B_k(z)\cap \widetilde C=\emptyset$, we must have $B_k(z)\cap C\neq\emptyset$. It now suffices to choose $x_j \in S_j^\infty\cap B_R(z)$, use the triangle inequality and remember that $R\leq k$ to get the claim.
		
		For every $j\in J$, fix some $x_j$ as above and introduce the event
		\begin{equation}
			\label{eq:122}
			\mathcal E_j=\left\{C\lr{B_m(x_j)} \widetilde C \text{ in } \xi+Z\right\}.
		\end{equation}
		By Hypothesis \ref{item:b} and and a sprinkling argument (as in the proof of Lemma~\ref{lem:8}), we have
		\begin{equation}
			\label{eq:123}
			\mathbf P[\mathcal E_j]\ge \delta^2 \min_{x,y\in B_{3k}(x_j)}\mathrm P_p [x\lr{B_m(x_j)} y]\ge \delta^3.
		\end{equation}
		As soon as one of the event $\mathcal E_j$ occurs, the cluster $C$ gets connected to another cluster of $\mathscr C_s(\xi)$ in the configuration $\xi+Z$. Since the events $\mathcal E_j$ are independent, we get
		\begin{equation}
			\label{eq:124}
			\mathbf P^{\xi}[C \text{ is not  merged}] \le (1-\delta^3)^{|J|}\le \exp\left(3\delta^3-\frac{\delta^3(s-r)}{3m}\right) \le 25 \exp\left(\frac{\delta^3(s-r)}{3m}\right). 
		\end{equation}	
		
	\end{proof}

	\begin{proof}[Proof of Proposition~\ref{prop:7}]
		We wish to prove that 
		
		\begin{equation}
			\label{eq:125}
			\mathbf P[N_n(Y_n)>1] \le \eta  + 100 |B_{2n}|^2 \exp\left(-\frac{\delta^3n}{8m \log|B_{2n}|)}\right).
		\end{equation}
		First, we have  
		\begin{equation}
			\label{eq:126}
			\mathbf P[N_n(Y_n)>1]\leq\mathbf P[\mathcal E^c]+\mathbf P[\mathcal E \cap \{N_n(Y_n)>1\}].
		\end{equation}
		The first term on the right hand side can be bounded as follows, using the union bound, automorphism-invariance and Hypothesis~\ref{item:a}:
		\begin{equation}
			\label{eq:127}
			\mathbf P[\mathcal E^c]\le \sum_{x\in B_{2n}} \mathrm P_p[B_k(x) \nlr{} \partial B_{6n}(x)]\le |B_{2n}| \mathrm P_p[B_k \nlr{} \partial B_{6n}]\le \eta.
		\end{equation}
		
		In order to bound the second term, we use the property established in Lemma~\ref{lem:17}, namely that  clusters merge with high probability  in annuli of the form $A_{r,r+\Delta}$ for $\Delta:= \lceil\frac{n}{4\log| B_{2n}|}\rceil $, when sprinkling from $p$ to $p+\delta$. Recall that $N_r(Y_r)$ counts the number of $p$-clusters crossing from $B_r$ to $\partial B_{4n}$, where the clusters are identified if they are $(p+\delta)$-connected in the annulus $A_{r,4n}$. We start with the trivial bound $N_{2n}(Y_{2n})\le |B_{2n}|$ for the number of clusters in $Y_{2n}$ intersecting the ball $|B_{2n}|$.  Applying Lemma~\ref{lem:17} to $r=2n-\Delta$ and $s=2n$, we obtain
		\begin{equation}
			\label{eq:128}
			\mathbf P[\mathcal E\cap \{N_{2n-\Delta}(Y_{2n-\Delta})>\max(1,\tfrac12  |B_{2n}|)\}]\le 25 |B_{2n}|\exp\left(- \frac{\delta^3\Delta}{3m}\right).
		\end{equation}
		Then, by induction, we get for every $i\le n/\Delta$, 
		\begin{equation}
			\mathbf P[\mathcal E\cap \{N_{2n-i\Delta}(Y_{2n-i\Delta})>\max(1,2^{-i}|B_{2n}|)\}]\le 25 i|B_{2n}|\exp\left(- \frac{\delta^3\Delta}{3m}\right).\label{eq:129}
		\end{equation}
		Choosing $i=\lfloor n/\Delta \rfloor$ and observing that $|B_{2n}| \leq 2^i$, we obtain
		\begin{equation}
			\label{eq:130}
			\mathbf P[\mathcal E\cap \{N_{n}(Y_{n})>1]\le 25\times 4\times|B_{2n}| \log|B_{2n}| \exp\left(- \frac{\delta^3\Delta}{3m}\right).
		\end{equation}
		Plugging the two bounds \eqref{eq:127} and \eqref{eq:130} in \eqref{eq:126} concludes the proof.
	\end{proof}

	\section{Proof of Proposition~\ref{prop:1}}
	\label{central-proposition}
	
	In this section, we combine all the tools introduced in the previous sections in order to establish Proposition~\ref{prop:1}. Our goal is to prove that there exists a unique large cluster  in large balls with high probability: for every $n$ large enough, we want to show that
	\begin{equation}
		\label{eq:132}
		\mathrm P_p[B_{n/10}\lr{}\partial B_{n},~ U(n/5,n/2)]\geq 1-e^{-\sqrt{n}}.
	\end{equation}
	Let us start with two reductions of the problem.
	
	\paragraph{First reduction: a lower bound on the two-point function in long corridors.} At several places in the paper, we saw that local uniqueness is intimately related to lower bounds on the two-point function inside finite regions. Notably, Proposition \ref{prop:2} asserts that  Equation~\eqref{eq:132} holds if we can prove the lower bound
	\begin{equation}
		\label{eq:133}
		\forall p>p_c\quad  \limsup_{n\to\infty} \kappa_p(n\log^{3d} n,n)>0.
	\end{equation}

	\paragraph{Second reduction: a lower bound on the two-point function in short corridors.} In Section~\ref{sec:sharp-thresh-seed}, we proved that crossing probabilities always undergo a sharp threshold phenomenon: the probabilities jump fast from $0$ to $1$ as $p$ varies. An important consequence  of this phenomenon is quantified in Corollary~\ref{cor:3}: a lower bound on the two-point function in a corridor at parameter $p$ implies a lower bound in a longer corridor at a slightly larger parameter $p+\delta$. Using this corollary, one can see that   Equation~\eqref{eq:133} will be established if we show the following weaker statement
	\begin{equation}
		\label{eq:134}
		\forall p>p_c\quad  \limsup_{n\to\infty} \kappa_p(n/\log^{4} n,n)>0.
	\end{equation}
	Getting Equation~\eqref{eq:134} is the core of the proof. In Section~\ref{sec:lower-bound-two}, we start by explaining the main ideas in a sketch of proof. We then give a formal derivation in Section~\ref{sec:lower-bound-two-1}.

	\subsection{Lower bound on the two-point function: sketch of proof}
	\label{sec:lower-bound-two}
	In this section we give the main ideas behind the proof of Equation~\eqref{eq:134}.
	
	Proceeding by contradiction, we assume that we can fix some $p_1>p_c$ such that
	\begin{equation}
		\label{eq:94}
		\lim_{n\to\infty} \kappa_{p_1}(n/\log^{4} n,n)=0.\tag{\textbf{H}}
	\end{equation}
	In order to reach a contradiction, we rely on two main ideas:
	\begin{itemize}
		\item  Let $p_0\in (p_c,p_1)$. By monotonicity relative to $p$,  Hypothesis \eqref{eq:94} entails that the two-point function at $p_0$ is small. This allows us  to apply the Hamming distance argument (Proposition~\ref{prop:5}) to show that the crossing probability in annulus is very high. More precisely, setting $\alpha=n/\log^{4} n$, we have
		\begin{equation}
			\label{eq:91}
			\mathrm P_{p_0}[B_{\alpha}\lr{} \partial B_n] \ge 1 - e^{-\log^3 n}.   
		\end{equation}
		
		\item The second idea is to use uniqueness via sprinkling (Proposition~\ref{prop:7}): the estimate   \eqref{eq:91} gives us the first assumption needed to apply this proposition. In rough terms, \eqref{eq:91} says that, with high probability,  every ball of radius $\alpha$ within $B_n$ belongs to some large cluster; and Proposition~\ref{prop:7} states that, by sprinkling, all these clusters get connected. More precisely, we  establish that
		\begin{equation}
			\label{eq:92}
			\mathbf P[U_{p_0,p_1}(n/4,n)]\xrightarrow[n\to\infty]{}1.
		\end{equation}
		From there, it is easy to reach a contradiction: fix two points $x$ and $y$ in $B_{n/\log^4 n}$ and consider the event that
		both $x$ and $y$ are $p_0$-connected to $\partial B_n$. By the Harris--FKG inequality, this occurs with probability at least $\theta(p_0)^2$. On this event,  both  $p_0$-clusters  of $x$ and $y$ cross the annulus  $B_n\setminus B_{n/4}$ and, by $~\eqref{eq:92}$, $x$ and $y$ must be $p_1$-connected inside $B_n$ with high probability. This concludes that
		\begin{equation}
			\label{eq:93}
			\kappa_{p_1}(n/\log^4 n,n)\ge \theta(p_0)^2/2   
		\end{equation}
		for $n$ large, which contradicts \eqref{eq:94}.
	\end{itemize}

	There is a gap in the proof exposed above: to apply uniqueness via sprinkling (Proposition~\ref{prop:7}), Equation~\eqref{eq:91} is not sufficient. We also need a second assumption, namely a lower bound on the two-point function of the form
	\begin{equation}
		\label{eq:95}
		\kappa_{p_0}(3\alpha,n/\log^4 n) \ge \delta> 0,
	\end{equation}
	where $\delta>0$ is some positive constant.  It is not clear that this last equation occurs with the choice $\alpha=n/\log^4 n$. This problem is fixed by reducing the value of $\alpha$ in a suitable way: we show that there exists $\alpha$  that is small  enough for \eqref{eq:95} to hold, but also large enough to have \eqref{eq:91}.

	\subsection{Lower bound on the two-point function: formal proof}\label{sec:lower-bound-two-1}
	
	For the formal proof of Equation~\eqref{eq:134}, we rely on the following lemma.
	
	\begin{lemma}
		\label{lem:11}
		Let $p_1>p_0>p_c$. Assume that  \eqref{eq:94} holds at $p_1$. Then, there exists a sequence $(\alpha(n))$ such that 
		\begin{enumerate}
			\item\label{item:4} $\mathrm P_{p_0}[B_{\alpha(n)}\lr{} \partial B_{n/2}] \ge 1 - e^{-\log^3 (n/2)}$, for every $n$ large enough,
			\item\label{item:5} $\displaystyle \limsup_{n\to\infty} \kappa_{p_0}(3\alpha(n), n/\log^4 (n))>0$, along the sequence $n\in 12\mathbb N$.
		\end{enumerate}
	\end{lemma}
	
	\begin{proof}
		Fix  $p\in(p_c,p_0)$ and $\delta<\theta(p)^2/2$.   Let $R \ge 1$ be so that the conclusion of Lemma~\ref{lem:1} holds.
		For every fixed $n\ge 1$, we have $\kappa_{p}(0 , n)=1$ and $\lim_{k\to \infty}  \kappa_{p}(k, n)=0$. Therefore, one can define
		\begin{equation}
			\label{eq:90}
			\alpha(n)=\min\{k\: : \:  \kappa_{p}(k , n)< \delta\}.  
		\end{equation}
		By \eqref{eq:94} and monotonicity relative to $p$, we have $\lim_{n\to\infty}\kappa_{p}( n/\log^4 n, n)=0$. On the other hand, by Corollary~\ref{cor:4}, we can fix some $\chi>0$ (depending on $p$) such that  $\liminf_{n\to\infty}\kappa_{p}(\exp(\log^\chi(n)), n)>\delta$.   Hence, for every $n$ large enough, $\alpha(n)$ satisfies 
		\begin{equation}
			\label{eq:96}
			\exp(\log^\chi(n))\le \alpha(n) \le n/\log^4 n.    
		\end{equation}
		
		The parameter $\alpha(n)$ is defined as a threshold of the two-point function.  If $k\simeq \alpha(n)$ we  know that the two-point function  $\kappa_p(k,n)$ is of order $\delta$. More precisely, we have the upper bound
		\begin{equation}
			\label{eq:97}
			\kappa_p(\alpha(n),n)< \delta,
		\end{equation}
		and the lower bound
		\begin{equation}
			\label{eq:131}
			\kappa_p(\alpha(n)-1 ,n)\ge\delta. 
		\end{equation}
		In order to prove the first and second items of the lemma, we will use the upper and lower bounds above respectively.

		Since $\delta<\theta(p)^2/2$, the upper bound~\eqref{eq:97} together with Proposition~\ref{prop:5} implies that 
		\begin{equation}
			\label{eq:139}
			\mathrm P_{p}[B_{\alpha(n)}\lr{} \partial B_{n/2}]\ge 1-e^{-\log^3(n/2)}.
		\end{equation}
		for every $n$ large enough. Therefore,  Item~\ref{item:4} holds by monotonicity, as $p_0>p$.
		
		It remains to prove the second item.The  bound \eqref{eq:131} for $\lfloor n/2\log^4 n\rfloor$ gives 
		\begin{equation}
			\label{eq:153}
			\kappa_p(\alpha(\lfloor n/2 \log^4 n\rfloor)-1 ,\lfloor n/2\log^4 n\rfloor)\ge\delta
		\end{equation}
		for every $n$ large enough. We wish to obtain a stronger lower bound when we increase the parameter from $p$ to $p_0$. To achieve this, apply Corollary~\ref{cor:3}. Write  $m=\alpha(\lfloor n/2 \log^4 n\rfloor)-1$. By the lower bound \eqref{eq:96}, we have
		\begin{equation}
			\label{eq:154}
			\log m\ge \log^{\chi/2}(n) 
		\end{equation}
		for every $n$ large enough. By applying Corollary~\ref{cor:3} to $m$, $\lfloor n/2 \log^4 n\rfloor$, and $k= 10/\chi$, we obtain
		\begin{equation}
			\label{eq:141}
			\kappa_{p_0}\left(   \log^5 (n)\cdot\left (\alpha(\lfloor n/2\log^4 n\rfloor)-1\right), n/\log^4 n \right)>\delta
		\end{equation}
		for every $n$ large enough. This is almost Item~\ref{item:5}, except that one needs to replace the first parameter by $\alpha(n)$. Since $\alpha(n)$ grows sublinearly, one can do this replacement along an infinite subsequence of good scales, defined as follows. A   scale $n$ is said to be \defini{good} if	\begin{equation}\label{eq:142}
			3 \alpha(n)\leq \log^5 (n)\cdot \left( \alpha(\lfloor n/2\log^4 n\rfloor)-1\right).
		\end{equation}
		It is  elementary to check that there exist infinitely many good scales. Indeed, if \eqref{eq:142} fails for every $n$ large enough, then we have $\limsup \frac{\alpha(n)}{n}>0$, which contradicts  the upper bound~\eqref{eq:96}. Now, if a scale $n$  is good, Equation~\eqref{eq:141} implies that 
		\begin{equation}
			\kappa_{p_0}\left(   3\alpha(n), n/\log^4 n \right)>\delta.        
		\end{equation}
		Since there are infinitely many good scales, this implies that  $\limsup \kappa_{p_0}\left(   \alpha(n), n/\log^4 n \right)\ge\delta$. In order to get the $\limsup$  with $n$ being a multiple of 12,  we can just notice that all the proof of Item~\ref{item:5} can be achieved exactly the same way with the additional assumption that $n$ is always a multiple of $12$.
	\end{proof}

	We are now ready to prove Equation~\eqref{eq:134}. Recall that once this equation is proved, Proposition~\ref{prop:1} follows.
	
	\begin{proof}[Proof of Equation~\eqref{eq:134}]
		We proceed by contradiction. Assume that we can take $p_1>p_c$ such that \eqref{eq:94} holds. Fix $p_0\in (p_c,p_1)$. Consider a sequence $(\alpha(n))$ as in Lemma~\ref{lem:11}. We can pick $\delta_0>0$ such that
		\begin{equation}\label{eq:157}
			\left\{
			\begin{array}[c]{l}
				\mathrm P_{p_0}[B_{\alpha(12n)}\lr{} \partial B_{6n}] \ge 1 - e^{-\log^3 (6n)}\\
				\kappa_{p_0}(3 \alpha(12n),12n/\log^4 (12 n))\ge \delta_0 
			\end{array}\right.
		\end{equation}
		for infinitely many $n$. Let $n$ be such that \eqref{eq:157} holds. Applying Proposition~\ref{prop:7} with $\eta=|B_{2n}| e^{-\log^3(6n)}$, $\delta=\min(\delta_0,p_1-p_0)$, $k=\alpha(12n)$, and $m= 12n/\log^4(12n)$, we obtain
		\begin{equation}
			\mathbf P[U_{p_0,p_1}(n,4n)]\geq 1-e^{-\log^2 n},
		\end{equation}
		provided $n$ is chosen large enough.
		Since the equation above holds for infinitely many $n$, we deduce that
		\begin{equation}
			\label{eq:300}
			\limsup_{n\to\infty}\mathbf P[U_{p_0,p_1}(n,4n)]=1.
		\end{equation}

		As in the proof of Lemma~\ref{lem:8}, we can use the ``sprinkled'' uniqueness estimate \eqref{eq:300} to deduce a lower bound on the two-point function. The origin $o$ is connected to another fixed vertex $x\in B_{n}$ by a $p_1$-open path in $B_{10n}$ as soon as both $o$ and $x$ are $p_0$-connected to $\partial B_{10n}$ and the uniqueness event $U_{p_0,p_1}(n,4n)$ occurs. Using the Harris--FKG Inequality and the union bound, for every $n\ge R$ we have
		\begin{align}
			\label{eq:145}
			\kappa_{p_1}(n,10n)&\ge \min_{x\in B_{n}}\mathbf P[\{o\lr[p_0]{}\partial B_{10n},x\lr[p_0]{}\partial B_{10n}\}\cap U_{p_0,p_1}(n,4n)]          \\
			&\ge \theta(p_0)^2 - \mathbf P[U_{p_0,p_1}(n,4n)^c].
		\end{align}
		It follows that $\limsup_{n\to\infty}\kappa_{p_1}(n,10n)>0$, which contradicts our hypothesis \eqref{eq:94}. \end{proof}
	
	\section{Proof of sharpness: coarse grains without rescaling}
	\label{sec:proof-theor-1}
	
	In this section, we prove Theorem~\ref{thm:1} and Theorem~\ref{thm:2}. Both proofs involve the use of Proposition~\ref{prop:1} as the building block of a Peierls-type argument.
	
	On the hypercubic lattice $\mathbb Z^d$, a natural way to conclude from Proposition~\ref{prop:1} is to use a standard renormalisation procedure: we look at blocks at a larger scale, and these blocks can be seen as vertices of a rescaled copy of $\mathbb Z^d$. This strategy relies on a self-similarity property of  $\mathbb Z^d$ that does not hold for general groups. In order to overcome this difficulty, we will keep the definition of large blocks (sometimes called ``coarse grains''), but we will replace the rescaling argument by the use of a $k$-independent\footnote{Recall that a site percolation process $X$ is \defini{$k$-independent} if for any two sets of vertices $V_1$ and $V_2$ such that $\forall (v_1,v_2)\in V_1\times V_2,~d(v_1,v_2)>k$, the restrictions of $X_{V_1}$ and $X_{V_2}$ are independent.} process with sufficiently high marginals. To each vertex $v$ of $G$, we attach a block consisting of vertices at distance $k$ around $v$ and  define a new process on the graph $G$ without rescaling it. Contrary to the renormalisation approach, a given block overlaps $\simeq k^d$ other blocks, inducing some dependencies on the new process. 
	
	In order to apply this strategy, we rely on the following lemma, which ensures that a finite-range site percolation with sufficiently high marginals does not have large closed cutsets. Recall that a  minimal cutset between $o$ and $\infty$ is a subset $\Pi\subset V$ such that the connected component of $o$ in $V\setminus \Pi$ is finite, and it is minimal for the inclusion among all sets satisfying this property. For every $n\ge1$, write $\mathcal P_{\ge n}$ for the set of all minimal cutsets $\Pi$ between $o$ and $\infty$ that satisfy $|\Pi|\ge n$.

	\begin{lemma}\label{lem:19}
		Let $R$ be so that the conclusion of  Lemma~\ref{lem:1} holds. For every $k\ge R$, there exists $c=c(k,G)>0$ such that the following holds.  Let $X$ be a $2k$-independent site percolation on $G$ with marginals satisfying  $\mathbb P[X(v)=1]\geq 1-\frac1{2e|B_{5k}|}$  for every $v\in V(G)$. Then, for every $n\ge 1$, we have
		\begin{equation}
			\mathbb P [\exists \Pi\in\mathscr{P}_{\ge n},~\Pi\text{ is closed in } X]\leq e^{-cn}.
		\end{equation}
	\end{lemma}

	\begin{proof}
		Fix $k\ge R$ and set $D=|B_{5k}|$. Let $\Pi\in\mathscr{P}_{\ge n}$. Since $R$ satisfies the conclusion of Lemma~\ref{lem:1}, the set $\Pi$ is $R$-connected. Besides, by Lemma~\ref{lem:2}, $\Pi$ intersects the ball of radius $R|\Pi|$. 
		
		Say that a set of vertices is $r$-separated if any two distinct vertices $u,v$ of it satisfy $d(u,v)> r$.
		Let $\Pi'$ be a $2k$-separated subset of $\Pi$ that is maximal for inclusion among all such subsets. By maximality, any ball of radius $2k$ centred at a vertex of $\Pi$ must contain a vertex of $\Pi'$. Therefore, the quantity $m:=|\Pi'|$ is at least $|\Pi|/D$. Furthermore, we claim that  $\Pi'$ satisfies the following properties:
		\begin{itemize}
			\item $\Pi'$ is $2k$-separated,
			\item $|\Pi'|=m$,
			\item    $\Pi'$ is $5k$-connected,    
			\item $\Pi'$ intersects $B_{RDm+2k}$. 
		\end{itemize}
		The first two items follow from the definitions above. The third and fourth items follow from the properties of  $\Pi$ together with the hypothesis $k\ge R$ and the observation that $\Pi'$ intersects  any ball of radius $2k$ centred at a vertex of $\Pi$.

		For every $m\ge 1$, write $\mathscr P'_{m}$ for the collection of all subsets of $V$ satisfying the four properties listed above. By definition, all the elements of  $\mathscr P'_{m}$ intersect $B_{RDm+2k}$ and they are  connected subsets of the  graph $G'=(V,E')$ with $E'=\{\{x,y\}~:~d(x,y)\leq 5k\}$. Since the degree of $G'$ is  $D$, standard bounds on the number of connected subsets of a graph (see e.g. \cite[Problem 45]{bollobas2006art}) give
		\begin{equation}
			\label{eq:146}
			|\mathscr P'_{m}|\le |B_{RDm+2k}|(e D)^{m}.
		\end{equation}
		From this bound, we get
		\begin{align}\label{eq:147}
			&	\mathbb P [\exists \Pi\in\mathscr{P}_{\geq n},~\Pi\text{ is closed in } X]\\
			&\qquad\leq \sum_{m\ge n/D}  \mathrm P[\exists \Pi'\in\mathscr{P}'_m,~\Pi'\text{ is closed in } X]\\
			&\qquad\le  \sum_{m\ge n/D}|\mathscr P'_m| \cdot \left(2eD\right)^{-m}\text{~~~~by our assumptions on $X$}\\
			&\qquad\le\sum_{m\ge n/D}|B_{RDm+2k}| \cdot 2^{-m} .
		\end{align}
		Since the volume of the ball $B_{Dn}$ grows at most polynomially in $n$, the proof follows from the estimate above. 
	\end{proof}
	
	\begin{proof}[Proof of Theorems~\ref{thm:1} and~\ref{thm:2} ]
		Let $p>p_c(G)$ and let $\omega$ be a Bernoulli bond percolation of parameter $p$ on $G$. By Proposition~\ref{prop:1} and by polynomial growth of $G$, we can choose $k$ such that
		\begin{equation}
			\mathrm P_p[B_{k/10}\lr{}\partial B_{k},U(k/5,k/2)]\geq 1-\frac1{2e|B_{5k}|}.
		\end{equation}
		From now on and until the end of the proof, we fix $k$ as above. For each $v\in V(G)$, we will denote by $U_v=U_v(k/5,k/2)$ the event $U(k/5,k/2)$ centred at $v$ instead of $o$. We define the site percolation  process $X$ on $G$ by setting
		\begin{equation}\label{eq:148}
			X(v)=
			\begin{cases}
				1&\text{ if both  $B_{k/10}(v)\lr{}\partial B_{k}(v)$ and $U_v$ occur,}\\
				0&\text{ otherwise.}
			\end{cases}
		\end{equation}
		Notice that $X$ is a $2k$-independent site percolation with marginals satisfying
		\begin{equation}
			\mathrm P_p[X(v)=1]\geq 1-\frac1{2e|B_{5k}|}.
		\end{equation} 
		Let us consider $n>|B_k|$. We denote by $C_o$ the cluster of $o$ in $\omega$. Assume that $|C_o|\geq n$. In this case, the key observation about the process $X$ is that having an infinite open path in $X$ touching $C_o$ entails the existence of a path between $o$ and $\infty$ in $\omega$. Let us explain this fact. Since $|C_o|\geq n>|B_k|$, for every $v\in C_o$, we have that $v\lr{}\partial B_{k}(v)$ in $C_o$. Then, let us assume that there is an infinite self-avoiding path $v_1,v_2,\ldots$ with $v_1\in C_o$ and such that $X(v_i)=1$ for all $i\geq 1$.  This implies that, for every $i\geq 1$, one can fix a cluster $C_i$ in $\omega\cap B_{k}(v_i)$ that connects $B_{k/10}(v_i)$ to $\partial B_{k}(v_i)$. Notice that for every $i\geq 1$, $C_i$ and $C_{i+1}$ both connect $B_{k/5}(v_i)$ to $B_{k/2}(v_i)$. Since the event $U_{v_i}(k/5,k/2)$ holds for every $i$, we have that all the $C_i$ are forced to be connected to each other in $\omega$. If we choose $C_1$ to be the connected component of $v_1$ in $C_o\cap B_k(v_1)$, we have that $o$ is connected to $\infty$ in $\omega$. Since $\{o\lr{}\partial B_n\}\subset \{|C_o|>n\}$, this observation implies that
		\begin{equation}\label{eq:149}
			\mathrm P_p[o\lr{}\partial B_n,o\nlr{}\infty] \leq\mathrm P_p[C_o\nlr[X]{}\infty, n\leq \mathrm{diam}(C_o)<\infty].
		\end{equation}
		If the event on the right-hand side of \eqref{eq:149} is satisfied, then, by Lemma~\ref{lem:3}, there exists an $X$-closed cutset $\Pi$ disconnecting $o$ and $\infty$ of diameter at least $n/2$. Since $\Pi$ is $R$-connected, this yields $|\Pi|\geq \frac{n}{2R}$. One can thus use Lemma~\ref{lem:19} to upper-bound the right-hand side of \eqref{eq:149} by  $\exp(-\frac{cn}{2R})$, for some constant $c>0$. This concludes the proof of Theorem~\ref{thm:1}.
		
		Regarding the proof of Theorem~\ref{thm:2}, we have, similarly to before, that for $n>|B_k|$,
		\begin{equation}\label{eq:150}
			\mathrm P_p[n\leq|C_o|<\infty]=\mathrm P_p[C_o\nlr[X]{}\infty, n\leq \mathrm|C_o|<\infty].
		\end{equation}
		In this case, if the event on the right-hand side of \eqref{eq:150} is satisfied, then Lemma~\ref{lem:4} 
		yields an $X$-closed cutset disconnecting $o$ and $\infty$ of size larger than $cn^{\frac{d-1}{d}}$. Hence, using Lemma~\ref{lem:19}, the right-hand side of \eqref{eq:150} is upper bounded by  $\exp(-c'n^\frac{d-1}{d})$ for some positive constant $c'$. This concludes the proof of Theorem~\ref{thm:2}.
	\end{proof}
	
	\newcommand{\etalchar}[1]{$^{#1}$}

	
	
\end{document}